\documentclass[11pt]{amsart}
\usepackage{amscd,amsfonts,amssymb,amsmath,amsthm,latexsym}
\usepackage[extra]{tipa} 
\usepackage{url}          
\usepackage{mathrsfs}
\usepackage{geometry}
\usepackage{float}
\usepackage[dvips]{graphicx}
\usepackage{graphics}
\usepackage[all]{xy}
\usepackage{color}

\numberwithin{equation}{section}

\theoremstyle{plain}
    \newtheorem{thm}{Theorem}[section]
    \newtheorem{lemma}[thm]{Lemma}
    \newtheorem{coro}[thm]{Corollary}
    \newtheorem{prop}[thm]{Proposition}

\theoremstyle{definition}
    \newtheorem{defi}[thm]{Definition}

    \newtheorem{remark}[thm]{Remark}
\theoremstyle{remark}

\newcommand{\df}{\colon}
\newcommand{\ra}{\rightarrow}
\newcommand{\abs}[1]{\lvert #1\rvert}
\newcommand{\Smin}{S_{\rm min}}

\newcommand{\add}{\operatorname{add}}
\newcommand{\Coker}{\operatorname{Coker}}
\newcommand{\Hom}{\operatorname{Hom}}
\newcommand{\Id}{\operatorname{Id}}
\newcommand{\Ker}{\operatorname{Ker}}
\newcommand{\Image}{\operatorname{Im}}
\newcommand{\md}{\operatorname{mod}}
\newcommand{\lmd}[1]{#1\text{-}\operatorname{mod}}
\newcommand{\rad}{\operatorname{rad}}

\newcommand{\depth}{\operatorname{depth}}
\newcommand{\oin}{{\operatorname{in}}}
\newcommand{\length}{\operatorname{length}}
\newcommand{\longest}{\operatorname{long}}
\newcommand{\oout}{{\operatorname{out}}}

\newcommand{\short}{\operatorname{short}}
\newcommand{\val}{\operatorname{val}}

\newcommand{\stmd}{\underline{\operatorname{mod}}}
\newcommand{\dimv}{\underline{\dim}}

\newcommand{\C}{\mathbb{C}}  
\newcommand{\D}{\mathbb{D}}
\newcommand{\NN}{\mathbb{N}} 

\newcommand{\ZZ}{\mathbb{Z}} 

\newcommand{\marked}{\mathbb{M}}
\newcommand{\punct}{\mathbb{P}}

\newcommand{\alp}{\alpha}
\newcommand{\bet}{\beta}
\newcommand{\gam}{\gamma}
\newcommand{\del}{\delta}
\newcommand{\lam}{\lambda}
\newcommand{\vph}{\varphi}
\newcommand{\eps}{\epsilon}

\newcommand{\LL}{\Lambda}
\newcommand{\surfnoM}{\Sigma}

\newcommand{\maxid}{\mathfrak{m}}

\newcommand{\bd}{\mathbf{d}}

\newcommand{\bv}{{\mathbf v}}
\newcommand{\bx}{{\mathbf x}}
\newcommand{\by}{{\mathbf y}}

\newcommand{\cA}{\mathcal{A}}

\newcommand{\cD}{\mathcal{D}}

\newcommand{\cN}{\mathcal{N}}
\newcommand{\cP}{\mathcal{P}}
\newcommand{\cW}{\mathcal{W}}

\newcommand{\m}{{\mathfrak m}}
\newcommand{\idealM}{\mathfrak{m}} 

\newcommand{\I}{{\rm I}}
\newcommand{\II}{{\rm II}}
\newcommand{\III}{{\rm III}}

\newcommand{\arc}{i}

\newcommand{\surf}{(\Sigma,\mathbb{M})}
\newcommand{\unpunctsurf}{(\Sigma,\mathbb{M}_0)}

\newcommand{\qtau}{Q(\tau)}
\newcommand{\unredqtau}{\widehat{Q}(\tau)}

\newcommand{\unredstaux}{\widehat{S}(\tau,\mathbf{x})}

\newcommand{\staux}{S(\tau,\mathbf{x})}

\newcommand{\unredqstaux}{(\widehat{Q}(\tau),\widehat{S}(\tau,\mathbf{x}))}
\newcommand{\qstaux}{(Q(\tau),S(\tau,\mathbf{x}))}
\newcommand{\qssigmax}{(Q(\sigma),S(\sigma,\mathbf{x}))}

\newcommand{\rA}[1]{\ar@<-.08ex>@{-_{>}}[#1]\ar@<.08ex>@{-^{>}}[#1]\ar[#1]}
\newcommand{\rV}[1]{\mathbf{#1}}

\newcommand{\ka}{\mathbb{C}}
\newcommand{\ebrace}[1]{\langle #1\rangle}
\newcommand{\CQ}[1]{\ka\langle\hspace{-0.05cm}\langle #1\rangle\hspace{-0.05cm}\rangle}
\newcommand{\kQ}[1]{\ka\langle #1\rangle} 
\newcommand{\falg}[1]{\ka\langle #1\rangle} 

\newcommand{\Mout}[2]{#1(#2_{\oout})}
\newcommand{\Min}[2]{#1(#2_{\oin})}
\newcommand{\Malp}[2]{#1(\alp_{#2})}
\newcommand{\Mbet}[2]{#1(\bet_{#2})}
\newcommand{\Mgam}[2]{#1(\gam_{#2})}

\newcommand{\blue}[1]{\textcolor{blue}{#1}}

\begin{document}

\date{03.01.2016}

\title{The representation type of Jacobian algebras}

\author{Christof Gei{\ss}}
\address{Christof Gei{\ss}\newline
Instituto de Matem\'aticas\newline
Universidad Nacional Aut{\'o}noma de M{\'e}xico\newline
Ciudad Universitaria\newline
04510 M{\'e}xico D.F.\newline
M{\'e}xico}
\email{christof@math.unam.mx}

\author{Daniel Labardini-Fragoso}
\address{Daniel Labardini-Fragoso\newline
Instituto de Matem\'aticas\newline
Universidad Nacional Aut{\'o}noma de M{\'e}xico\newline
Ciudad Universitaria\newline
04510 M{\'e}xico D.F.\newline
M{\'e}xico}
\email{labardini@matem.unam.mx}

\author{Jan Schr\"oer}
\address{Jan Schr\"oer\newline
Mathematisches Institut\newline
Universit\"at Bonn\newline
Endenicher Allee 60\newline
53115 Bonn\newline
Germany}
\email{schroer@math.uni-bonn.de}


\begin{abstract}
We show that the representation type of the Jacobian algebra
$\cP(Q,S)$ associated to a 2-acyclic quiver $Q$ with
non-degenerate potential $S$ is invariant under QP-mutations.
We prove that, apart from very few exceptions,
$\cP(Q,S)$ is of tame representation type
if and only if $Q$ is of finite mutation type.
We also show that most quivers $Q$ of finite
mutation type admit
only one non-degenerate potential
up to weak right equivalence.
In this case, the isomorphism class of
$\cP(Q,S)$ depends only on $Q$ and not on $S$.
\end{abstract}

\maketitle

\setcounter{tocdepth}{1}

\tableofcontents

\parskip2mm


\section{Introduction}\label{sec1}


\subsection{Cluster algebras and Jacobian algebras}
Cluster algebras were invented by Fomin and Zelevinsky \cite{FZ1}
in an attempt to obtain a combinatorial approach to
the dual
of Lusztig's canonical basis of quantum groups.
Another
motivation for cluster algebras was the concept of total positivity, which dates back about 80 years, and was generalized and connected to Lie theory by Lusztig in 1994, see \cite{Lu1,Lu2}.
By now numerous connections between cluster algebras
and other branches of mathematics have been discovered, e.g. representation theory of quivers and algebras and
Donaldson-Thomas invariants of 3-Calabi-Yau categories.

By definition, the \emph{cluster algebra} $\cA_B$ associated to a
skew-symmetrizable integer matrix $B$
is the subalgebra of a field of rational functions generated by an inductively constructed set of so-called cluster
variables.
(Starting with $B$ and a set of initial cluster variables the other cluster
variables are obtained via iterated \emph{seed mutations}.)
In this paper, we assume that the matrix $B$ is skew-symmetric
(for arbitrary skew-symmetrizable matrices $B$ the theory of cluster algebras is
much less developed).
The set of skew-symmetric integer matrices corresponds
bijectively to the set of 2-acyclic quivers.
We denote
$\cA_Q := \cA_B$ if the quiver $Q$ corresponds to the skew-symmetric matrix $B$.
The set of cluster algebras $\cA_Q$ can be divided naturally into two
classes.
Namely, the quiver $Q$ is either of finite or of infinite mutation
type.

One of the main links between the theory of cluster algebras
and the representation theory of algebras is given by the
work of Derksen, Weyman and Zelevinsky \cite{DWZ1,DWZ2} on quivers with potentials $(Q,S)$ and the representations of the associated
Jacobian algebras
$\cP(Q,S)$.
Here $S$ is a non-degenerate potential on a $2$-acyclic quiver $Q$, and $\cP(Q,S)$
is by definition the completed path algebra of $Q$ modulo
the closure of the ideal generated by the cyclic derivatives
of $S$.
Our first main result (Theorem~\ref{thmtamewild})
shows that the mutation type of $Q$ is closely
related to the representation type of $\cP(Q,S)$.

Let $\Gamma := \Gamma(Q,S)$ be the Ginzburg dg-algebra
associated to $(Q,S)$, see \cite{Gi} and also the survey article
\cite{Ke}.
The derived category $\cD_{\rm f.d.}(\Gamma)$
of dg-modules over $\Gamma$ with finite-dimensional total homology
is a 3-Calabi-Yau category.
It has a natural $t$-structure whose heart is canonically equivalent
to the category $\md(\cP(Q,S))$ of finite-dimensional
modules over the Jacobian algebra $\cP(Q,S)$.
The triangle quotient
$\cD_{\rm per}(\Gamma)/\cD_{\rm f.d.}(\Gamma)$
is the \emph{cluster category} associated to $(Q,S)$ defined by
Amiot \cite{A1,A2}.
Here $\cD_{\rm per}(\Gamma)$ is the perfect derived category of $\Gamma$.
Our second main result (Theorem~\ref{thmunique})
implies that
for most quivers $Q$ of finite mutation type, there is essentially
just one way to associate a 3-Calabi-Yau category and
a cluster category to $Q$. (If one drops the assumption of the
non-degeneracy of $S$, then these categories no longer reflect
the mutations of seeds of the cluster algebra $\cA_Q$ associated to
$Q$.)

Jacobian algebras
are not only important for the study of cluster algebras, they also
appear in other contexts such as the study of
BPS spectra of certain classes of quantum field theories.
Quivers with potentials arising from triangulations of marked
oriented surfaces often are of particular importance,
see for example \cite{ACCERV,BS,C,CV,Sm}.

\subsection{Representation type of Jacobian algebras}
For definitions related to quivers with potentials and Jacobian
algebras we refer to Section~\ref{secprelim}.
By Drozd's \cite{D} celebrated Tame-Wild Theorem, any finite-dimensional
algebra over an algebraically closed field is either of tame
or of wild representation type.
This is called the \emph{tame-wild dichotomy} or
\emph{representation type dichotomy} for
finite-dimensional algebras.
One of our aims is to determine the representation type of Jacobian
algebras.
The following statement may not come as a surprise,
but the proof is not so straightforward.

\begin{thm}\label{thmmutationinv}
For any non-degenerate potential $S$
on a $2$-acyclic quiver $Q$, the representation type of the Jacobian algebra
$\cP(Q,S)$ is preserved under QP-mutation.
\end{thm}

The quiver $Q$ is of \emph{finite cluster type}
if there are only finitely many cluster variables in the cluster
algebra $\cA_Q$ associated to $Q$.
We say that $Q$ is of \emph{finite mutation type} if there are only
finitely many quivers mutation equivalent to $Q$.
Otherwise, $Q$ is of \emph{infinite mutation type}.

Fomin and Zelevinsky \cite{FZ2} proved that a quiver $Q$ is of finite
cluster type if and only if it is mutation equivalent to a Dynkin
quiver.
In particular, $Q$ is of finite mutation type in this case.
Using this, it is not difficult to show that for any non-degenerate
potential $S$ on $Q$ the Jacobian algebra
$\cP(Q,S)$ is representation-finite if and only if
$Q$ is of finite cluster type.

Based on work by Fomin, Shapiro and Thurston \cite{FST},
the quivers of finite mutation type were classified
by Felikson, Shapiro and Tumarkin \cite{FeST}.
We recall their classification in Section~\ref{sec5}.

We call a quiver $Q$
\emph{Jacobi-tame} or \emph{Jacobi-wild}
if for all non-degenerate potentials $S$ the Jacobian algebra
$\cP(Q,S)$ is of tame or wild representation type, respectively.
Otherwise, we call $Q$ \emph{Jacobi-irregular}.

The following theorem is our first main result.
(The quivers $T_1$, $T_2$, $X_6$, $X_7$ and $K_m$ appearing in
the statement can be found in Sections~\ref{secclass}
and \ref{secnodouble}.
They are of finite mutation type.
Note that the mutation equivalence classes of the quivers
$T_1$, $T_2$ and $K_m$ contain only one quiver up to isomorphism.)

\begin{thm}\label{thmtamewild}
Let $Q$ be a $2$-acyclic quiver.
If $Q$ is not mutation equivalent to one of the quivers
$T_1$, $T_2$, $X_6$, $X_7$ or $K_m$ with $m \ge 3$,
then the following hold:
\begin{itemize}

\item[(i)]
$Q$ is Jacobi-tame if and only if
$Q$ is of finite mutation type.

\item[(ii)]
$Q$ is Jacobi-wild if and only if
$Q$ is of infinite mutation type.

\end{itemize}
For the exceptional cases the following hold:
\begin{itemize}

\item[(iii)]
If $Q$  is mutation equivalent to one of the quivers
$X_6$, $X_7$ or $K_m$ with $m \ge 3$, then $Q$ is
Jacobi-wild.

\item[(iv)]
If $Q$ is mutation equivalent to one of the quivers
$T_1$ or $T_2$, then $Q$ is Jacobi-irregular.

\end{itemize}
\end{thm}

The following result is
a direct consequence
of Theorems~\ref{thmmutationinv} and
\ref{thmtamewild}.

\begin{coro}\label{maincoro1}
Let $S$ be a non-degenerate potential on $Q$, and
assume that $Q$ is not equal to $T_1$ or $T_2$.
Then
the representation type of $\cP(Q,S)$ depends only on $Q$, i.e. it is independent from $S$,
and is preserved under QP-mutation.
\end{coro}

Using Theorem~\ref{thmmutationinv}, Drozd's Tame-Wild Theorem
and Felikson, Shapiro and Tumarkin's \cite{FeST} classification of quivers of
finite mutation type,
the proof of Theorem~\ref{thmtamewild} is reduced to
the problem of showing that the Jacobian algebras $\cP(Q,S)$
arising from triangulations of marked surfaces
and also a small
number of exceptional Jacobian algebras
are tame.

\subsection{Classification of potentials}
The existence of non-degenerate potentials (over uncountable fields) was proved in \cite{DWZ1}.
However, it can be rather difficult to construct a
non-degenerate potential explicitly.
Potentials are usually studied up to right equivalence.
One of several reasons is that two Jacobian algebras $\cP(Q,S_1)$ and $\cP(Q,S_2)$ are
isomorphic provided $S_1$ and $S_2$ are right equivalent. In this paper we introduce a slightly more general notion, namely that of \emph{weak right equivalence}, that still implies isomorphism of Jacobian algebras, see Section~\ref{subsec:right-equivalences} for the definition. It turns out that QP-mutations of weakly right equivalent QPs are again weakly right equivalent, and that QPs weakly right equivalent to non-degenerate ones are non-degenerate as well.

If $Q$ is a quiver
which is mutation equivalent to some
acyclic quiver, then there is only one non-degenerate potential
on $Q$, up to right equivalence.
(This follows from Derksen, Weyman and Zelevinsky's result
that QP-mutation induces a bijection between right equivalence
classes of potentials, and by the trivial fact that an acyclic quiver
has only one potential, namely the zero-potential $S = 0$.)
In general it seems to be hopeless to classify
all non-degenerate potentials on an arbitrary $2$-acyclic quiver.
However, the following theorem, which is our second main result, says that most quivers of
finite mutation type admit only one non-degenerate
potentials up to right equivalence or weak right equivalence.
Fomin, Shapiro and Thurston \cite{FST} associate to each triangulation
$\tau$
of a marked surface $(\Sigma,\marked)$ a quiver $Q(\tau)$
of finite mutation type, see Section~\ref{sectriang} for more details.
By Felikson, Shapiro and Tumarkin \cite{FeST}, most quivers of finite mutation
type arise in this way.

\begin{thm}\label{thmunique}
Assume that
$Q = Q(\tau)$
for a triangulation $\tau$ of a marked surface $(\Sigma,\marked)$.
Then the following hold:
\begin{itemize}

\item[(i)]
If the boundary of $\Sigma$ is non-empty,
and $(\Sigma,\marked)$ is not a torus
with $\abs{\marked} = 1$,
then there exists only one non-degenerate potential $S$ on $Q$ up to right equivalence.

\item[(ii)]
If the boundary of $\Sigma$ is empty, and
\[
|\marked| \geq
\begin{cases}
5 & \text{if $\Sigma$ is a sphere},\\
3 & \text{otherwise},
\end{cases}
\]
then there exists only one non-degenerate potential $S$ on $Q$ up to weak right equivalence.

\end{itemize}
\end{thm}

There is a short list of mutation classes of quivers of finite mutation type that do
not arise from triangulations of surfaces.
We show that for most of these exceptional quivers there exists again only one non-degenerate potential up to right equivalence.

Assume that
$Q = Q(\tau)$
for a triangulation $\tau$ of a marked surface $(\Sigma,\marked)$
with empty boundary, and $(\Sigma,\marked)$
is not a sphere with $\abs{\marked} = 4$.
If $\abs{\marked} = 1$,
then we show that there exist at least two non-degenerate
potentials on $Q$ up to weak right equivalence.
For the case $\abs{\marked} = 2$,
it remains an open problem to classify the non-degenerate potentials.

For a triangulation $\tau$ of a marked surface,
Labardini \cite{Lthesis,Labardini1} defined and studied a potential
$S(\tau)$
on $Q(\tau)$.
(The definition of $S(\tau)$ is recalled in Section~\ref{labardinipotential}.)
Our proof of Theorem~\ref{thmunique} does not rely on Labardini's
previous results.
However, using his results we get
that the potential $S$ appearing in Theorem~\ref{thmunique}
is equal to $S(\tau)$ up to weak right equivalence.

\subsection{Which tame algebras do occur?}
In many cases, tame Jacobian algebras are deformations of skewed-gentle 
algebras in the sense of~\cite{GP}.
But it remains a challenge to understand their representation theory
as most representation theoretic knowledge gets lost under deformations.

Let $Q$ be a quiver of finite mutation type, and let $S$ be
a non-degenerate potential on $Q$.
As mentioned above, in most cases we have $Q = Q(\tau)$ for
some triangulation $\tau$ of some marked surface $(\Sigma,\marked)$, and
$\cP(Q,S)$ is
isomorphic to $\cP(Q(\tau),S(\tau))$, where $S(\tau)$ is the
Labardini potential on $Q(\tau)$.
\begin{itemize}

\item[(1)]
If $(\Sigma,\marked)$ is unpunctured (i.e. all marked points are in the 
boundary of $\Sigma$), then the Jacobian algebra $\cP(Q(\tau),S(\tau))$ 
is a gentle algebra in the sense of Assem and Skowro\'nski~\cite{AS}
for all triangulations $\tau$ of $(\Sigma,\marked)$, see~\cite{ABCP}.

\item[(2)]
If the boundary of $\Sigma$ is non-empty, and $(\Sigma,\marked)$ is
not a monogon, and not a torus with $|\marked| = 1$, then
there exists a triangulation $\sigma$ of $(\Sigma,\marked)$ such
that $\cP(Q(\sigma),S(\sigma))$ is a skewed-gentle algebra,
and for all $\tau$ the algebras $\cP(Q(\tau),S(\tau))$ and 
$\cP(Q(\sigma),S(\sigma))$ are related by a composition of finitely many nearly
Morita equivalences, see Section~\ref{secnearmorita} for the definition of a 
nearly Morita equivalence.
Skewed-gentle algebras belong to the class of clannish algebras. Therefore, 
one can classify the indecomposable $\cP(Q(\sigma),S(\sigma))$-modules
combinatorially~\cite{CB2}.

\item[(3)]
If the boundary of $\Sigma$ is empty, and $(\Sigma,\marked)$ not a
sphere with $\abs{\marked} = 4$, then the algebra
$\cP(Q(\tau),S(\tau))$ is symmetric for all $\tau$, see \cite{L2}.
These algebras show a similar behaviour as algebras of quaternion
type, compare \cite{E,V}.

\item[(4)]
If $(\Sigma,\marked)$ is a sphere with $\abs{\marked} = 4$, or if
$Q$ is mutation equivalent to one of the exceptional quivers $E_m^{(1,1)}$ (see Section~\ref{secclass}), then for all $\tau$ the
algebra
$\cP(Q(\tau),S(\tau))$
is closely related to the class of tubular algebras, see \cite{GeGo}
for further details.

\end{itemize}

\subsection{}
This article is organized as follows.
Some definitions and basic results on
quivers with potentials and Jacobian algebras are recalled in
Section~\ref{secprelim}.
In Section~\ref{secreptype} we show that the representation type
of a Jacobian algebra is invariant under QP-mutations.
Section~\ref{sectriang} consists of a reminder on the construction
of quivers associated to triangulations of marked surfaces, and
on the relation between flips of a triangulation and mutations
of the corresponding quiver.
This is taken from Fomin, Shapiro and Thurston's \cite{FST} work
on cluster algebras arising from surfaces.
In Section~\ref{sec5} we recall the classification of mutation finite
quivers due to Felikson, Shapiro and Tumarkin \cite{FeST}.
In Section~\ref{sec7} we give a tameness criterion for Jacobian
algebras based on a general result by Gei{\ss} \cite{G}, who showed that deformations of tame algebras are tame.
Apart from some exceptional cases, the representation type of
Jacobian algebras
defined by quivers with non-degenerate potential
is determined in Section~\ref{sec11}.
Here we rely heavily on the tameness criterion from Section~\ref{sec7},
and we construct triangulations of marked surfaces satisfying
some favourable properties.
In
Section~\ref{sec10} we classify all non-degenerate potentials
for most quivers of finite mutation type, and we provide a
uniqueness criterion for non-degenerate potentials.
Finally, Section~\ref{sec11b} deals with the exceptional cases.


\section{Preliminaries}
\label{secprelim}


\subsection{}
Let
$\ka$ be the field of complex numbers.
For a
$\ka$-algebra
$A$ let $\md(A)$ be the category of finite-dimensional
left $A$-modules.
If not mentioned otherwise, by an $A$-\emph{module}
we mean  a module in $\md(A)$.

\subsection{Basic algebras}
A \emph{quiver} is a quadruple $Q = (Q_0,Q_1,s,t)$ with $Q_0$
and $Q_1$ finite sets and two maps $s,t\df Q_1 \to Q_0$.
The elements in $Q_0$ and $Q_1$ are called \emph{vertices}
and \emph{arrows} of $Q$, respectively.
We say that an arrow $\alp$ in $Q_1$ \emph{starts} in $s(\alp)$ and
\emph{ends} in $t(\alp)$.

A \emph{path} of length $m \ge 1$ in $Q$ is a tuple
$(\alp_1,\ldots,\alp_m)$ of arrows of $Q$ such that
$s(\alp_i) = t(\alp_{i+1})$ for all $1 \le i \le m-1$.
We also just write $\alp_1 \cdots \alp_m$ instead of
$(\alp_1,\ldots,\alp_m)$, and we set
$s(\alp_1 \cdots \alp_m) := s(\alp_m)$ and
$t(\alp_1 \cdots \alp_m) := t(\alp_1)$.
Additionally, for each vertex $i$ of $Q$ there is a path $e_i$
of length $0$ with $s(e_i) = t(e_i) = i$.

The path algebra of $Q$ is denote by $\kQ{Q}$.
The paths in $Q$ form a $\ka$-basis of $\kQ{Q}$.
Let $\CQ{Q}$ be the completed path algebra of $Q$.
As a $\C$-vector space we have
\[
\CQ{Q} = \prod_{m \ge 0} \ka Q_m
\]
where $\ka Q_m$ is a $\ka$-vector space with a basis
labeled by the paths of length $m$ in $Q$.
The multiplication of $\kQ{Q}$ and $\CQ{Q}$
is induced by the concatenation of paths.
Both algebras are naturally graded by the length of paths.

Let
\[
\m := \prod_{m \ge 1} \ka Q_m
\]
be the \emph{arrow ideal} of $\CQ{Q}$.
For any subset $U \subseteq \CQ{Q}$ let
\[
\overline{U} := \bigcap_{p \ge 0} (U+\m^p)
\]
be the $\m$-\emph{adic closure} of $U$.

An algebra $\LL$ is called \emph{basic} provided it is isomorphic to an algebra
of the form $\CQ{Q}/I$, where $I$ is an ideal with
$I \subseteq \m^2$.
For a basic algebra $\LL = \CQ{Q}/I$ define
$\overline{\LL} := \CQ{Q}/\overline{I}$, and
for $p \ge 2$ let $\LL_p := \CQ{Q}/(I+\m^p)$
be the $p$-\emph{truncation} of $\LL$.
The algebras $\LL_p$ are obviously finite-dimensional,
and we have
$\overline{\LL} = \underleftarrow{\lim}(\LL_p)$, i.e. $\overline{\LL}$ is the inverse
limit of the algebras $\LL_p$.
It is straightforward to show that
\[
\md(\LL) = \md(\overline{\LL}) = \bigcup_{p \ge 2} \md(\LL_p).
\]

\subsection{The representation type of a basic algebra}\label{secreptypeintro}
Let
\[
\LL = \CQ{Q}/I
\]
be a basic algebra.
Then $\LL$ is called \emph{representation-finite} if there are only
finitely many indecomposable $\LL$-modules in $\md(\LL)$
up to isomorphism.
We say that $\LL$ is \emph{tame} provided for each
dimension vector $\bd$ there are finitely many
$\LL$-$\ka[X]$-bimodules $M_1,\ldots,M_t$, which are
free of finite rank as $\ka[X]$-modules, such that
all indecomposable $\LL$-modules $X$ with $\dimv(X) = \bd$
are isomorphic to a module of the form
\[
M_i \otimes_{\ka [X]} \ka [X]/(X-\lambda)
\]
for some $1 \le i \le t$ and some $\lambda \in \ka $.
Clearly,
every representation-finite algebra is tame.
The algebra $\LL$ is called \emph{wild} if there is a
$\LL\text{-}\falg{X,Y}$-bimodule $M$, which is free of finite rank as
a $\falg{X,Y}$-module, such that
the exact functor
\[
M \otimes_{\falg{X,Y}} - \colon \md(\falg{X,Y}) \to \md(\LL)
\]
maps indecomposable to indecomposable and non-isomorphic to non-isomorphic
modules.
An exact functor with these properties is called a
\emph{representation embedding}.
It follows that
for every finitely generated $\ka$-algebra $B$ there is a representation
embedding
$\md(B) \to \md(\LL)$.
The following theorem is due to Drozd \cite{D}.
A more detailed proof can be found in \cite{CB1}. Since
\[
\md(\LL) = \bigcup_{p \ge 2} \md(\LL_p),
\]
Drozd's Theorem also holds for arbitrary basic algebras, and not just for finite-dimensional ones.

\begin{thm}[Drozd]\label{thmdrozd}
Every finite-dimensional algebra is either tame or wild, but not both.
\end{thm}

We say that $\LL$ is of \emph{finite}, \emph{tame} or
\emph{wild} \emph{representation type} provided $\LL$
is representation-finite, tame or wild, respectively.

\subsection{Quivers with potentials and Jacobian algebras}
\label{ssec:QPJac}
Let $Q$ be a quiver.
A path $\alp_1 \cdots \alp_m$ of length $m \ge 1$ in $Q$
is a \emph{cycle}
or more precisely an $m$-\emph{cycle} if $s(\alp_m) = t(\alp_1)$.
Quivers without cycles are called \emph{acyclic}.
A quiver is $2$-\emph{acyclic} if it does not contain
any 2-cycles.
Note that there is a bijection between the set of 2-acyclic quivers and the
set of skew-symmetric integer matrices.
More precisely,
let $Q=(Q_0,Q_1,s,t)$ be a quiver with $Q_0 = \{ 1,\ldots, n\}$.
Define a matrix $B_Q = (b_{ij}) \in M_n(\ZZ)$ by
\[
b_{ij}=\abs{\{ \alp \in Q_1 \mid s(\alp)=j, t(\alp)=i \}} -
\abs{\{ \alp \in Q_1 \mid s(\alp)=i, t(\alp)=j \}}.
\]
Then $Q \mapsto B_Q$ gives a bijection
\[
\{ \text{2-acyclic quivers with vertices $1,\ldots,n$} \}
\to \{ \text{skew-symmetric matrices in $M_n(\ZZ)$} \}.
\]

An element $S \in \CQ{Q}$ is
a \emph{potential} on $Q$ if $S$ is a (possibly infinite)
linear combination of cycles in $Q$.
The pair $(Q,S)$ is called a \emph{quiver with potential} or
just a \emph{QP}.
A QP $(Q,S)$ is 2-\emph{acyclic} if $Q$ is 2-acyclic.

We recall Derksen, Weyman and Zelevinsky's \cite{DWZ1}
definition of the \emph{Jacobian algebra} $\cP(Q,S)$ associated to
a quiver with potential $(Q,S)$.
For a cycle $\alp_1 \cdots \alp_m$ in $Q$ and an arrow $\alp\in Q_1$ define
\[
\partial_\alp(\alp_1 \cdots \alp_m) := \sum_{p\df \alp_p = \alp}
\alp_{p+1} \cdots \alp_m\alp_1 \cdots \alp_{p-1}.
\]
Then we extend by linearity and continuity to define the 
\emph{cyclic derivative}
$\partial_\alp(S)$ of a potential $S$ on $Q$.
Let
$\partial(S) :=\{ \partial_\alp(S) \mid \alp \in Q_1\}$.
(Note that our usage of $\partial(S)$
differs from the one in \cite{DWZ1}.)

Define
\[
\cP(Q,S) := \CQ{Q}/J(S)
\]
where $J(S)$ is the $\m$-adic closure of the ideal $I(S)$
generated by the set $\partial(S)$
of cyclic derivatives of $S$.

\subsection{Right equivalences}\label{subsec:right-equivalences}

For a quiver $Q$ let $R := R_Q$ be the semisimple
subalgebra of $\CQ{Q}$ which is generated by the idempotents
$e_1,\ldots,e_n$ of $\CQ{Q}$.
Using the canonical inclusion $R \to \CQ{Q}$, we can see
$\CQ{Q}$ as an $R$-algebra.

\begin{prop}[{\cite[Proposition 2.4]{DWZ1}}]
Let $Q$ and $Q'$ be quivers on the
same vertex set, and let their respective arrow spans be $A$ and $A'$.
Every pair $(\vph^{(1)},\vph^{(2)})$ of $R$-bimodule homomorphisms
$\vph^{(1)}\colon A\ra A'$, $\vph^{(2)}\colon A\ra\idealM(Q')^2$,
extends uniquely to a continuous $R$-algebra homomorphism
$\vph\colon \CQ{Q}\ra \CQ{Q'}$ such that
$\vph|_A = \vph^{(1)} + \vph^{(2)}$.
Furthermore, $\vph$ is an $R$-algebra isomorphism if and only if
$\vph^{(1)}$ is an $R$-bimodule isomorphism.
\end{prop}

Let us recall some definitions from \cite[Section 2]{DWZ1}.
An $R$-algebra automorphism
$\vph\colon \CQ{Q}\ra \CQ{Q}$ is
\begin{itemize}

\item
\emph{unitriangular} if $\vph^{(1)}$ is the identity of $A$;

\item
\emph{of depth $\ell<\infty$} if $\vph^{(2)}(A)\subseteq\maxid^{\ell+1}$,
but $\vph^{(2)}(A)\nsubseteq\maxid^{\ell+2}$;

\item
\emph{of infinite depth} if $\vph^{(2)}(A)=0$.

\end{itemize}
The depth of $\vph$ will be denoted by $\depth(\vph)$.

Thus for example, the identity is the only unitriangular
$R$-algebra automorphism of $\CQ{Q}$ of infinite depth.
Also, the composition of
unitriangular automorphisms is unitriangular.

Let $V := V_Q$ be the
$\m$-adic closure of
the $\ka$-vector subspace of $\CQ{Q}$ generated by all
elements of the form $\alp_1\alp_2\ldots\alp_d-\alp_2\ldots\alp_d\alp_1$
with $\alp_1\ldots \alp_d$ a
cycle in $Q$.
Then $V$ is also a $\ka$-vector subspace of
$\CQ{Q}$, and by definition two potentials $S$ and $S'$ are
\emph{cyclically equivalent} if their difference belongs to $V$.
In this case, we write $S \sim_{\rm cyc} S'$.

Let $(Q,S)$ and $(Q',S')$ be QPs such that $Q$ and $Q'$
have the same set of vertices.
Thus we have $R := R_Q = R_{Q'}$.
An $R$-algebra isomorphism $\psi\colon \CQ{Q} \to \CQ{Q'}$
is called a \emph{right equivalence} if the potentials
$\psi(S)$ and $S'$ are cyclically equivalent.
In this case, we say that $(Q,S)$ and $(Q',S')$ are \emph{right equivalent},
and we write $\psi\colon (Q,S) \to (Q',S')$.

Two potentials $S$ and $S'$ are called
\emph{weakly right equivalent}
if the potentials $S$ and $tS'$ are right
equivalent for some $t \in \C^*$.
In this case, the Jacobian algebras $\cP(Q,S)$
and $\cP(Q,S')$ are isomorphic.

Two cycles $v$ and $w$ in $Q$ are \emph{rotationally equivalent}
if $v = w_2w_1$ for some paths $w_1$ and $w_2$ in $Q$ such that
$w = w_1w_2$.
In this case, we write $v \sim_{\rm rot} w$.
Note that $v \sim_{\rm rot} w$ if and only if $v \sim_{\rm cyc} w$.

A cycle $v$ \emph{appears} in a potential $S = \sum_w \lambda_w w$
if there is some $w$ with $\lambda_w \not= 0$ and
$w \sim_{\rm rot} v$.
In this case, we also say that $S$ \emph{contains} $v$.
Two potentials $S$ and $S'$
on a quiver $Q$ are \emph{rotationally disjoint} if there is no cycle
$v$ in $Q$ appearing in both potentials $S$ and $S'$.

Given a non-zero element
$u\in \CQ{Q}$, we denote by $\short(u)$ the unique
integer such that $u\in\maxid^{\short(u)}$ but $u\notin\maxid^{\short(u)+1}$.
We
also set $\short(0)=\infty$.
If $S$ is a \emph{finite potential}, i.e. if $S$ is a linear combination of
finitely many cycles, then
we denote by $\longest(S)$ the length of the
longest cycle appearing in $S$.

\subsection{Mutations of quivers}
Let $Q = (Q_0,Q_1,s,t)$ be a 2-acyclic quiver with
set of vertices $Q_0 = \{ 1,\ldots,n \}$.
The \emph{mutation} $\mu_k(Q)$ of $Q$ at $k$ is
another 2-acyclic quiver obtained from $Q$ in three steps:
\begin{itemize}

\item[(1)]
For each pair $(\bet,\alp)$ of arrows
in $\{ (\bet,\alp) \in Q_1^2 \mid s(\bet) = k = t(\alp) \}$
add a new arrow $[\beta\alp]$ with $s([\beta\alp]) = s(\alp)$ and
$t([\beta\alp]) = t(\bet)$.

\item[(2)]
Each arrow $\alp$ with $k \in \{s(\alp),t(\alp)\}$ is replaced by an arrow
$\alp^*$ in the opposite direction.

\item[(3)]
Remove a maximal collection of pairwise disjoint 2-cycles.

\end{itemize}
Mutations of quivers or equivalently of skew-symmetric
matrices were introduced and studied by Fomin and Zelevinsky
\cite{FZ1}
as part of their definition of a cluster algebra.

\subsection{Mutations of quivers with potential}\label{sssec:mutQP}
Let $(Q,S)$ be a quiver with potential, and let
$\cP(Q,S)$ be the corresponding Jacobian algebra.
Let $k$ be a vertex of $Q$ such that $k$ does not lie on a
2-cycle.
In this situation, Derksen, Weyman and Zelevinsky \cite[Section~5]{DWZ1}
define the \emph{premutation}
$\tilde{\mu}_k(Q,S) := (\tilde{Q},\tilde{S})$ in three
steps:
\begin{itemize}

\item[(1)]
For every pair of arrows $(\bet,\alp)$ in
$Q_{2,k}:=\{(\bet,\alp)\in Q_1^2\mid s(\bet)=k=t(\alp)\}$ add a
new arrow $[\bet\alp]$ with $s([\bet\alp]) = s(\alp)$ and $t([\bet\alp]) = t(\bet)$.

\item[(2)]
Each arrow $\alp$ with $k\in\{s(\alp),t(\alp)\}$ is replaced by an
arrow $\alp^*$ in opposite direction.

\item[(3)]
Set
\[
\tilde{S} := [S]+\Delta_k(Q)
\]
where $[S]$ is obtained from $S$ by replacing (after possibly some rotation
of cycles) each occurrence of a sequence of arrows
$\bet\alp$ with $(\bet,\alp)\in Q_{2,k}$ by $[\bet\alp]$, and
\[
\Delta_k(Q) := \sum_{(\bet,\alp)\in Q_{2,k}}[\bet\alp]\alp^*\bet^*.
\]

\end{itemize}

Following \cite[Definition~5.5]{DWZ1} let
$\mu_k(Q,S)$ be the reduced part of the QP $\tilde{\mu}_k(Q,S)$.
Note that the QP $\mu_k(Q,S)$ is only defined up to right equivalence.
The QP $\mu_k(Q,S)$ is the QP-\emph{mutation} of $(Q,S)$ at $k$.

\subsection{Non-degenerate and rigid potentials}
For cluster algebra purposes one wishes that
a QP $(Q,S)$ and also all possible iterations of QP-mutations
of $(Q,S)$ are 2-acyclic.
In this case, $S$ and also $(Q,S)$ is called \emph{non-degenerate}.
Derksen, Weyman and Zelevinsky \cite{DWZ1} have shown that
(for uncountable ground fields)
any given 2-acyclic quiver
admits at least one non-degenerate potential.

To decide whether a given potential $S$ is non-degenerate turns out
to be an extremely difficult problem.
At least in principle, one needs to
apply all possible sequences of QP-mutations and check that one always
obtains 2-acyclic QPs.
There is, however, a condition on QPs
that guarantees non-degeneracy without the need of applying
QP-mutations.
This condition is called \emph{rigidity}.
A potential $S$ on $Q$ and also $(Q,S)$ is called \emph{rigid} if every oriented
cycle in $Q$ is
cyclically equivalent to an element of the Jacobian ideal $J(S)$.

\subsection{Restriction of potentials}
For a QP $(Q,S)$ and a subset $I$ of the vertex set $Q_0$
let $(Q|_I,S|_I)$ be the \emph{restriction} of $(Q,S)$ to $I$.
By definition (see \cite[Definition~8.8]{DWZ1}),
$Q|_I$ is the full subquiver of $Q$ with vertex set $I$, and
$S|_I$ is
obtained from $S$ by deleting all summands
of the form $\lam_w w$ with $w$ not a cycle in the subquiver $Q|_I$
of $Q$.
The following statement follows from
\cite[Proposition~8.9]{DWZ1} and
\cite[Corollary~22]{Labardini1}.

\begin{prop}\label{restriction}
Let $(Q,S)$ be a QP, and let $I$ be a subset of $Q_0$.
If $(Q,S)$ is non-degenerate (resp. rigid), then $(Q|_I,S|_I)$
is non-degenerate (resp. rigid).
\end{prop}

\subsection{The appearance of cycles in non-degenerate
potentials}

\begin{prop}\label{proptcycle}
Let $(Q,S)$ be a QP, and let $I$ be a subset of $Q_0$
such that the following hold:
\begin{itemize}

\item[(i)]
$Q|_I$ contains exactly $t$ arrows $\alpha_1,\ldots,\alpha_t$;

\item[(ii)]
$c := \alpha_1 \cdots \alpha_t$ is a cycle in $Q$;

\item[(iii)]
The vertices
$s(\alpha_1),\ldots,s(\alpha_t)$ are pairwise different;

\item[(iv)]
$S$ is non-degenerate.

\end{itemize}
Then the cycle $c$ appears in $S$.
\end{prop}

\begin{proof}
The restriction $S|_I$ of $S$ is a non-degenerate potential on $Q|_I$
by Proposition~\ref{restriction}.
By our assumptions the quiver $Q|_I$ is a cyclically oriented quiver of 
Euclidean type $\widetilde{A}_{t-1}$.
This implies that the cycle $c$ appears in $S|_I$.
(It is fairly easy to see that a potential $W$ on $Q|_I$ is non-degenerate if 
and only if $c$ appears in $W$ up to rotation.)
Therefore $c$ appears also in $S$.
\end{proof}

\begin{coro}\label{prop3cycle}
Let $(Q,S)$ be a QP with $S$
non-degenerate.
Suppose that $\alp\bet\gam$ is a $3$-cycle
in $Q$ such that there are no multiple arrows between the three vertices
$s(\alp)$, $s(\bet)$ and $s(\gam)$.
Then $\alp\bet\gam$ appears in $S$.
\end{coro}


\section{Mutation invariance of representation type}
\label{secreptype}


In this section we prove
Theorem~\ref{thmmutationinv}.
After some preparations in Section~\ref{secreptype1}, where
we prove some elementary properties of split morphisms between
free modules over commutative $\ka$-algebras, we show
in Section~\ref{secreptype2} that
the mutation of certain bimodules commutes with taking tensor
products.
This is the main ingredient for the proof of Theorem~\ref{thmmutationinv}, which is given in
Subsection~\ref{secreptype3}.

\subsection{Mutations of representations}\label{sssec:MutRep}
Let $(Q,W)$ be a QP, and assume that $k$ is a vertex of $Q$ which
does not lie on a 2-cycle.
Define $(\tilde{Q},\tilde{W}) := \tilde{\mu}_k(Q,W)$.
If $M$ is a representation of $\cP (Q,W)$, Derksen, Weyman and
Zelevinsky \cite{DWZ1} defined a
representation $\tilde{\mu}_k(M):=\tilde{M}$ of $\cP(\tilde{Q},\tilde{W})$.
To describe it properly, let
\[
\Min{M}{k}  := \bigoplus_{\alp\in Q_1\df t(\alp)=k} M(s(\alp))
\text{\quad and \quad}
\Mout{M}{k} := \bigoplus_{\bet\in Q_1\df s(\bet)=k} M(t(\bet)).
\]
Furthermore, we need the three linear maps
\begin{alignat*}{2}
\Malp{M}{k} &\colon\ \Min{M}{k} \ra M(k),&
(m_\alp)_{\alp\in Q_1\df t(\alp)=k} &\mapsto
\sum_{\alp\in Q_1\df t(\alp)=k} M(\alp)(m_\alp),\\
\Mbet{M}{k} &\colon\  M(k)\ra \Mout{M}{k},&
m&\mapsto \left(M(\bet)(m)\right)_{\bet\in Q_1\df s(\bet)=k}, \\
\Mgam{M}{k} &\colon\ \Mout{M}{k} \ra \Min{M}{k}, & \quad
(m_\bet)_{\bet\in Q_1\df s(\bet)=k}&\mapsto
\left(\sum_{\bet\in Q_1\df s(\bet)=k}  M(\partial_{\bet,\alp}(W))(m_\bet)\right)_{\alp\in Q_1\df t(\alp)=k}.
\end{alignat*}
The definition of $\partial_{\bet,\alp}(W)$ can be found in
\cite[Section~4]{DWZ2}.
Furthermore, we have to choose retractions
$p_\alp\colon \Min{M}{k}\ra \Ker(\Malp{M}{k})$ to the inclusion $i_\alp$, and
$p_\gam\colon \Mout{M}{k}\ra \Ker(\Mgam{M}{k})$ to the inclusion $i_\gam$.

Now, we define $\bar{M}:=\tilde{\mu}_k(M)$ as follows:
\[
\bar{M}(i)=\begin{cases} M(i) &\text{ if } i\neq k,\\
 \Ker(\Mgam{M}{k})/\Image(\Mbet{M}{k})\oplus \Ker(\Malp{M}{k}) &\text{ if } i=k,
\end{cases}
\]
and on arrows
\begin{itemize}

\item
$\bar{M}([\bet\alp]) = M(\bet) M(\alp)$ for all $(\bet,\alp)\in Q_{2,k}$,

\item
$\bar{M}(\gam)=M(\gam)$ for all $\gam\in Q_1\cap \tilde{Q}_1$,

\item
$\Malp{\bar{M}}{k}=
\left(\begin{smallmatrix}
p\circ p_\gam\\ p_\alp\circ \Mgam{M}{k}
\end{smallmatrix}\right)\colon \Min{\bar{M}}{k}\ra \bar{M}(k)$, with
$\Min{\bar{M}}{k}=\Mout{M}{k}$,\\
and $p\colon \Ker(\Mgam{M}{k})\ra \Ker(\Mgam{M}{k})/\Image(\Malp{M}{k})$
the canonical projection,

\item
$\Mbet{\bar{M}}{k}= (0,i_\alp)\colon \bar{M}(k)\ra \Mout{\bar{M}}{k}=\Min{M}{k}$.

\end{itemize}
Recall, that
$\Malp{\bar{M}}{k}\equiv ((\bar{M}(\bet^*))_{\bet\in Q_1\colon s(\bet)=k})^T$ and
$\Mbet{\bar{M}}{k}\equiv ((\bar{M}(\alp^*))_{\alp\in Q_1\colon t(\bet)=k})$.

Note that our description for mutations of representations
is slightly simplified with respect to the original definition by
Derksen, Weyman and Zelevinsky, as for example in
~\cite[Section~4]{DWZ2}.
First of all, we only consider \emph{undecorated} representations.
Moreover, we use for $\bar{M}(k)$ a less symmetric, though isomorphic
definition which is more convenient for our purpose.
(The original definition makes it more clear that $\bar{M}(k)$ should be
viewed as a glueing of $\Ker(\Malp{M}{k})$ and $\Coker(\Mbet{M}{k})$ along
$\Image(\Mgam{M}{k}) \cong \Mout{M}{k}/\Ker(\Mgam{M}{k})$.)

\subsection{Generic Images and Kernels}\label{secreptype1}
Let $R$ be a commutative $\ka$-algebra.
In this subsection all tensor products
are over $R$ so that we write $\otimes$ instead of $\otimes_R$ for typographical
reasons.

We say that a homomorphism $g$ in $\Hom_R(R^m,R^n)$ \emph{splits} if there exist submodules
$K'\leq R^m$ and $I'\leq R^n$ such that $R^m=\Ker(g)\oplus K'$ and
$R^n=\Image(g)\oplus I'$.

\begin{lemma} \label{lem:split1}
For $g \in \Hom_R(R^m,R^n)$ consider the following diagram
consisting of the obvious inclusions and projections:
\[\xymatrix@-1.2pc{
0\ar[rd]&                   &                           &&&                &0\\
        &\Ker(g)\ar[rd]_{i_0}&                           &&&\Coker(g)\ar[ru]&\\
        &                   &R^m\ar[rr]^g\ar[rd]_>>>>{\bar{g}}&&R^n\ar[ru]_p\\
        &                   &&\Image(g)\ar[rd]\ar[ru]_{i_1}&\\
        &                   &0\ar[ru]&&0
}\]
Assume that $g$ splits.
Then for each $R$-module $S$ the following hold:
\begin{itemize}

\item[(a)]
The maps $i_1\otimes S$
and $i_0\otimes S$ induce isomorphisms of vector spaces
\begin{align}
\label{eq:isoIM}
\overline{i_1\otimes S}\colon\ \Image(g)\otimes S &
\xrightarrow{\sim}  \Image(g\otimes S) \text{ and}\\
\label{eq:isoKER}
\overline{i_0\otimes S}\colon \Ker(g)\otimes S &
\xrightarrow{\sim} \Ker(g\otimes S), \text{ respectively}.
\end{align}
In particular, with
$i_1^S\colon\Image(g\otimes S)\hookrightarrow R^n\otimes S$ and
$i_0^S\colon\Ker(g\otimes S)\hookrightarrow R^m\otimes_R S$ we have
$i_1^S\circ(\overline{i_1\otimes S})=i_1\otimes S$ and
$i_0^S\circ(\overline{i_0\otimes S})=i_0\otimes S$.

\item[(b)]
Let $p_0$ be a retraction for $i_0$, i.e.~$p_0\circ i_0=\Id_{\Ker(g)}$, then
$(\overline{i_0\otimes S})\circ (p_0\otimes S) $ is a retraction for
$i_0^S$.

\end{itemize}
\end{lemma}

\begin{proof}
In the above diagram the two short exact sequences split by hypothesis.
In particular, they remain exact under the functor $-\otimes S$.

(a) Since $-\otimes S$ is right exact we have
\[
\Image(g\otimes S)=\Ker(p\otimes S)=\Image(i_1 \otimes S).
\]
By our hypothesis $i_1\otimes S$ is injective which implies~\eqref{eq:isoIM}.

By the same token
\[
\Ker(g\otimes S)=\Ker(\bar{g} \otimes S)=\Image(i_0\otimes S).
\]
Again, by hypothesis $i_0 \otimes S$ is injective which
implies~\eqref{eq:isoKER}.

(b) We have
\[
\left((\overline{i_0\otimes S})\circ (p_0\otimes S)\circ i_0^S)\right)
\circ (\overline{i_0\otimes S})
= (\overline{i_0\otimes S}))\circ (p_0\otimes S)\circ (i_0\otimes S)
= (\overline{i_0\otimes S})\circ \Id_{\Ker(g)\otimes S},
\]
which implies our claim, since $\overline{i_0\otimes S}$ is an isomorphism.
\end{proof}

\begin{lemma} \label{lem:split2}
Consider split morphisms $f\in\Hom_R(R^l,R^m)$ and $g\in\Hom_R(R^m,R^n)$ with
$g \circ f=0$. We have then inclusions $i\colon \Image(f)\hookrightarrow\Ker(g)$,
$i_0\colon\Ker(g)\hookrightarrow R^m$ and
$i_1=i_0\circ i\colon\Image(f)\hookrightarrow R^m$. Moreover, we have the
projection $p\colon\Ker(g)\ra\Ker(g)/\Image(f)$.
With this notation we obtain for any $R$-module $S$ the following
commutative diagram with exact rows and all vertical arrows isomorphisms:
\[\xymatrix{
0\ar[r]&{\Image(f)\otimes S}\ar[d]_{\overline{i_1\otimes S}}\ar[r]^{i\otimes S}
       &{\Ker(g)\otimes S}\ar[d]^{\overline{i_0\otimes S}}\ar[r]^{p\otimes S}
       &\dfrac{\Ker(g)}{\Image(f)}\otimes S\ar@{.>}[d]^{c_{f,g}^S}\ar[r]& 0\\
0\ar[r]&\Image(f\otimes S)\ar[r]_{i^S}&\Ker(g\otimes S)\ar[r]_{p^S}
&\dfrac{\Ker(g\otimes S)}{\Image(f\otimes S)}\ar[r]& 0
}\]
The morphisms in the bottom row are the canonical inclusion and projection.
\end{lemma}

\begin{proof}
Since  the morphisms $f$ and $g$ split, we know that
$i_0\otimes S$ and $i_1\otimes S$ are injective, and that
$\overline{i_0\otimes S}$ and ${i_1\otimes S}$ are isomorphisms by
Lemma~\ref{lem:split1}.
Since $(i_0\otimes S)\circ(i\otimes S)=i_1\otimes S$, also $i\otimes S$ is
injective.
It remains to show that in the above diagram the
left square commutes. To this end let
$i_0^S\colon\Ker(g\otimes S)\hookrightarrow R^m\otimes S$ and
$i_1^S\colon\Image(f\otimes S)\hookrightarrow R^m\otimes S$ be the canonical
inclusions. Then $i_0^S \circ i^S = i_1^S$. Thus, by Lemma~\ref{lem:split1}~(a)
we have
\[
i_0^S\circ \left(i^S\circ(\overline{i_1\otimes S})\right)= i_1\otimes S=
(i_0\otimes S)\circ(i\otimes S)=
i_0^S\circ \left((\overline{i_0\otimes S})\circ (i\otimes S)\right).
\]
Since $i_0^S$ is injective the required commutativity follows.
\end{proof}

\begin{lemma} \label{lem:split-loc}
Let $R$ be an  integral domain and
$g\in\Hom_R(R^m,R^n)$, then there exists some $d \in R$ such that the
localization $g_d\colon R_d^m\ra R_d^n$ splits.
\end{lemma}

\begin{proof}
Let $F$ be the fraction field of $R$ and $g_F\colon F^m\ra F^n$ the
localization of $g$. If $r$ is the rank of $g_F$ we can find $F$-bases
$\underline{b}=(b_1,\ldots,b_m)$ resp. $\underline{b}'=(b'_1,\ldots,b'_n)$
of $F^m$ resp. $F^n$ such that $(b_1,\ldots,b_{m-r})$ is a basis
of $\Ker(g_F)$ and $g(b_{m-r+i})=b'_i$ for $1 \le i \le r$.

Denote by $(v_1,\ldots,v_m)$ resp. $(v'_1,\ldots,v'_n)$ the standard basis
of $F^m$ resp. $F^n$. Thus,
\[
b_i=\sum_{j=1}^m \frac{x_{ij}}{y_{ij}} v_j \text{ and }
b'_i=\sum_{j=1}^n \frac{x_{ij}}{y_{ij}} v'_j
\]
for certain $x_{ij}, x'_{ij}\in R$ and $y_{ij}, y'_{ij}\in R\setminus\{0\}$. With
\[
d:= (\prod_{i,j=1}^m y_{ij})\cdot(\prod_{i,j=1}^n y'_{ij})\cdot
\det((\frac{x_{ij}}{y_{ij}})_{i,j=1,\ldots,m})\cdot
\det((\frac{x'_{ij}}{y'_{ij}})_{i,j=1,\ldots,n})
\]
we can view $\underline{b}$ as an $R_d$-basis
of $R_d^m$ and $\underline{b}'$ as an $R_d$-basis of $R_d^n$. Moreover,
$(b_1,\ldots,b_{m-r})$ is an $R_d$-basis of $\Ker(g_d)$ and
$(b'_1,\ldots,b'_r)$ is an $R_d$-basis of $R_d^n$. Thus, $g_d$ splits.
\end{proof}

\begin{remark}
It follows from the proof that in Lemma~\ref{lem:split-loc} above we may
assume that $\Ker(g_d)$ and $\Image(g_d)$ are free $R_d$-modules.
\end{remark}

\subsection{Mutation of Bimodules}\label{secreptype2}
In this section $(Q,W)$ is a QP over the field $\ka$,
$\cP(Q,W)$  is the corresponding Jacobian algebra, and $R$ denotes
a commutative $\ka$-algebra.

We view a $\cP(Q,W)\text{-}R$-bimodule $M$
(which is free as an $R$-module) as a nilpotent covariant
representation of $Q$ into the category of (free) $R$-modules which respects
the relations $\partial_\alpha(W)$ for all arrows $\alpha$ of $Q$.

The mutation operation for modules as described in
Section~\ref{sssec:MutRep} also makes  sense for any
$\cP(Q,W)$-$R$-bimodule
$M$ which is finitely generated and free as an
$R$-module, if we assume that the corresponding $R$-module homomorphisms
$\Malp{M}{k}$ and $\Mgam{M}{k}$ between free $R$-modules split.
In fact,
this allows us to choose the required retractions $p_\alp$ and $p_\gam$.

If $R$ is a domain, and $M$ is just finitely generated and
free as an $R$-module,  we can find $d\in R$ such that the
localized $\cP(Q,W)\text{-}R_d$-bimodule $M_d$  has the property that
$\Malp{M_d}{k}, \Mbet{M_d}{k}$ and $\Mgam{M_d}{k}$ split.
Moreover, we may assume
then that the $\cP(\tilde{Q},\tilde{W})\text{-}R_d$-bimodule
$\tilde{\mu}_k(M_d)$ is free (and finitely generated) as an
$R_d$-module.

\begin{prop} \label{prp:mut-tens}
Let $(Q,W)$ be a QP, and let $k$ be a vertex of $Q$ which does not lie on a $2$-cycle, and let $R$ be a $\ka$-algebra which is a domain.
Then consider the premutation $(\tilde{Q},\tilde{W}) :=
\tilde{\mu}_k(Q,W)$.
Let  $M$ be a
$\cP(Q,W)\text{-}R$-bimodule which is free and finitely generated as
an $R$-module, and suppose that the maps $\Malp{M}{k}, \Mbet{M}{k}$ and
$\Mgam{M}{k}$ split.
Then we have for each $R$-modules $S$  a canonical
$\cP(\tilde{Q},\tilde{W})$-module isomorphism
\[
\tilde{\mu}_k(M)\otimes_R S \xrightarrow{\sim} \tilde{\mu}_k(M\otimes_R S).
\]
\end{prop}

\begin{proof}
Recall that
\begin{align*}
(\tilde{\mu}_k(M)\otimes_R S)(k) &=\left(
\frac{\Ker(\Mgam{M}{k})}{\Image(\Mbet{M}{k})}\oplus
\Ker(\Malp{M}{k})\right)\otimes_R S,\\
(\tilde{\mu}_k(M\otimes_R S))(k) &=
\frac{\Ker(\Mgam{M}{k}\otimes_R S)}{\Image(\Mbet{M}{k}\otimes_R S)}\oplus
\Ker(\Malp{M}{k}\otimes_R S).
\end{align*}
For the rest of the proof, all tensor products are over $R$, so that we
will write $\otimes$ instead of $\otimes_R$ for typographical reasons.
Given our hypotheses, we may use by Lemma~\ref{lem:split1}~(b) the
map $p_\alp^S:=(\overline{i_\alp\otimes S})\circ (p_\alp\otimes S)$ as retraction
for the natural inclusion
$i_\alp^S\colon\Ker(\Malp{M}{k}\otimes S)\hookrightarrow \Min{M}{k}\otimes S$,
and the map $p_\gam^S:=(\overline{i_\gam\otimes S})\circ (p_\gam\otimes S)$
as retraction for the natural inclusion
$i_\gam^S\colon\Ker(\Mgam{M}{k}\otimes S)\hookrightarrow \Mout{M}{k}\otimes S$.
We will take these maps for the definition of the structure maps of
$\tilde{\mu}_k(M\otimes S)$.

Moreover, our hypotheses imply that we have by Lemma~\ref{lem:split1}~(a)
and Lemma~\ref{lem:split2} a natural isomorphism
\[
\vph_k:=\left(\begin{smallmatrix}
c_{\gam,\bet}^S & 0\\ 0& \overline{i_\alp\otimes S}
\end{smallmatrix}\right)\colon
\left
(\frac{\Ker(\Mgam{M}{k})}{\Image(\Mbet{M}{k})}\oplus
\Ker(\Malp{M}{k})\right)\otimes S
\xrightarrow{\sim}
\frac{\Ker(\Mgam{M}{k}\otimes S)}{\Image(\Mbet{M}{k}\otimes S)}\oplus
\Ker(\Malp{M}{k}\otimes S).
\]
We set $\vph_i=\Id_{M(i)\otimes S}$ for all $i\in Q_0\setminus\{k\}$ and claim
the $\vph=(\vph_j)_{j\in Q_0}$ is the required isomorphism of
$\cP(\tilde{Q},\tilde{W})$-modules.
By the construction of $\vph$ it is
sufficient to show that it is in fact a homomorphism of representations
of $\cP(\tilde{Q},\tilde{W})$.
From the definition of mutation of representations it follows that to this end
it is sufficient to verify the two equalities
\begin{align}
\label{eq:MutBiC1}
\vph_k\circ\left((\Malp{(\tilde{\mu}_k(M))}{k})\otimes S\right)&=
(\Malp{(\tilde{\mu}_k(M\otimes S))}{k}\text{ and } \\
\label{eq:MutBiC2}
\Mbet{(\tilde{\mu}_k(M))}{k}&=
\left(\Mbet{(\tilde{\mu}_k(M\otimes S))}{k}\right)\circ\vph_k.
\end{align}
Now, after expanding, Equation~\eqref{eq:MutBiC1} reads as follows:
\begin{multline*}
c_{\gam,\bet}^s\circ ((p\otimes S)\circ (p_\gam\otimes S))=
p^S\circ \left((\overline{i_\gam\otimes S})\circ (p_\gam\otimes S)\right)
\text{ and }\\
(\overline{i_\alp\otimes S})\circ ((p_\alp\otimes S)\circ (\Mgam{M}{k}\otimes S))
=((\overline{i_\alp\otimes S})\circ (p_\alp\otimes S))\circ
(\Mgam{M}{k}\otimes S)).
\end{multline*}
The first of these equations holds by Lemma~\ref{lem:split2}, whilst the
second one is trivial.
Finally, after expanding, Equation~\eqref{eq:MutBiC2} becomes
\[
i_\alp\otimes S = i_\alp^S \circ (\overline{i_\alp\otimes S}),
\]
which holds by Lemma~\ref{lem:split1}~(a).
\end{proof}

\subsection{Mutation invariance of tameness}\label{secreptype3}
Recall that all simple $\ka[X]$-modules are of the form
\[
S_\lambda := \ka[X]/(X-\lam)
\]
with $\lambda \in \ka$, and that $S_\lam\cong S_\mu$ if and only if $\lam=\mu$.
Let $d \in \ka[X]$, and let $\ka[X]_d$ be the localization of
$\ka[X]$ at $d$.
By a slight abuse of notation, for
$\lam \in \ka$ with $d(\lam) \not= 0$, we also write $S_\lam$
for the simple $\ka[X]_d$-module $\ka[X]_d/(X-\lam)$.

In the above definition, some of the bimodules might be trivial in the
sense that $M_i\otimes_{\ka[X]_d} S_\lam \cong M_i\otimes_{\ka[X]_d} S_\mu$ for
all $\lam,\mu\in \ka$ with $d(\lam)\neq 0 \neq d(\mu)$.

A Jacobian algebra $\cP(Q,W)$ is \emph{tame}, if for
each dimension vector $\bv \in \NN^{Q_0}$ there exists a polynomial
$d \in \ka[X]$, and a finite number
of $\cP(Q,W)$-$\ka[X]_d$-bimodules
$M_1,\ldots M_m$, which are finitely
generated and free as $\ka[X]_d$-modules, with the following property:
Each indecomposable
$\cP(Q,W)$-module $N$ with $\dimv(N)=\bv$ is isomorphic to
$M_i \otimes_{\ka[X]_d} S_\lam$ for some $1 \le i \le m$ and some $\lam\in\ka$ with $d(\lam) \neq 0$.
It is not difficult to show that this definition of tameness is equivalent
to the one from Section~\ref{secreptypeintro}.

Using Drozd's Tame-Wild Theorem~\ref{thmdrozd}, the invariance
of the representation type of Jacobian algebras under mutation, as
stated in Theorem~\ref{thmmutationinv}, follows from the following
result.

\begin{thm}
Let $(Q,W)$ be a QP  and suppose that $k\in Q_0$ does not
lie on a
$2$-cycle, so that $\tilde{\mu}_k(Q,W)$ is defined.
Then the following are equivalent:
\begin{itemize}

\item[(i)]
$\cP(Q,W)$ is tame;

\item[(ii)]
$\cP(\tilde{\mu}_k(Q,W))$ is tame;

\item[(iii)]
$\cP(\mu_k(Q,W))$ is tame.

\end{itemize}
\end{thm}

\begin{proof}
Define $(\tilde{Q},\tilde{W}) := \tilde{\mu}_k(Q,W)$.
Assume that $\cP(Q,W)$ is tame.
Let $\bv\in\NN^{Q_0}$ be a dimension vector. We may suppose that $\bv$ is
not the dimension vector of the simple $\cP(\tilde{Q},\tilde{W})$-module $S_k$.
It follows from (the proof of)~\cite[Theorem 10.13]{DWZ1} that each
indecomposable $\cP(\tilde{Q},\tilde{W})$-module $N$ with $\dimv(N)=\bv$
is isomorphic to a module of the form $\tilde{\mu}_k(M)$ for some
$\cP(Q,W)$-module $M$ with $\dim M(i)=\dim N(i)$ for all
$i\in Q_0\setminus\{k\}$ and $\dim M'(k)\leq \dim \Min{N}{k}+\dim \Mout{N}{k}$.

Thus, since $\cP(Q,W)$ is tame, there exist some $d\in\C[X]$ and some
$\cP(Q,W)\text{-}\ka[X]_d$-bimodules $M_1,\ldots, M_m$, which are finitely
generated and free as $\ka[X]_d$-modules, such that the following holds:
Each indecomposable $\cP(\tilde{Q},\tilde{W})$-module $N$ with $\dimv(N)=\bv$
is isomorphic to $\tilde{\mu}_k(M_i\otimes_{\ka_d[X]}  S_\lam)$ for some
$i$ and some $\lam\in\C$ with $d(\lam)\neq 0$.
By Lemma~\ref{lem:split-loc} we may assume  without loss of generality that the
maps $\Malp{M_i}{k}$, $\Mbet{M_i}{k}$ and $\Mgam{M_i}{k}$
(defined in~\ref{sssec:MutRep}) split for all $1 \le i \le m$.
In fact, if necessary we just localize and add a finite number of
trivial bimodules. Then, by Proposition~\ref{prp:mut-tens} we have for the
$\cP(\tilde{Q},\tilde{W})\text{-}\ka[X]_d$-bimodules $\tilde{\mu}_k(M_i)$
that $\tilde{\mu}_k(M_i\otimes_{\ka[X]_d} S_\lam)\cong
\tilde{\mu}_k(M_i)\otimes_{\ka[X]_d} S_\lam$ for all $i$ and all $\lam\in\ka$
with $d(\lam)\neq 0$.
Thus $\cP(\tilde{Q},\tilde{W})$ is tame.
The opposite direction is proved in the same way since mutation of
QPs is involutive on isomorphism classes of Jacobian algebras.
Note that the QP-mutation at a vertex $k$ is defined as long as $k$
does not lie on a $2$-cycle, so reducedness is not required.
Thus (i) and (ii) are equivalent.

The equivalence of (ii) and (iii) is straightforward.
\end{proof}

\subsection{Nearly Morita equivalence for Jacobian algebras}\label{secnearmorita}
In this section, we discuss an alternative strategy for proving
Theorem~\ref{thmmutationinv}.
For a $K$-algebra $\LL$ and an $A$-module $M \in \md(\LL)$
let $\add(M)$ be the additive subcategory generated by $M$, i.e.
$\add(M)$ consists of the $A$-modules that are isomorphic to
finite direct sums of direct summands of $M$.
The
$M$-\emph{stable category}
\[
\stmd_M(\LL) := \md(\LL)/\add(M)
\]
has by definition the same objects as $\md(\LL)$, and the
morphism spaces are the morphism spaces from $\md(\LL)$
modulo the subspaces of morphism factoring through some object
in $\add(M)$.

Let $S$ be a potential on $Q$, and let
$(Q',S') = \mu_k(Q,S)$ be a QP-mutation of $(Q,S)$.
Set $\LL := \cP(Q,S)$ and $\LL' := \cP(Q',S')$.
Buan, Iyama, Reiten and Smith \cite{BIRSm} prove the following result
showing that the mutation of Jacobian algebras almost
induces a Morita equivalence of the corresponding module
categories.

\begin{thm}[{\cite[Theorem~7.1]{BIRSm}}]\label{nearmorita}
There is an equivalence of additive categories
\[
f\df \stmd_{S_k}(\LL) \to \stmd_{S_k}(\LL').
\]
\end{thm}

Using Crawley-Boevey \cite[Theorem~4.4]{CB2},
Krause \cite[Corollary~3.4]{K} proved that stable equivalences of dualizing algebras
preserve the representation type.
For details and definitions we refer to \cite{K}.
Combining Krause's result with Theorem~\ref{nearmorita}
we get the following statement.

\begin{thm}\label{krause1}
Suppose that $\LL$ and $\LL'$ are dualizing algebras.
Then the representation types of $\LL$ and $\LL'$ coincide.
\end{thm}

If $\LL$ is finite-dimensional, then $\LL$ and $\LL'$
are both dualizing algebras.
Thus we got a proof for Theorem~\ref{thmmutationinv} for
all finite-dimensional Jacobian algebras.
(Recall from \cite{DWZ1} that mutations of finite-dimensional
Jacobian algebras are again finite-dimensional.)
We expect that a slight modification
of Krause's proof of \cite[Corollary~3.4]{K} yields a proof
of
Theorem~\ref{krause1} also for
infinite dimensional Jacobian algebras $\LL$ and $\LL'$.


\section{Triangulations of marked surfaces and quiver mutations}
\label{sectriang}


\subsection{}
This section is aimed at recalling some definitions and facts on
quivers arising from
triangulations of oriented surfaces with marked points.
Our main reference for this section is \cite{FST}.

\subsection{Triangulations of marked surfaces}
A \emph{marked surface}, or simply a \emph{surface}, is a pair $\surf$, where $\surfnoM$ is a compact connected oriented surface with (possibly empty) boundary, and $\marked$ is a finite set of points on $\surfnoM$, called \emph{marked points}, such that $\marked$ is non-empty and has at least one point from each connected component of the boundary of $\surfnoM$.
The marked points that lie in the interior of $\surfnoM$ are called \emph{punctures}, and the set of punctures of $\surf$ is denoted by $\punct$.
A marked surface $\surf$ is \emph{unpunctured} if
$\punct =\varnothing$.

Throughout the paper we will always assume that $\surf$ is none of the following:
\begin{itemize}

\item
an unpunctured monogon, digon or triangle;

\item
a once-punctured monogon or digon;

\item
a sphere with less than four punctures.

\end{itemize}
By a \emph{monogon} (resp. \emph{digon}, \emph{triangle}) we mean a disk with exactly one (resp. two, three) marked point(s) on the boundary.
A \emph{disk} (resp. \emph{annulus}) is by definition a surface of genus $0$ with exactly
one (resp. two) boundary component(s).
A \emph{sphere} is a surface with empty boundary and genus $0$, and
by a \emph{torus} we mean a surface with possibly non-empty
boundary and genus $1$.

Let $\surf$ be a marked surface.
\begin{enumerate}

\item
An \emph{arc} in $\surf$ is a curve $\arc$ in $\surfnoM$ such that the following hold:
\begin{itemize}
\item
the endpoints of $\arc$ belong to $\marked$;
\item
$\arc$ does not intersect itself, except that its endpoints may coincide;
\item
the relative interior of $\arc$ is disjoint from $\marked$ and from the boundary of $\surfnoM$;
\item
$\arc$ does not cut out an unpunctured monogon nor an unpunctured digon.
\end{itemize}

\item
An arc whose endpoints coincide will be called a \emph{loop}.

\item
Two arcs $\arc_1$ and $\arc_2$ are \emph{isotopic} if there exists an isotopy $H\df I \times \surfnoM \rightarrow \surfnoM$ such that $H(0,x)=x$ for all $x\in\surfnoM$, $H(1,\arc_1)=\arc_2$, and $H(t,m)=m$ for all $t\in I := [0,1]$ and all $m \in \marked$.
Arcs will be considered up to isotopy and orientation.

\item
Two arcs are \emph{compatible} if there are arcs in their respective isotopy classes whose relative interiors do not intersect.

\item
An \emph{ideal triangulation}, or simply a \emph{triangulation},
of $\surf$ is any maximal collection of pairwise compatible arcs whose relative interiors do not intersect each other.

\end{enumerate}

The pairwise compatibility of any collection of arcs can be realized simultaneously. That is, given any collection of pairwise compatible arcs, it is always possible to find representatives in their isotopy classes whose relative interiors do not intersect each other.

For an ideal triangulation $\tau$, the \emph{valency} $\val_\tau(p)$ of
a puncture $p \in \punct$ is the number of arcs in $\tau$ incident
to $p$, where each loop at $p$ is counted twice.

Let $\tau$ be an ideal triangulation of a surface $\surf$.
\begin{enumerate}

\item
For each connected component of the complement in $\surfnoM$ of the union of the arcs in $\tau$, its topological closure $\triangle$ will be called an \emph{ideal triangle}, or simply a \emph{triangle},
of $\tau$.

\item
An ideal triangle $\triangle$ is called \emph{interior} if its intersection with the boundary of $\surfnoM$ consists only of (possibly none) marked points. Otherwise it will be called \emph{non-interior}.

\item
An interior ideal triangle $\triangle$ is \emph{self-folded} if it contains exactly two arcs $a$ and $b$ of $\tau$, see Figure~\ref{fig:selffoldedtriang}.
The arc $b$ is called the
\emph{folded side} of the ideal triangle.

\end{enumerate}
\begin{figure}[!htb]
\centering
\includegraphics[scale=.6]{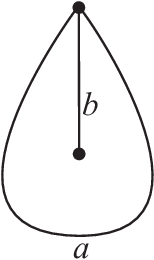}
\caption{
Self-folded triangle.}
\label{fig:selffoldedtriang}
\end{figure}

The number $n$ of arcs in any ideal triangulation of $\surf$ is determined by the following numerical data:
\begin{itemize}\item the genus $g=g(\Sigma)$ of $\Sigma$;
\item the number $b=b(\Sigma)$ of boundary components of $\Sigma$;
\item the number $p=p(\Sigma,\marked)$ of punctures;
\item the number $c=c(\Sigma,\marked)$ of marked points that lie on the 
boundary of $\surfnoM$.
\end{itemize}
Indeed, we have
\[
n=6g+3b+3p+c-6,
\]
a formula that can be proved using the definition and basic properties of the 
Euler characteristic.
Hence $n$ is an invariant of $\surf$, called the \emph{rank} of $\surf$.

\subsection{Flips of tagged triangulations and mutations of quivers}
Let $\tau$ be an ideal triangulation of $\surf$ and let $\arc\in\tau$ be an arc.
If $\arc$ is not the folded side of a self-folded triangle,
then there exists exactly one arc $\arc'$, different from $\arc$, such that $\sigma := (\tau\setminus\{\arc\})\cup\{\arc'\}$ is an ideal
triangulation of $\surf$.
We say that $\sigma$ is obtained from $\tau$ by applying a \emph{flip}
at $i$, or by \emph{flipping} the arc $\arc$, and we write $\sigma=f_i(\tau)$.
In order to be able to flip the folded sides of self-folded triangles,
Fomin, Shapiro and Thurston \cite{FST} introduce the notion of a \emph{tagged arc}.
A \emph{tagged triangulation} of $\surf$ is then defined to be any maximal collection of pairwise compatible tagged arcs.

All tagged triangulations of $\surf$
have the same cardinality (equal to the rank $n$ of $\surf$) and every collection of $n-1$ pairwise compatible tagged arcs is contained in precisely two tagged triangulations. This means that every tagged arc in a tagged triangulation can be replaced by a uniquely defined, different tagged arc that together with the remaining $n-1$ arcs forms a tagged triangulation.
In analogy with the ordinary case, this combinatorial replacement is called a \emph{flip}.

To every tagged triangulation $\tau$ Fomin-Shapiro-Thurston associate a skew-symmetric $n\times n$ integer matrix $B(\tau)$. When $\tau$ is actually an ideal triangulation the matrix $B(\tau)$ is defined as follows. The rows and columns of $B(\tau)$ correspond to the
arcs of $\tau$. Let $\pi_\tau\df \tau \rightarrow \tau$ be the function that is the identity on the set of arcs
that are not folded sides of self-folded triangles of $\tau$, and sends the folded side of a self-folded triangle to the unique loop of $\tau$
enclosing it. For each non-self-folded ideal triangle $\triangle$ of $\tau$, let $B^\triangle=b^\triangle_{ij}$ be the $n\times n$ integer
matrix defined by
\begin{equation}
b^\triangle_{ij}=
\begin{cases} 1\ \ \ \ \ \text{if $\triangle$ has sides $\pi_\tau(i)$ and $\pi_\tau(j)$, with $\pi_\tau(j)$ preceding
$\pi_\tau(i)$}\\ \ \ \ \ \ \ \ \ \ \text{in the clockwise order defined by the orientation of $\surfnoM$;}\\ -1\ \ \ \text{if the same holds,
but in the counter-clockwise
order;}\\
0\ \ \ \ \ \text{otherwise.}
\end{cases}
\end{equation}
The \emph{signed adjacency matrix} $B(\tau)$ is then defined as
\begin{equation}
B(\tau)=\underset{\triangle}{\sum}B^\triangle,
\end{equation}
where the sum runs over all non-self-folded triangles of $\tau$.

In the general situation, where the tagged triangulation $\tau$ is not necessarily an ideal triangulation, the signed adjacency matrix $B(\tau)$ is defined as the signed adjacency matrix of an ideal triangulation $\tau^\circ$ obtained from $\tau$ by \emph{deletion of notches},
see \cite{FST} for details.

Note that all entries of $B(\tau)$ have absolute value at most 2.
Moreover, $B(\tau) = (b_{ij})$ is skew-symmetric, hence gives rise to a 2-acyclic quiver, the \emph{adjacency quiver} $\qtau$, whose vertices $1,\ldots,n$ correspond to the tagged arcs in $\tau$,
with $b_{ij}$ arrows from $j$ to $i$ whenever $b_{ij}>0$.
Numerous examples of triangulations $\tau$ and quivers $Q(\tau)$
can be found in Section~\ref{sec11}.
For a tagged triangulation $\tau$ and the associated ideal
triangulation $\tau^\circ$, the adjacency quivers $Q(\tau)$
and $Q(\tau^\circ)$ are isomorphic.

The next result
due to Fomin, Shapiro and Thurston \cite{FST} shows that
flipping tagged arcs is compatible with quiver mutation.

\begin{prop}[{\cite[Proposition 4.8 and Lemma 9.7]{FST}}]
\label{thm:flip<->matrix-mutation}
Let $\tau$ and $\sigma$ be tagged triangulations. If $\sigma$ is obtained from $\tau$ by flipping the tagged arc $k$ of $\tau$, then $Q(\sigma)=\mu_k(Q(\tau))$.
\end{prop}


\section{Quivers of finite mutation type}\label{sec5}


\subsection{Classification of quivers of finite mutation type}\label{secclass}
A quiver of finite mutation type is called \emph{exceptional}
if it is not of the form $Q(\tau)$ for some triangulation $\tau$ of
some marked surface $(\Sigma,\marked)$.
The acyclic quivers of finite mutation type were classified
by Buan and Reiten~\cite{BR}.
They show that these are exactly the
quivers of type $A_n$, $D_n$, $E_m$, $\widetilde{A}_n$,
$\widetilde{D}_n$ and $\widetilde{E}_m$ for $m=6,7,8$
and the $m$-Kronecker quiver $K_m$ with two vertices
and $m \ge 3$ arrows.
Representatives of mutation equivalence classes of the exceptional acyclic quivers of finite mutation type are
displayed in Figure~\ref{figureL1}.
Up to mutation equivalence,
there are only five exceptional non-acyclic quivers.
One representative from each mutation equivalence class is
displayed in Figure~\ref{figureL2}.

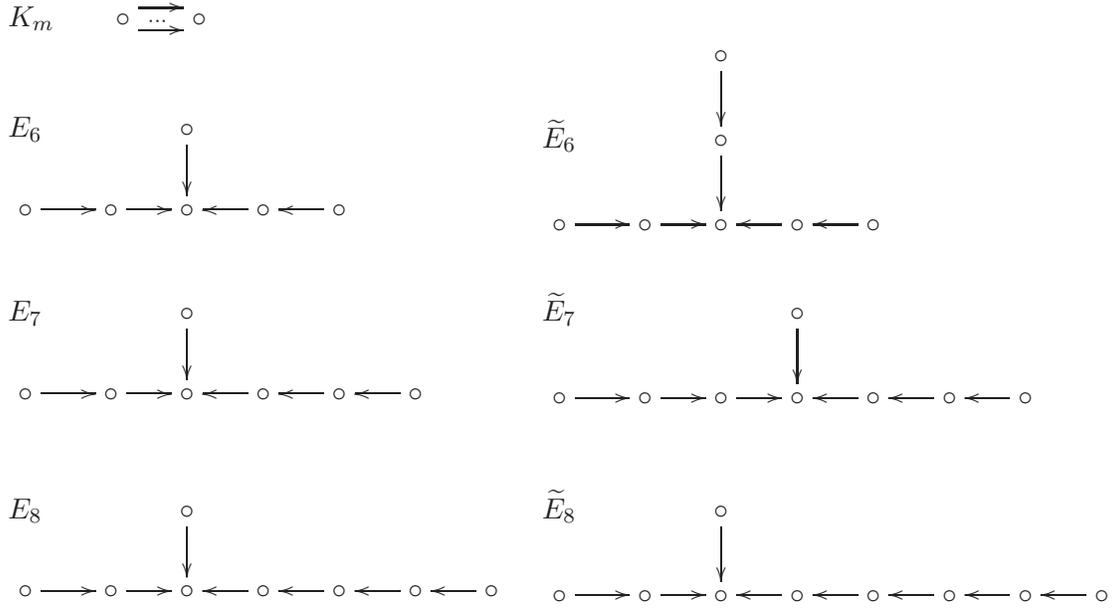
\begin{figure}[H]
\begin{tabular}{ll}
$\xymatrix@-0.6pc{
K_m & \circ \ar@<0.9ex>[r]_>>>>\cdots\ar@<-0.9ex>[r]
&\circ
}$
&\\
$\xymatrix@-0.6pc{
&&&\\
E_6 && \circ\ar[d]\\
\circ \ar[r] &\circ \ar[r] & \circ & \circ \ar[l] &\circ \ar[l]
}$
&
$\xymatrix@-0.6pc{
&& \circ\ar[d]\\
\widetilde{E}_6 && \circ \ar[d]\\
\circ \ar[r] &\circ \ar[r] & \circ & \circ \ar[l] &\circ \ar[l]
}$\\[2.8cm]
$\xymatrix@-0.6pc{
E_7 && \circ\ar[d]\\
\circ \ar[r] &\circ \ar[r] & \circ & \circ \ar[l] &\circ \ar[l] &\circ \ar[l]
}$
&
$\xymatrix@-0.6pc{
\widetilde{E}_7 &&& \circ\ar[d]\\
\circ \ar[r] &\circ \ar[r] &\circ \ar[r] & \circ & \circ \ar[l] &\circ \ar[l] &\circ \ar[l]
}$\\[2.0cm]
$\xymatrix@-0.6pc{
E_8 && \circ\ar[d]\\
\circ \ar[r] &\circ \ar[r] & \circ & \circ \ar[l] &\circ \ar[l] &\circ \ar[l] &\circ \ar[l]
}$
&
$\xymatrix@-0.6pc{
\widetilde{E}_8 && \circ\ar[d]\\
\circ \ar[r] &\circ \ar[r] & \circ & \circ \ar[l] &\circ \ar[l] &\circ \ar[l] &\circ \ar[l] &\circ \ar[l]
}$
\end{tabular}\caption{Acyclic exceptional quivers of finite mutation type}\label{figureL1}
\end{figure}

\begin{figure}[H]
\begin{tabular}{ll}
$
\xymatrix@-0.6pc{
E_6^{(1,1)}
&& \circ \ar[dl]\ar[dr]\ar[drrr]\\
\circ \ar[r] &\bullet \ar[dr] && \circ \ar[dl]& \circ \ar[l]& \circ \ar[dlll]
& \circ\ar[l] \\
&& \circ \ar@<0.5ex>[uu]\ar@<-0.5ex>[uu]
}
$
&
$
\xymatrix@-1.4pc{
X_6 &&&
\circ \ar[dd]\\
& \circ \ar[rrd]&&&&\circ \ar@<0.5ex>[dd]\ar@<-0.5ex>[dd]\\
&&&\bullet \ar[rru]\ar[lld]&&\\
&\circ \ar@<0.5ex>[uu]\ar@<-0.5ex>[uu]&&&& \circ \ar[llu]\\
}
$
\\[2cm]
$
\xymatrix@-0.6pc{
E_7^{(1,1)}
&&& \circ \ar[dl]\ar[dr]\ar[drr]\\
\circ \ar[r]&\circ \ar[r] &\bullet \ar[dr] && \circ \ar[dl]& \circ \ar[dll]
& \circ\ar[l] & \circ\ar[l]\\
&&& \circ \ar@<0.5ex>[uu]\ar@<-0.5ex>[uu]
}
$
&
$
\xymatrix@-1.4pc{
X_7 &&
\circ\ar@<0.5ex>[rr]\ar@<-0.5ex>[rr]&& \circ \ar[ddl]\\
& \circ \ar[rrd]&&&&\circ \ar@<0.5ex>[dd]\ar@<-0.5ex>[dd]\\
&&&\bullet \ar[uul]\ar[rru]\ar[lld]&&\\
&\circ \ar@<0.5ex>[uu]\ar@<-0.5ex>[uu]&&&& \circ \ar[llu]\\
}
$
\\[2cm]
$
\xymatrix@-0.6pc{
E_8^{(1,1)}
&& \circ \ar[dl]\ar[dr]\ar[drr]\\
\circ \ar[r] &\bullet \ar[dr] && \circ \ar[dl]& \circ \ar[dll]
& \circ\ar[l] & \circ\ar[l]& \circ\ar[l] & \circ\ar[l]\\
&& \circ \ar@<0.5ex>[uu]\ar@<-0.5ex>[uu]
}
$
&
\end{tabular}\caption{Non-acyclic exceptional quivers of finite mutation type}\label{figureL2}
\end{figure}
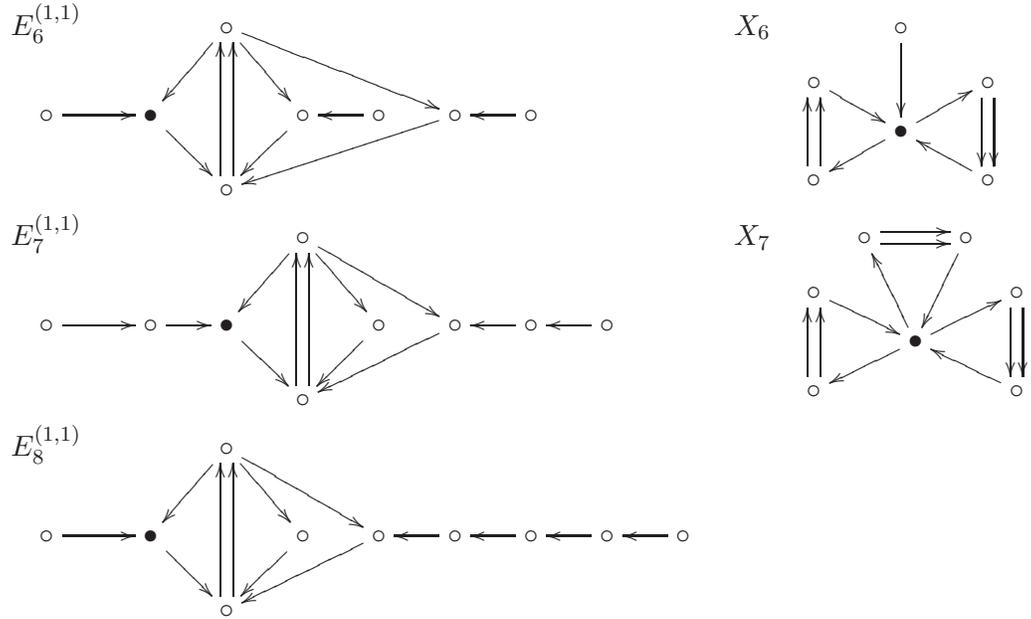

Based on the work of Fomin, Shapiro and
Thurston \cite{FST}, the quivers of
finite mutation type were classified by Felikson, Shapiro and
Tumarkin \cite{FeST}.

\begin{thm}[{\cite[Main Theorem]{FeST}}]
A $2$-acyclic quiver $Q$ is of finite mutation type if and only if one of
the following holds:
\begin{itemize}

\item[(i)]
$Q = Q(\tau)$ for some triangulation $\tau$ of some
marked surface $(\Sigma,\marked)$;

\item[(ii)]
$Q$ is mutation equivalent to
one of the quivers $X_6$, $X_7$,
$K_m$ for $m \ge 3$,
$E_m$, $\widetilde{E}_m$, $E_m^{(1,1)}$ for $m=6,7,8$.

\end{itemize}
\end{thm}

\subsection{Block decomposition of quivers arising from triangulations}\label{secblocks}
We recall some notation and results from \cite[Section~13]{FST}.
Each non-self-folded triangle of a triangulation $\tau$ of a marked
surface gives rise to a certain quiver called a \emph{block}.
The quiver $Q(\tau)$ associated to $\tau$ is then obtained by glueing these blocks in the obvious way (and by deleting possible 2-cycles).
For details we refer to \cite{FST}.
In Figure~\ref{blocks} we display the six different kinds of blocks arising from the six different kinds of non-self-folded triangles, see
\cite{FST} for details.
Note that some of the vertices of the blocks are white ($\circ$) and others are black ($\bullet$).
(Actually, there is yet another kind of non-self-folded triangle.
However, it occurs
only in one exceptional case.
Namely, take
the sphere with 4 punctures and its triangulation consisting of three self-folded triangles and the non-self-folded
triangle whose
sides are the three non-folded sides of the three self-folded triangles.)

\begin{figure}
\begin{tabular}{c|c|c|c|c|c}
I & II & IIIa & IIIb & IV & V
\\\hline
$
\xymatrix@+0.1pc{
\circ \ar[r] & \circ
}
$
&
$
\xymatrix@-1.9pc{
&& \circ \ar[dddrr]\\
&&&&\\
&&&&\\
\circ \ar[rruuu] &&&& \circ \ar[llll]
}
$
&
$
\xymatrix@-0.7pc{
& \circ \\
\bullet \ar[ur] && \bullet \ar[ul]
}
$
&
$
\xymatrix@-0.8pc{
& \circ\ar[dl]\ar[dr]\\
\bullet && \bullet
}
$
&
$
\xymatrix@-0.8pc{
& \bullet \ar[dr]\\
\circ \ar[ur]\ar[dr] && \circ \ar[ll]\\
& \bullet \ar[ur]
}
$
&
$
\xymatrix@-0.8pc{
\bullet \ar[dr] && \bullet \ar[ll]\ar[dd]\\
& \circ \ar[ur] \ar[dl]& \\
\bullet \ar[rr]\ar[uu] && \bullet \ar[ul]
}
$
\end{tabular}
\caption{List of blocks}\label{blocks}
\end{figure}
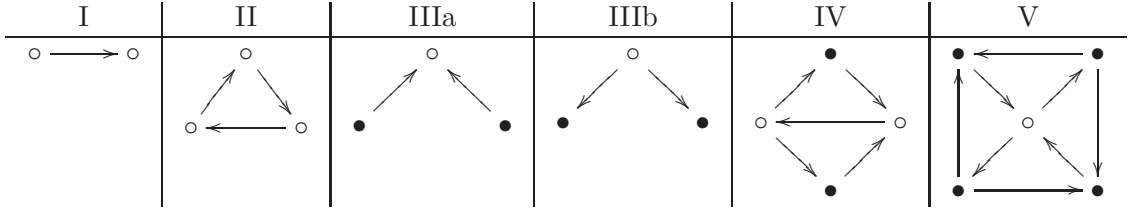

One can also start with a list of block quivers, and then
glue them (without referring to a triangulation) as follows:
Let $B_1,\ldots,B_t$ be a list of copies of blocks, and for
$1 \le k \le t$ let $\cW_k$ be the set of white vertices
of the quiver $B_k$.
Let $\cW = \cW_1 \cup \cdots \cup \cW_t$ be the disjoint union
of the sets $\cW_k$.
A \emph{glueing map} for $B_1,\ldots,B_t$ is a map $g\df \cW \to \cW$ such that
the following hold:
\begin{itemize}

\item[(gl1)]
$g^2 = {\rm id}_\cW$;

\item[(gl2)]
For $1 \le k \le t$, the restriction of $g$ to
$g(\cW_k) \cap \cW_k$ is the identity.

\end{itemize}
Let ${\rm pre.glue}(B_1,\ldots,B_t;g)$
be the quiver obtained from the disjoint union of the
quivers $B_1,\ldots,B_t$ by identifying
the vertices $i$ and $g(i)$, where $i$ runs through $\cW$.
Next, let ${\rm glue}(B_1,\ldots,B_t;g)$ be the quiver obtained from
${\rm pre.glue}(B_1,\ldots,B_t;g)$ by removing $2$-cycles in the
usual sense.
Following Fomin, Shapiro and Thurston \cite{FST},
we say that a connected quiver $Q$ is \emph{block decomposable}
if $Q = {\rm glue}(B_1,\ldots,B_t;g)$ for some $B_1,\ldots,B_t$
and $g$ as above.
The trivial quiver with just one vertex and no arrows is also called
\emph{block decomposable}.

\begin{thm}[{\cite[Theorem~13.3]{FST}}]
For a quiver $Q$ the following are equivalent:
\begin{itemize}

\item[(i)]
$Q$ is block decomposable;

\item[(ii)]
$Q = Q(\tau)$ for some triangulation $\tau$ of some marked surface.

\end{itemize}
\end{thm}


\section{Tame algebras and deformations of Jacobian algebras}
\label{sec7}


\subsection{}
One of our aims is to determine which
Jacobian algebras $\cP(Q,S)$ with $(Q,S)$ a non-degenerate QP
are of tame representation type.
We will show that for most quivers $Q$ of finite mutation type,
$\cP(Q,S)$ is a skewed-gentle algebra or
a deformation of a skewed-gentle algebra.

\subsection{Deformations of algebras}
The following theorem due to Crawley-Boevey \cite{CB4}
is a slightly more general version of a theorem by Gei{\ss}~\cite{G}.
It is a crucial tool for the proof of our first main result 
Theorem~\ref{thmtamewild}.

\begin{thm}[{\cite[Theorem B]{CB4}}]\label{thmdeformation}
Let $A$ be a finite-dimensional $\ka$-algebra, and let $X$ be an irreducible 
affine variety over $\C$, and let $f_1,\ldots,f_r\df X \to A$
be morphisms of affine varieties (where $A$ has its natural structure
as an affine space). For $x \in X$ set
\[
A(x) := A/(f_1(x),\ldots,f_r(x)).
\]
Let $x_0,x_1 \in X$ be such that the following hold:
\begin{itemize}

\item[(i)]
$A(x_0)$ is tame;

\item[(ii)]
$A(x) \cong A(x_1)$ for all $x$ in some dense open subset of $X$.
\end{itemize}
Then $A(x_1)$ is tame.
\end{thm}

In most applications, the variety $X$ appearing in Theorem~\ref{thmdeformation} will be just the affine line $\C$.
There are several notions of deformations
of algebras.
In this article, we say that an algebra $B$ is a \emph{deformation} of
an algebra $A$ if in the situation of
Theorem~\ref{thmdeformation} we have $A \cong A(x_0)$ and $B \cong A(x_1)$.

\subsection{Deformations of Jacobian algebras}\label{secdeformjac}
As before, let $Q$ be a 2-acyclic quiver. Let
\[
S = \sum_{t \ge 3}\sum_{w_t} \mu_{w_t}w_t
\]
be a potential on $Q$, where $\mu_{w_t} \in \ka$ and the
$w_t$ run over all cycles of length $t$ in $Q$.
Recall that
\[
\short(S) = \min\{ t \ge 3 \mid \mu_{w_t} \not= 0
\text{ for some } w_t \},
\]
and define
\[
\Smin := \sum_{t=\short(S)} \sum_{w_t} \mu_{w_t}w_t.
\]

Let $I(S)$ be the ideal of $\CQ{Q}$
generated by the cyclic derivatives $\partial_\alp(S)$ of $S$.
Furthermore, for $p \ge 3$ let
\[
S_p :=  \sum_{3 \le t \le p} \sum_{w_t} \mu_{w_t}w_t
\]
be the $p$-\emph{truncation} of $S$. We also set $S_2 := 0$.
The following lemma is straightforward.

\begin{lemma}\label{deform1}
For $p \ge 2$ we have
$I(S_p) + \m^p = I(S) + \m^p$.
\end{lemma}

For $\lambda \in \ka^*$ define
\[
S(\lambda) := \sum_{t \ge 3}\sum_{w_t} \lambda^{t-\short(S)}\mu_{w_t}w_t,
\]
and set $S(0) := \Smin$.
Let $S(\lambda)_p$ denote the $p$-truncation of the potential $S(\lambda)$.

We define an algebra homomorphism
\[
f_\lambda\df \CQ{Q} \to \CQ{Q}
\]
by $\alp \mapsto \lambda \alp$ for all $\alp \in Q_1$.
Clearly, $f_\lambda$ is an isomorphism for all $\lambda \not= 0$.

\begin{lemma}\label{deform2}
For all $\alp \in Q_1$ and $\lam \not= 0$ we have
\[
f_\lambda(\partial_\alp(S)) = \sum_{t \ge 3} \sum_{w_t}
\lambda^{t-1}\mu_{w_t}\partial_\alp(w_t)
\quad\text{ and }\quad
\partial_\alp(S(\lambda)) = \lambda^{1-\short(S)}f_\lambda(\partial_\alp(S)).
\]
\end{lemma}

\begin{proof}
The first equality follows from the definition of $f_\lam$
and the fact that $\partial_\alp(w_t)$ is a linear
combination of paths of length $t-1$.
We have
\begin{align*}
\lam^{1-\short(S)}f_\lam(\partial_\alp(S)) &=
\lam^{1-\short(S)} \sum_{t \ge 3} \sum_{w_t} \lam^{t-1}
\mu_{w_t}\partial_\alp(w_t)\\ &=
\sum_{t \ge 3} \sum_{w_t} \lam^{t-\short(S)} \mu_{w_t} \partial_\alp(w_t)\\
&= \partial_\alp(S(\lam)).
\end{align*}
Thus the second equality also holds.
\end{proof}

\begin{coro}\label{deform3}
For $\lam \not= 0$ we have
$f_\lambda(I(S)) = I(S(\lambda))$.
\end{coro}

Define $\LL := \cP(Q,S)$ and $\LL(\lambda) := \cP(Q,S(\lambda))$.
Note that $\LL = \LL(1)$. For $p \ge 2$ we have
\[
\LL(\lambda)_p = \CQ{Q}/(I(S(\lambda))+\m^p) =
\CQ{Q}/(I(S(\lambda)_p)+\m^p)
\]
and 
$\CQ{Q}_p = \CQ{Q}/\m^p$.

\begin{lemma}\label{deform4}
For $\lambda \not= 0$ we have
$\LL \cong \LL(\lambda)$ and $\LL_p \cong \LL(\lambda)_p$.
\end{lemma}

\begin{proof}
For $\lambda \not= 0$ we get $f_\lambda(\m^p) = \m^p$.
Now the lemma follows from Corollary~\ref{deform3}.
\end{proof}

\begin{prop}\label{deformcriterion1}
If $\LL(0)_p$ is tame, then $\LL_p$ is also tame.
\end{prop}

\begin{proof}
For each $\alp \in Q_1$ and $p \ge 2$ define a map
\[
d_{\alp,p}\df \C \to \CQ{Q}_p
\]
by
$\lambda \mapsto \partial_\alp(S(\lambda)_p) + \m^p$.
The vector spaces $\C$ and
$\CQ{Q}_p$ can be seen as affine spaces, and $d_{\alp,p}$
is then obviously a morphism of affine varieties.

For each $\lambda \in \ka$ there is a canonical isomorphism
\[
\LL(\lambda)_p \cong \CQ{Q}_p/I(\lambda)_p,
\]
where $I(\lambda)_p$ is the ideal of $\CQ{Q}_p$ generated
by the elements $d_{\alp,p}(\lambda)$, where $\alp$ runs through $Q_1$.
Using Lemma~\ref{deform4}, the result follows now from Theorem~\ref{thmdeformation}.
\end{proof}

A basic algebra $\LL$ is tame if and only if all $p$-truncations
$\LL_p$ are tame.
This follows from the fact that
\[
\md(\LL) = \bigcup_{p \ge 2} \md(\LL_p).
\]
Thus the following tameness criterion for Jacobian algebras
follows directly from Proposition~\ref{deformcriterion1}.

\begin{coro}\label{deformcriterion2}
Let $Q$ be a $2$-acyclic quiver, and let $S$ be a potential
on $Q$.
If $\cP(Q,\Smin)$ is tame, then $\cP(Q,S)$ is also tame.
\end{coro}

\subsection{Gentle and skewed-gentle algebras}\label{defgentle}
We repeat the definition of a gentle and a skewed-gentle algebra.
These are typical classes of tame algebras.

Let $Q = (Q_0,Q_1,s,t)$ be a quiver, and let $\LL$ be an algebra
isomorphic to an algebra of the form $\CQ{Q}/I$ or $\kQ{Q}/I$.
Following~\cite{AS} we call
$\LL$ a \emph{gentle algebra} if the following hold:
\begin{itemize}

\item[(g1)]
For each $k \in Q_0$ we have $\abs{\{ \alp \in Q_1 \mid
s(\alp) = k \}} \le 2$, and
 $\abs{\{ \alp \in Q_1 \mid
t(\alp) = k \}} \le 2$;

\item[(g2)]
The ideal $I$ is generated by a set of paths of length $2$;

\item[(g3)]
If for some arrow $\bet\in Q_1$ there exist two paths
$\alp_1\bet \neq \alp_2\bet$ of length $2$, then precisely one of these paths belongs
to $I$;

\item[(g4)]
If for some arrow $\bet\in Q_1$ there exist two paths
$\bet\gam_1 \neq \bet\gam_2$ of length $2$, then precisely one of these paths belongs
to $I$.

\end{itemize}
Note that our definition of a gentle algebra is slightly more general
than the original definition from \cite{AS}.

Let $Q = (Q_0,Q_1,s,t)$ be a quiver, and let $\LL$ be an algebra
isomorphic to an algebra of the form $\CQ{Q}/I$ or $\kQ{Q}/I$.
Following \cite[Definition~4.2]{GP} we call
$\LL$ a \emph{skewed-gentle algebra} if there is a set
\[
L \subseteq \{ \alp \in Q_1 \mid s(\alp) = t(\alp) \}
\]
of loops of $Q$ such that the following hold:
\begin{itemize}

\item[(sg1)]
For each $k \in Q_0$ we have $\abs{\{ \alp \in Q_1 \mid
s(\alp) = k \}} \le 2$, and
 $\abs{\{ \alp \in Q_1 \mid
t(\alp) = k \}} \le 2$;

\item[(sg2)]
The ideal $I$ is generated by a set of paths of length $2$, and
by all elements of the form $\alp^2 - \alp$ with $\alp \in L$;

\item[(sg3)]
If for some arrow $\bet \in Q_1 \setminus L$
there exist two paths
$\alp_1\bet \neq \alp_2\bet$ of length $2$, then precisely one of these paths belongs
to $I$;

\item[(sg4)]
If for some arrow $\bet \in Q_1 \setminus L$ there exist two paths
$\bet\gam_1 \neq \bet\gam_2$ of length $2$, then precisely one of these paths belongs
to $I$.

\end{itemize}

On the first glance, skewed-gentle algebras do not seem to be basic
algebras, since the elements $\alp^2-\alp$ with $\alp \in L$
are not contained in $\m^2$.
However, it is not difficult to show that skewed-gentle algebras are indeed basic algebras.
Note also that
every gentle algebra is a skewed-gentle algebra.

\begin{thm}
Skewed-gentle algebras are tame.
\end{thm}

\begin{proof}
Skewed-gentle algebras belong to the class
of clannish algebras introduced by Crawley-Boevey \cite{CB2}.
The tameness of clannish algebras follows from \cite[Theorem~3.8]{CB2}.
\end{proof}

\subsection{Examples}

\subsubsection{}\label{ex:5.5.1}
Let $H$ be the quiver
\[
\xymatrix{1' \ar[dr]^{\del'} & & & & 5' \\
 & 2 \ar[r]^{\gam} & 3 \ar[r]^{\bet} & 4 \ar[ur]^{\alp'} \ar[dr]_{\alp''} & \\
 1'' \ar[ur]_{\del''} & & & & 5''}
\]
 of affine type $\widetilde{\D}_6$.
Its path algebra $\kQ{H}$ is isomorphic to $\kQ{Q}/I$,
where $Q$ is the quiver
\[
\xymatrix{1 \ar@(ul,dl)[]_{\varepsilon_1} \ar[r]^{\del} & 2 \ar[r]^{\gam} & 3 \ar[r]^{\bet} & 4 \ar[r]^{\alp} & 5 \ar@(ur,dr)[]^{\varepsilon_5}}
\]
and the ideal $I$ is generated by  $\varepsilon_1^2-\varepsilon_1$
and $\varepsilon_5^2-\varepsilon_5$.
It follows that
$\kQ{H}$ is a skewed-gentle algebra.

\subsubsection{}\label{block4clannish}
Let $H$ be the quiver
\[
\xymatrix{
& 2' \ar[dr]^{\alp_1}\\
1 \ar[ur]^{\bet_1}\ar[dr]_{\bet_2} && 3 \ar[ll]_{\gam}\\
& 2'' \ar[ur]_{\alp_2}
}
\]
and let $S := \alp_1\bet_1\gam + \alp_2\bet_2\gam$.
Thus the ideal $I(S)$ of $\CQ{H}$ is generated by the
elements
$\bet_i\gam$, $\gam\alp_i$ for $i=1,2$ and
$\alp_1\bet_1 + \alp_2\bet_2$.
One easily checks that $\cP(H,S) = \CQ{H}/J(S) = \kQ{H}/I(S)$.
(The first equality holds by the definition of a Jacobian algebra.)
The Jacobian algebra $\cP(H,S)$ is
isomorphic to the algebra $\kQ{Q}/I$, where $Q$ is the quiver
\[
\xymatrix{
&&\\
& 2 \ar[dr]^\alp \ar@(ul,ur)[]^{\varepsilon} \\
1 \ar[ur]^\bet && 3 \ar[ll]_{\gam}
}
\]
and $I$ is generated by the elements
$\alp\bet$, $\bet\gam$, $\gam\alp$ and $\varepsilon^2-\varepsilon$.
Thus $\cP(H,S)$ is a skewed-gentle algebra.

\subsubsection{}
Let $Q$ be the quiver
\[
\xymatrix@-1.2pc{
&&2 \ar@<0.4ex>[dddrr]^{\alp_1}\ar@<-0.4ex>[dddrr]_{\alp_2}\\
&&&&\\
&&&&\\
1 \ar@<0.4ex>[rruuu]^{\bet_1}\ar@<-0.4ex>[rruuu]_{\bet_2} &&&&3 \ar@<0.4ex>[llll]^{\gam_1}\ar@<-0.4ex>[llll]_{\gam_2}
}
\]
and let
\[
S := \alp_1\bet_1\gam_1 + \alp_2\bet_2\gam_2 - \alp_1\bet_2\gam_1\alp_2\bet_1\gam_2.
\]
Using the notation from Section~\ref{secdeformjac} we get
\[
S(\lambda) := \alp_1\bet_1\gam_1 + \alp_2\bet_2\gam_2 - \lambda^3 \alp_1\bet_2\gam_1\alp_2\bet_1\gam_2
\]
for $\lam \in \ka$.
By Lemma~\ref{deform4} we have
$\cP(Q,S) \cong \cP(Q,S(\lambda))$ for all $\lambda \not= 0$.
The algebra $\cP(Q,S(0))_p$ is isomorphic to the path algebra $\kQ{Q}$
modulo the ideal generated by
$\alp_1\bet_1$, $\bet_1\gam_1$, $\gam_1\alp_1$,
$\alp_2\bet_2$, $\bet_2\gam_2$, $\gam_2\alp_2$
and all paths of length $p$.
Thus $\cP(Q,S(0))_p$ is a factor algebra of a gentle algebra.
It follows that $\cP(Q,S(0))_p$ is tame for all $p$.
Thus $\cP(Q,S(0)) = \cP(Q,\Smin)$ is tame.
Now Corollary~\ref{deformcriterion2} implies that
$\cP(Q,S)$ is tame as well.

\subsection{Jacobian algebras as deformations of skewed-gentle algebras}
We say that a quiver $Q$ \emph{has no double arrows} provided the
number of arrows from $i$ to $j$ is at most one for all vertices
$i$ and $j$ of $Q$.

\begin{prop}\label{proptamecrit}
Let $Q = {\rm glue}(B_1,\ldots,B_t;g)$ be a block decomposable
quiver.
Assume that
the following
hold:
\begin{itemize}

\item[(gl3)]
For all $i \not= j$ we have $\abs{g(\cW_i) \cap \cW_j} \le 1$;

\item[(gl4)]
Each $3$-cycle in $Q$ is contained in one of the blocks
$B_1,\ldots,B_t$;

\item[(gl5)]
None of the blocks $B_1,\ldots,B_t$ is of type V.

\end{itemize}
Let $S$ be any non-degenerate potential on $Q$.
Then the
Jacobian algebra $\cP(Q,S)$ is a deformation of a skewed-gentle algebra.
In particular, $\cP(Q,S)$ is tame.
Assume additionally that
$Q = {\rm glue}(B_1,\ldots,B_t;g)$ satisfies the following:
\begin{itemize}

\item[(gl6)]
Each of the blocks $B_1,\ldots,B_t$ is of type I or II.

\end{itemize}
Then $\cP(Q,S)$ is a deformation of a gentle algebra.
\end{prop}

\begin{proof}
Assume that (gl3), (gl4) and (gl5) are satisfied.
Note that condition (gl3) in Proposition~\ref{proptamecrit}
implies that $Q$ does not have any double arrows, and that
${\rm pre.glue}(B_1,\ldots,B_t;g)$ does not have any $2$-cycles.
Now let $S$ be a non-degenerate potential on $Q$, and
let $\Smin$ be defined as in Section~\ref{secdeformjac}.
By Corollary~\ref{prop3cycle} we know that $\Smin$ contains all 3-cycles
of $Q$ up to rotation.
Now one easily checks that the algebra $\cP(Q,\Smin)$
is a skewed-gentle algebra.
(Blocks of type IIIa, IIIb or IV are handled in the same way as in
Examples~\ref{ex:5.5.1} and \ref{block4clannish}.)
In particular, $\cP(Q,\Smin)$ is tame.
By Corollary~\ref{deformcriterion2} this implies that $\cP(Q,S)$
is tame.
Similarly, if we assume additionally (gl6), it follows easily that
$\cP(Q,\Smin)$ is a gentle algebra.
\end{proof}

\subsection{Definition of gentle and skewed-gentle triangulations}
We will show in Section~\ref{sec11}
that for most marked surfaces $(\Sigma,\marked)$
there is a triangulation $\tau$ such that
$Q = Q(\tau) = {\rm glue}(B_1,\ldots,B_t;g)$
satisfies the assumptions (gl3), (gl4), (gl5) and (gl6) of Proposition~\ref{proptamecrit}.
In this case, we call $\tau$ a \emph{gentle triangulation}.
If $Q(\tau)$ satisfies (gl3), (gl4) and (gl5), then $\tau$ is a
\emph{skewed-gentle triangulation}.
If $Q(\tau)$ is the trivial quiver (this only happens if $\tau$ is
one of the two triangulations of an unpunctured 4-gon), then
$\tau$ is also called a gentle or skewed-gentle triangulation.

Note that our definition of a gentle or skewed-gentle triangulation is
quite restrictive.
For example, condition (gl3) implies that the quiver $Q(\tau)$ of a skewed-gentle
triangulation $\tau$ has no double arrows.
Let $\tau$ be a triangulation, set $Q := Q(\tau)$, and
let $S$ be a non-degenerate potential on $Q$.
If $\tau$ is a gentle triangulation, we get that
$\cP(Q,\Smin)$ is a gentle algebra,
but $\cP(Q,S)$ is in general not gentle.
It can also happen that both $\cP(Q,S)$ and $\cP(Q,\Smin)$
are gentle algebras, even though $\tau$ is not a gentle triangulation.
For skewed-gentle triangulations the situation is
analogous.


\section{The representation type of Jacobian algebras: Regular cases}\label{sec11}


\subsection{Representation type and mutation type}
Let
$Q$ be a 2-acyclic quiver.
Recall that $Q$ is of \emph{finite cluster type} if the
associated cluster algebra $\cA_Q$ has only finitely many
cluster variables.
Apart from a few exceptional cases, which are treated separately
in Section~\ref{sec11b}, the following theorem will be proved in this
section.
The proof relies heavily on the construction of suitable triangulations of marked surfaces, which can be found in Sections~\ref{secgentle} and \ref{secskewedgentle}.

\begin{thm}\label{thmtamewild2}
Assume that $Q$ is not mutation equivalent to one of the quivers
$T_1$, $T_2$, $X_6$, $X_7$ or $K_m$ with $m \ge 3$.
Then for any non-degenerate potential $S$ on $Q$ the following hold:
\begin{itemize}

\item[(f)]
$\cP(Q,S)$ is representation-finite if and only if
$Q$ is of finite cluster type.

\item[(t)]
$\cP(Q,S)$ is tame if and only if
$Q$ is of finite mutation type.

\item[(w)]
$\cP(Q,S)$ is wild if and only if
$Q$ is of infinite mutation type.

\end{itemize}
\end{thm}

The quivers $T_1$, $T_2$, $X_6$, $X_7$ and $K_m$ will also be
studied in Section~\ref{sec11b}.

\subsection{Reduction to quivers of finite mutation type}

\begin{lemma}\label{tamewild1}
Part (f) of Theorem~\ref{thmtamewild2} is true.
\end{lemma}

\begin{proof}
As a consequence of
\cite[Corollary~5.3]{DWZ2} there is an injective map from the set of cluster
variables of any cluster algebra $\cA_Q$ to the set of
isomorphism classes of indecomposable decorated representations
of $\cP(Q,S)$.
Thus if $\cP(Q,S)$ is representation-finite, then
$Q$ is of finite cluster type.
Now suppose that $Q$ is of finite cluster type.
Then it follows from \cite[Theorem~1.4]{FZ2} that
$Q$ is mutation equivalent to a Dynkin quiver.
Let $\LL = \cP(Q,S)$ be a Jacobian algebra, and let
$\LL' = \cP(Q',S')$, where $(Q',S') = \mu_k(Q,S)$ for
some vertex $k$ of $Q$.
Mutation yields a bijection between the sets of isomorphism
class of decorated representations of $\LL$ and $\LL'$, respectively.
The proof is implicitely contained in \cite{DWZ1}.
We refer to \cite[Section~7]{BIRSm}, where it is shown that $\LL$ and
$\LL'$ are nearly Morita equivalent, see Section~\ref{secnearmorita}.
In particular, it is known
that mutations of indecomposable decorated representations are
again indecomposable, see \cite[Corollary~10.14]{DWZ1}.
Since Dynkin quivers are representation-finite,
the result follows.
\end{proof}

\begin{lemma}\label{tamewild2}
Assume that $Q$ is of infinite mutation type.
Then $\cP(Q,S)$ is wild for all non-degenerate potentials $S$
on $Q$.
\end{lemma}

\begin{proof}
Let $S$ be a non-degenerate potential on $Q$.
Since $Q$ is of infinite mutation type, there exists
some $(Q',S')$ such that $(Q,S)$ is QP-mutation equivalent to $(Q',S')$
and $Q'$ contains two vertices $i$ and $j$ with at least three arrows
from $i$ to $j$.
This yields an exact embedding $\md(\C K_3) \to \md(\cP(Q',S'))$.
Since the path algebra $\C K_3$ is wild, we get that $\cP(Q',S')$
is wild.
By Theorem~\ref{thmmutationinv} this implies that $\cP(Q,S)$ is
wild.
\end{proof}

\begin{coro}\label{tamewild3}
If $\cP(Q,S)$ is tame for some non-degenerate potential $S$ on $Q$,
then $Q$ is of finite mutation type.
\end{coro}

\begin{proof}
Combine Lemma~\ref{tamewild2} and Drozd's Theorem~\ref{thmdrozd}.
\end{proof}

\begin{lemma}\label{tamewild4}
Suppose that part (t) of Theorem~\ref{thmtamewild2} is true.
Then (w) is also true.
\end{lemma}

\begin{proof}
This follows directly from Drozd's Theorem~\ref{thmdrozd}.
\end{proof}

To prove Theorem~\ref{thmtamewild2}
it remains to study $\cP(Q,S)$ for $Q$ of finite mutation type
and decide when $\cP(Q,S)$ is tame.

\subsection{Proof of Theorem~\ref{thmtamewild2}}
The following lemma takes care of a large part of the
proof of
Theorem~\ref{thmtamewild2}.

\begin{lemma}
Let $\surf$ be a marked surface which is not equal to a torus
with $|\marked| = 1$, or to a sphere with $|\marked| = 4$.
Let $Q = Q(\tau)$ for some triangulation $\tau$ of $\surf$.
Then $\cP(Q,S)$ is tame for all non-degenerate potentials
$S$ on $Q$.
\end{lemma}

\begin{proof}
If $\surf$ is a monogon with $|\punct| = 2$, then
there is a triangulation $\sigma$ of $\surf$ such that $Q(\sigma)$ is
an acyclic quiver of type $\widetilde{A}_3$.
There is only one potential $S$ on $Q(\sigma)$, namely $S=0$, and
$\cP(Q(\sigma),S)$ is tame.
In case $\surf$ is a sphere with $|\marked|=5$, a monogon with 
$|\punct| \ge 3$, a digon, a triangle or an annulus with 
$|\marked \setminus \punct| = 2$,
there is a skewed-gentle triangulation $\sigma$ of $\surf$ by
Theorem~\ref{uniquehelp2} and Sections~\ref{sphere5} and
\ref{sec6.4.1}.
In all remaining cases, there is a gentle triangulation $\sigma$ of $\surf$
by Theorem~\ref{thmgoodtriang}.
Thus by Proposition~\ref{proptamecrit}, in each of the above cases, the
Jacobian algebra $\cP(Q(\sigma),W)$ is a tame for any non-degenerate
potential $W$ on $Q(\sigma)$.

By \cite[Proposition~4.8]{FST}, $Q(\tau)$ and $Q(\sigma)$ are
mutation equivalent quivers.
Since the potential $S$ is non-degenerate, we see that $(Q(\tau),S)$
is QP-mutation equivalent to $(Q(\sigma),W)$ for some
non-degenerate potential $W$
on $Q(\sigma)$. Thus $\cP(Q(\sigma),W)$ is tame.
Now Theorem~\ref{thmmutationinv} implies that $\cP(Q(\tau),S)$ is tame.
\end{proof}

To finish the proof of Theorem~\ref{thmtamewild2}, we need
the following lemma, which will be proved in Section~\ref{sec11b}.

\begin{lemma}
For the following quivers $Q$ and all non-degenerate potentials $S$ on
$Q$, the Jacobian algebra $\cP(Q,S)$
is tame:
\begin{itemize}

\item[({E})]
$Q$ is mutation equivalent to one of the quivers $E_m$,
$\widetilde{E}_m$ or $E_m^{(1,1)}$ for $m=6,7,8$;

\item[({sp4})]
$Q = Q(\tau)$ for some triangulation $\tau$ of a sphere $(\Sigma,\marked)$ with empty boundary and $|\marked| = 4$.

\end{itemize}
\end{lemma}

\subsection{Gentle triangulations}\label{secgentle}

\begin{thm}\label{thmgoodtriang}
Let $\surf$ be a marked surface which is not equal to
one of the following:
\begin{itemize}

\item
a monogon, a digon, or a triangle;

\item
an annulus with $|\marked \setminus \punct| = 2$;

\item
a sphere with $|\marked| = 4,5$;

\item
a torus with $|\marked| = 1$.

\end{itemize}
Then there exists a gentle triangulation $\tau$ of $\surf$.
\end{thm}

We prove
Theorem~\ref{thmgoodtriang} by induction.
The rather lengthy induction base is dealt with in Section~\ref{indbase}, and the induction
step is performed in Section~\ref{indstep}.
In the course of the proof,
we draw several triangulations $\tau$ of marked surfaces
and their adjacency quivers $Q(\tau)$.
The shaded regions in the quivers $Q(\tau)$ correspond to the
interior non-self-folded triangles of $\tau$.
Vertices with the same label have to be identified.

\subsection{Proof of Theorem~\ref{thmgoodtriang}:
Induction base}\label{indbase}

\subsubsection{Unpunctured $4$-gon}
An unpunctured $4$-gon has only two triangulations, and
their adjacency quiver is the trivial quiver with one vertex
and no arrows.
By definition these triangulations are gentle.

\subsubsection{Unpunctured annulus $\surf$ with $|\marked| = 3$}
In Figure~\ref{Fig:new_type_A_2_1.eps} we show a gentle triangulation
$\tau$ and its quiver $Q(\tau)$ for an unpunctured annulus
$\surf$ with $|\marked|=3$.
\begin{figure}[!htb]
                \centering
                \includegraphics[scale=.7]{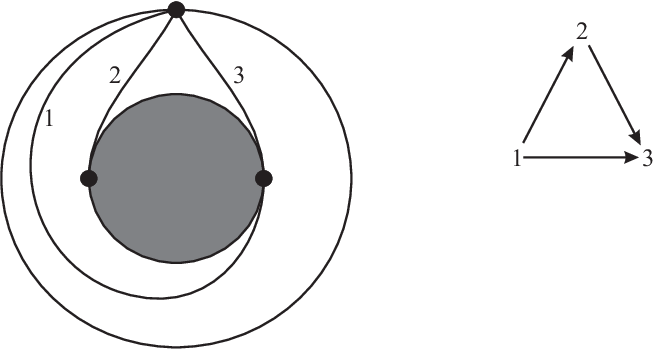}
\caption{
A gentle triangulation for an unpunctured annulus $\surf$ with
$|\marked| = 3$.}
\label{Fig:new_type_A_2_1.eps}
\end{figure}
%

\subsubsection{Unpunctured surface $\surf$ with genus $0$, three boundary components and $|\marked| = 3$}
As shown in Figure~\ref{Fig:new_sphere_3boundaries.eps}, an unpunctured surface $\surf$ with genus $0$, three boundary components and $|\marked| = 3$ has a gentle
triangulation.
\begin{figure}[!htb]
                \centering
                \includegraphics[scale=.7]{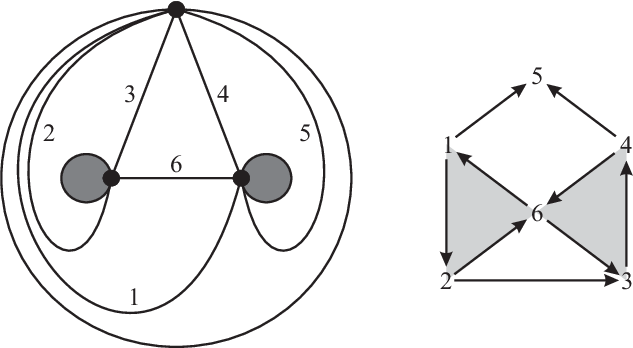}
\caption{
A gentle triangulation for an unpunctured surface $\surf$ with genus $0$, three boundary components and $|\marked| = 3$.}
\label{Fig:new_sphere_3boundaries.eps}
\end{figure}
%

\subsubsection{Sphere $\surf$ with $|\marked| = 6$}
Each sphere $\surf$ with $|\marked| = 6$ has a gentle triangulation
$\tau$.
Figure~\ref{Fig:new_sphere_6puncts.eps} illustrates the construction
of such a $\tau$ and its quiver $Q(\tau)$.
(The punctures are denoted by $p_1,\ldots,p_6$.)
\begin{figure}[!htb]
                \centering
                \includegraphics[scale=.7]{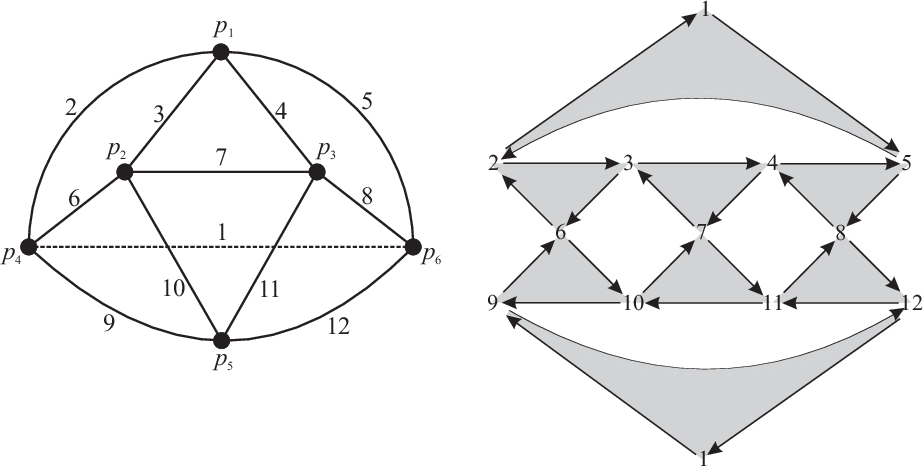}
\caption{
A gentle triangulation for a sphere $\surf$ with $|\marked| = 6$.}
\label{Fig:new_sphere_6puncts.eps}
\end{figure}
%

\subsubsection{Torus $\surf$ with empty boundary and
$|\marked| =2$}
Each torus $\surf$ with empty boundary and $|\marked| = 2$ has a gentle triangulation $\tau$.
Figure~\ref{Fig:2_punct_torus_no_double_arrows.eps} illustrates the construction
of such a $\tau$ and its quiver $Q(\tau)$.
(In the picture, in order to obtain a torus one has to identify the two sides labeled by $1$ and
also
the two sides labeled by $2$
along the indicated orientations.
The two punctures of the torus are denoted by $p$ and $q$.)
\begin{figure}[!htb]
\centering
\includegraphics[scale=.7]{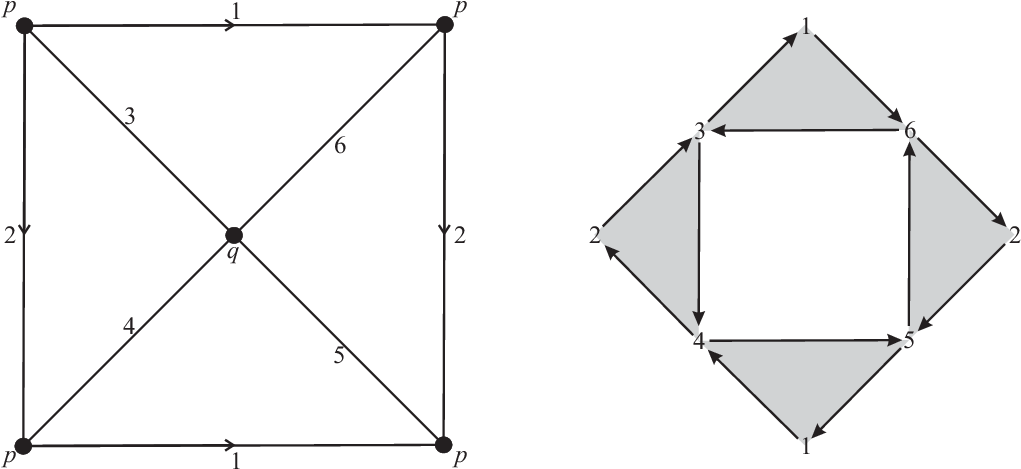}
\caption{
A gentle triangulation for a torus
$(\Sigma,\marked)$ with empty boundary and $|\marked| = 2$.}
\label{Fig:2_punct_torus_no_double_arrows.eps}
\end{figure}
%

\subsubsection{Torus $\surf$ with $|\marked|=2$ and
$|\marked \setminus \punct| = 1$}
In the case of a torus with one boundary component, one
marked point $p$ on this component, and one puncture $q$, the triangulation shown
in Figure~\ref{Fig:torus_1bound_1punct} is a gentle triangulation.
\begin{figure}[!htb]
                \centering
                \includegraphics[scale=.7]{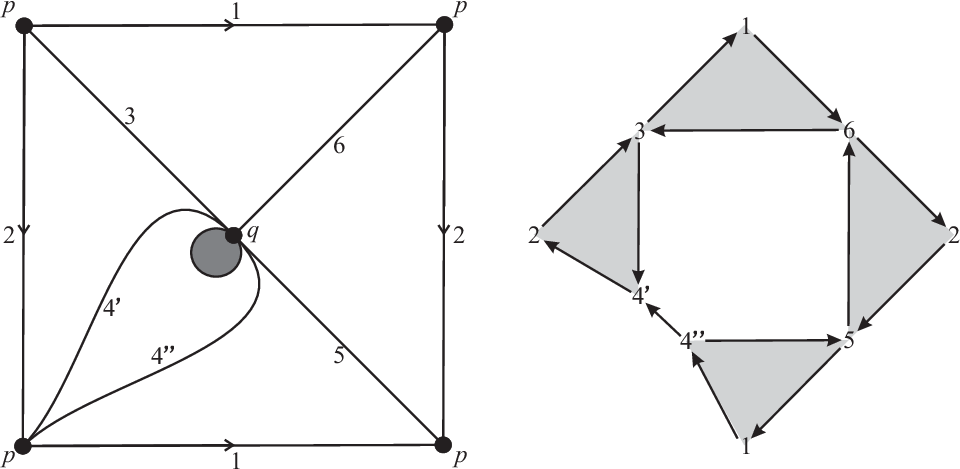}
\caption{
A gentle triangulation for a torus with one boundary component, one marked point on it, and one puncture.}
\label{Fig:torus_1bound_1punct}
\end{figure}
%

\subsubsection{Torus $\surf$ with two boundary components
and $|\marked| = 2$}
In the case of a torus with two boundary components, and
one marked point on each of these components, the triangulation
shown in
Figure~\ref{Fig:torus_2bound} is gentle.
\begin{figure}[!htb]
                \centering
                \includegraphics[scale=.7]{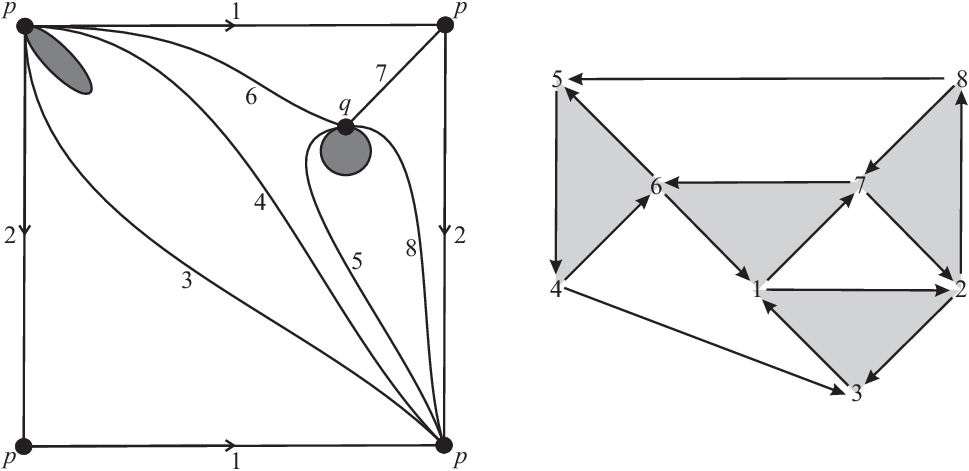}
\caption{A gentle triangulation for a torus with two boundary components and one marked point on each of these.}
\label{Fig:torus_2bound}
\end{figure}
%

\subsubsection{Unpunctured torus $\surf$ with one boundary component and $|\marked|=2$}
In the case of an unpunctured torus with one boundary component and
two marked points $p$ and $q$ on this component, the triangulation displayed in
Figure~\ref{Fig:torus_1bound_2points} is gentle.
\begin{figure}[!htb]
                \centering
                \includegraphics[scale=.7]{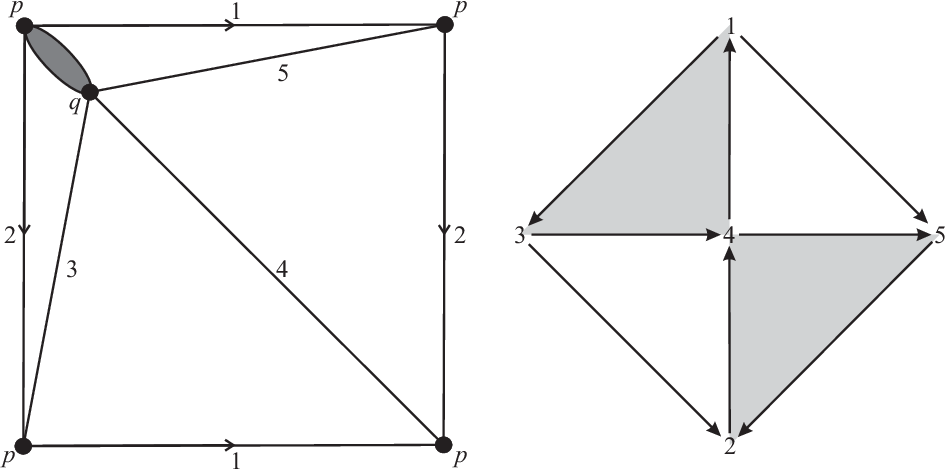}
\caption{A gentle triangulation for an unpunctured torus with one boundary component and two marked points on it.}
\label{Fig:torus_1bound_2points}
\end{figure}
%

\subsubsection{Surfaces $\surf$ with empty boundary, $|\marked| = 1$ and genus at least two}\label{sec6.2.9}
Let $g \ge 2$.
Let $P_g$ be a $4g$-gon with vertices $q_1,\ldots,q_{4g}$
ordered clockwise along the boundary of $P_g$ as shown in
Figure~\ref{Fig:2g_gon}.
For each $1 \le i \le 4g$ we glue the two sides labeled by $i$
along the orientation shown in Figure~\ref{Fig:2g_gon}.
This yields a surface $\Sigma$ of genus $g$.
We denote the corresponding identification map by
$\pi\df P_g\to\Sigma$.
\begin{figure}[!htb]
                \centering
                \includegraphics[scale=.7]{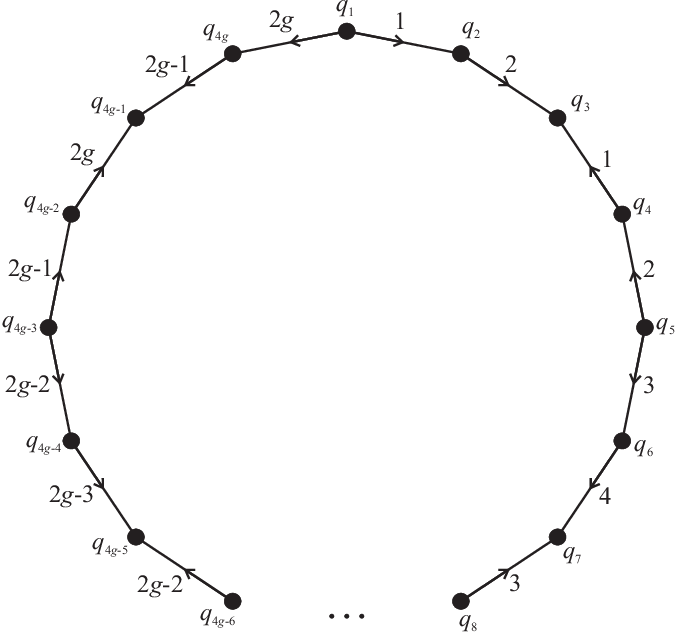}
\caption{
A $4g$-gon $P_g$ with a side pairing that yields a once-punctured surface $\Sigma$ with empty boundary and genus $g$.}
\label{Fig:2g_gon}
\end{figure}
For each $1 \le t \le 2g$ let $t'$ be the diagonal of $P_g$
that connects $q_{2t-2}$ with $q_{2t}$. (We set $q_0 := q_{4g}$.)
The set $\{ t' \mid 1 \le t \le 2g \}$ forms a $2g$-gon inscribed in $P_g$.
By inserting arcs $t''$ from $q_2$ to $q_{2t+2}$ for $2 \le t \le 2g-2$,
as illustrated in Figure~\ref{Fig:43-gon_nobound_1punct.eps} for $g=3$, we
obtain a triangulation $T$ of this $2g$-gon.
Let  $s_1,\ldots,s_{4g}$ be the sides of $P_g$, where $s_k$ connects $q_k$ with 
$q_{k+1}$. (We set $q_{4g+1} := q_1$.)
Then the set
\[
\tau := \{ \pi(s_l)\mid 1\leq l\leq 4g\} \cup \{ \pi(t') \mid 1 \le t \le 2g \} \cup \{ \pi(t'') \mid 2 \le t \le 2g-2 \}
\]
is a triangulation of the once-punctured surface $\surf$, where
$\marked = \{ \pi(q_1) \}$.
(Note that $\pi(q_1) = \pi(q_2) = \cdots = \pi(q_{4g})$, and
$\pi(s_{4k+\ell}) = \pi(s_{4k+\ell+2})$ for $0 \le k \le g-1$ and
$\ell = 1,2$.)
For $g \ge 2$ the
quiver $Q(\tau)$ looks as indicated in Figure~\ref{Fig:43-gon_nobound_1punct.eps}.
\begin{figure}[!htb]
                \centering
                \includegraphics[scale=.7]{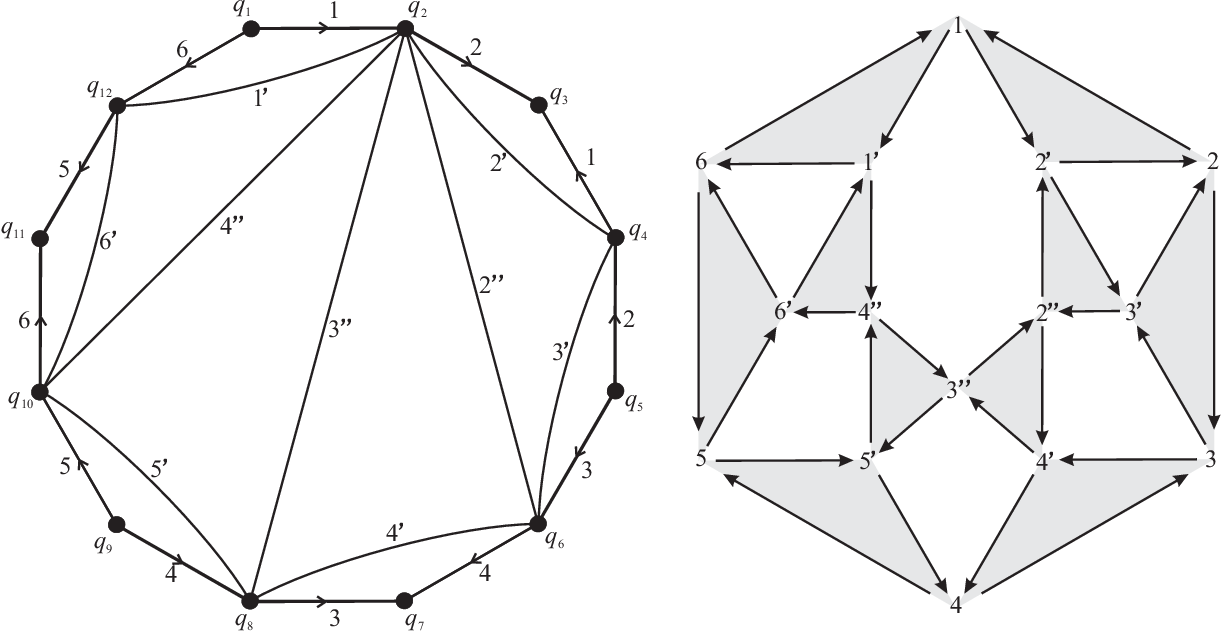}
\caption{A gentle triangulation $\tau$ of a surface $\surf$
with empty boundary, genus $3$ and
$|\marked|=1$.}
\label{Fig:43-gon_nobound_1punct.eps}
\end{figure}
It is easy to see that $\tau$ is a gentle triangulation.
Thus we obtained a gentle triangulation for surfaces $\surf$
with empty boundary, $|\marked| = 1$ and
genus $g \ge 2$.

\subsubsection{Surfaces $\surf$ with non-empty boundary,
$|\marked| = 1$ and genus $g \ge 2$}
Let $(\Sigma',\marked')$ be a once-punctured surface with empty
boundary and genus $g \ge 2$.
Let $\sigma$ be a gentle triangulation of $(\Sigma',\marked')$
as constructed in Section~\ref{sec6.2.9}.
Let $\surf$ be the surface obtained from $(\Sigma',\marked')$ by
cutting out a disc whose boundary contains the unique marked point in
$\marked'$ (hence $\marked=\marked'$).
We cut such a disc as shown in
Figure~\ref{Fig:new_43-gon_1bound_nopunct.eps} on the left.
In particular, the elements of $\sigma$ (which are arcs in
$(\Sigma',\marked')$) remain arcs in $\surf$.
We complete $\sigma$ to a
triangulation $\tau$ of $\surf$ by adding an arc $k$ as shown in
Figure~\ref{Fig:new_43-gon_1bound_nopunct.eps}.
The quiver $Q(\tau)$ is shown in Figure~\ref{Fig:new_43-gon_1bound_nopunct.eps} on the right.
\begin{figure}[!htb]
                \centering
                \includegraphics[scale=.7]{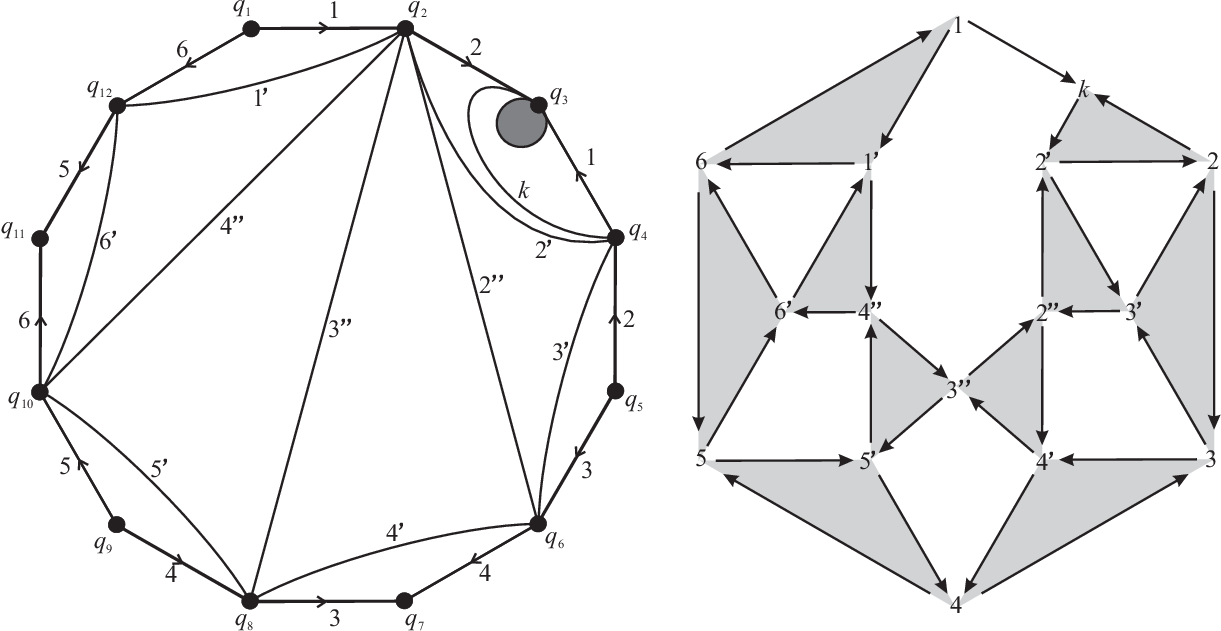}
\caption{A gentle triangulation $\tau$ of a surface $\surf$
with non-empty boundary, genus $3$ and
$|\marked|=1$.}
\label{Fig:new_43-gon_1bound_nopunct.eps}
\end{figure}
Now a straightforward check shows
that $\tau$ is a gentle triangulation.

\subsection{Proof of Theorem~\ref{thmgoodtriang}:
Induction step}\label{indstep}

\begin{lemma}\label{lemma:no-double-arrows-ind-step}
Suppose that a marked surface $\surf$ has a gentle
triangulation $\sigma$, and assume that a marked
surface $(\Sigma',\marked')$ can be obtained from
$\surf$ by one of the following operations:
\begin{itemize}

\item[(a)]
Adding a puncture;

\item[(b)]
Adding a boundary component with exactly one marked point on it;

\item[(c)]
Adding a marked point to a boundary component.

\end{itemize}
Then $(\Sigma',\marked')$ has a gentle triangulation.
\end{lemma}

\begin{proof}
(a)
Let
$a$ be an arc of $\sigma$.
The arc $a$ is part of exactly two triangles $\triangle_1$ and
$\triangle_2$. (These triangles are not self-folded, since $\sigma$
is a gentle triangulation.)
Pick any point $q$ lying on the arc $a$ such that $q$ is not one
of the end points of $a$.
Declare $q$ to be a new puncture,
so that $(\Sigma',\marked') := (\Sigma,\marked\cup\{q\})$ is obtained from
$\surf$ by adding a puncture.
As illustrated in Figure~\ref{Fig:adding_puncture}, we split $a$ into two arcs $a_1$ and $a_3$, and
draw two new arcs $a_2$ and $a_4$ contained in $\triangle_1$ and $\triangle_2$, respectively.
The result is a
triangulation $\tau$ of $(\Sigma',\marked')$.
Using the assumption that $\sigma$ is a gentle triangulation,
it is easy to verify that $\tau$ is also a gentle triangulation.
(The lower part of Figure~\ref{Fig:adding_puncture} shows how $Q(\tau)$
is obtained from $Q(\sigma)$ by replacing the full subquiver of $Q(\sigma)$ shown on the left
by the quiver shown on the right.)
\begin{figure}[!htb]
                \centering
                \includegraphics[scale=.43]{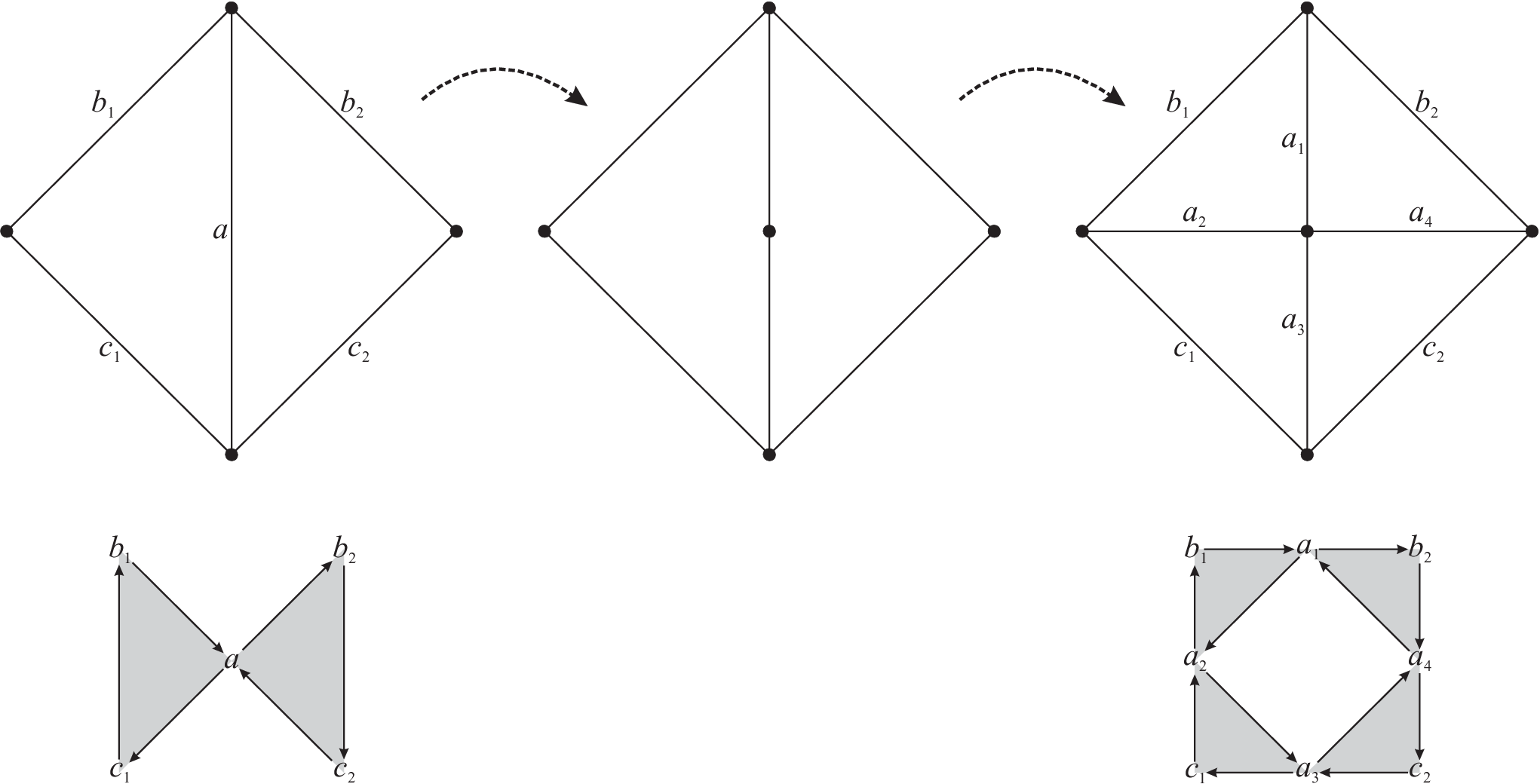}
\caption{Adding a puncture.}
\label{Fig:adding_puncture}
\end{figure}

(b)
Pick a triangle $\triangle$ of $\sigma$, and a topological open disc
$D \subseteq \triangle$ whose closure does not intersect any of the arcs in $\sigma$.
In this way we obtain a surface $\Sigma' := \Sigma\setminus D$ that has
one more boundary component than $\Sigma$.
Let $q$ be a point on the new
boundary component, and draw four arcs in
$(\Sigma',\marked\cup\{q\})$ as
shown in Figure~\ref{Fig:adding_boundary}.
Using the fact that $\sigma$ is a gentle triangulation, we get that
$\tau$ is also gentle.
\begin{figure}[!htb]
                \centering
                \includegraphics[scale=.43]{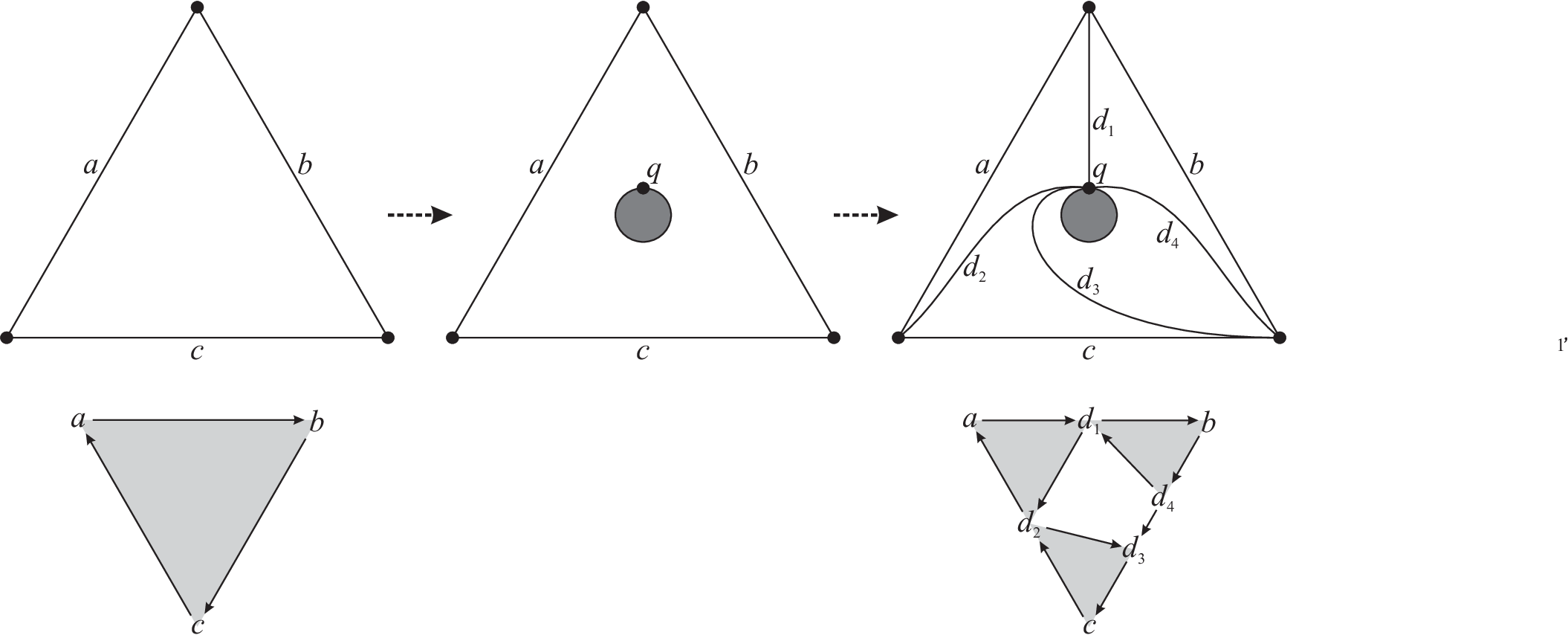}
\caption{Adding a boundary component.}
\label{Fig:adding_boundary}
\end{figure}

(c)
Let $C$ be any boundary component of $\Sigma$.
Pick any point $q$ on
$C \setminus \marked$.
Then $(\Sigma',\marked') := (\Sigma,\marked\cup\{q\})$
is obtained from $\surf$ by adding a marked point to a boundary component of $\Sigma$.
Draw the unique arc in $(\Sigma',\marked')$ that joins $q$ to a
point in $\marked$.
This completes $\sigma$ to a triangulation
$\tau$ of $(\Sigma',\marked')$.
This is illustrated in Figure~\ref{Fig:adding_boundarypoint}.
(Note that some of the marked points $q_1,q_2,q_3$ in
Figure~\ref{Fig:adding_boundarypoint} might be identical.)
Again, since $\sigma$ is gentle, we see that $\tau$ must be
gentle as well.
\begin{figure}[!htb]
                \centering
                \includegraphics[scale=.5]{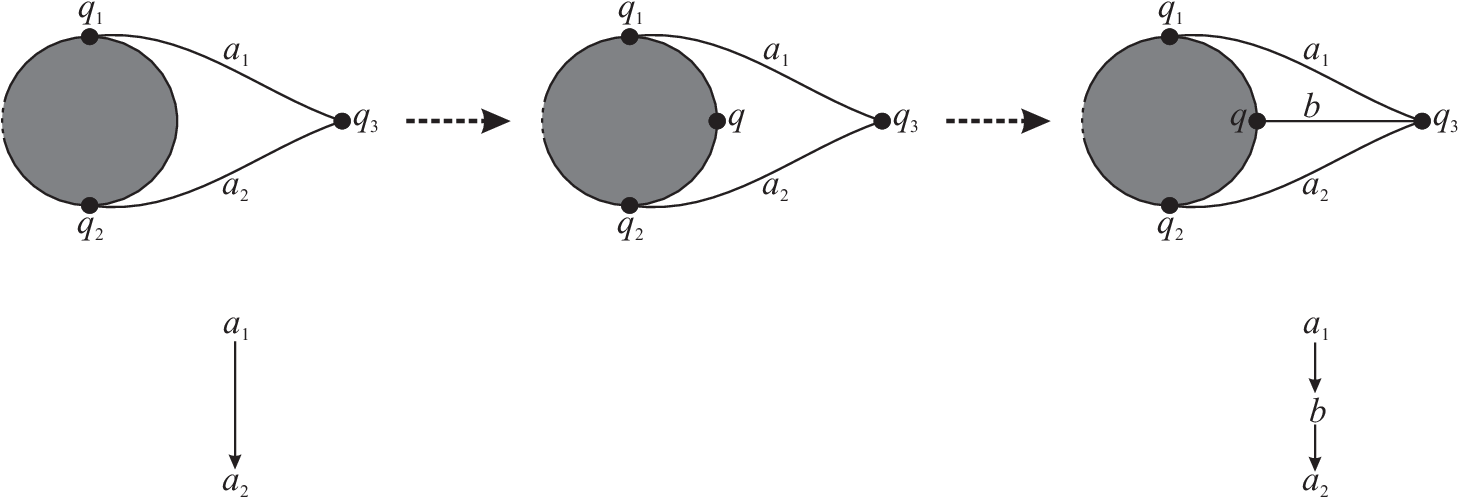}
\caption{Adding a marked point to a boundary component.}
\label{Fig:adding_boundarypoint}
\end{figure}
\end{proof}

To finish the proof of Theorem~\ref{thmgoodtriang},
observe that
any marked
surface $\surf$ different from the surfaces excluded in the assumptions
of Theorem~\ref{thmgoodtriang}
can be obtained from one of the surfaces treated in
Section~\ref{indbase}
by repeatedly adding
punctures, boundary components, and marked points to boundary components.
Now
Theorem~\ref{thmgoodtriang} follows from
Lemma~\ref{lemma:no-double-arrows-ind-step}.

\subsection{A skewed-gentle triangulation for a sphere
$\surf$ with $|\marked|=5$}\label{sphere5}
Each sphere $\surf$ with $|\marked| = 5$ has a skewed-gentle triangulation $\tau$.
Figure~\ref{Fig:new_sphere_5puncts.eps} illustrates the construction
of such a $\tau$ and its quiver $Q(\tau)$.
\begin{figure}[!htb]
                \centering
                \includegraphics[scale=.5]{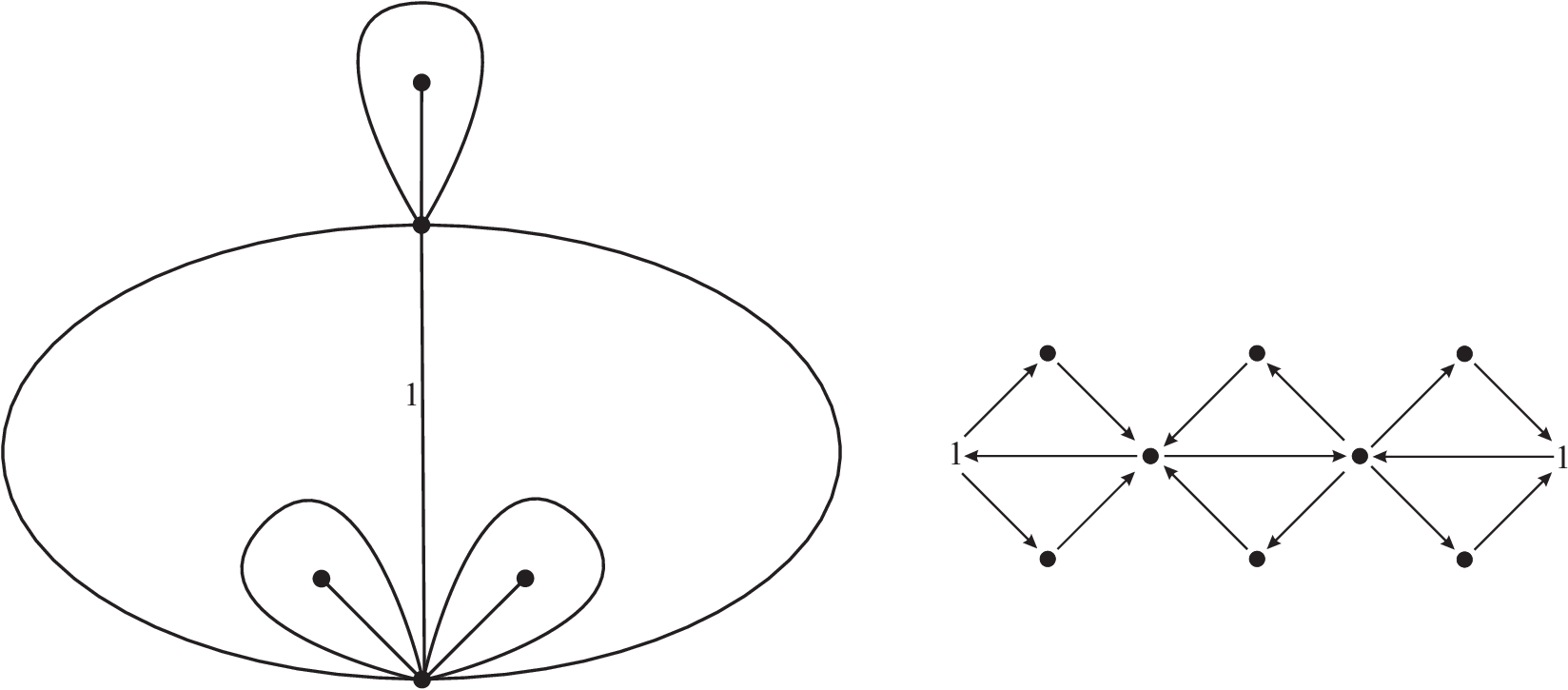}
\caption{
A skewed-gentle triangulation for a sphere $\surf$ with $|\marked| = 5$}
\label{Fig:new_sphere_5puncts.eps}
\end{figure}
%

\subsection{Skewed-gentle triangulations for monogons, digons
and triangles}\label{secskewedgentle}

\subsubsection{Monogons}\label{sec6.4.1}
Let $\surf$ be a monogon with $t \ge 3$ punctures.
For $t \ge 4$, Figure~\ref{Fig:new_monogon_clannish.eps}
shows a skewed-gentle triangulation $\tau$ of $\surf$ together
with the quiver $Q(\tau)$.
(The marked point on the boundary of $\Sigma$ is labeled by $p$.)
The case $t=3$ is dealt with in Figure~\ref{Fig:new_monogon_3puncts.eps}.
\begin{figure}[!htb]
                \centering
                \includegraphics[scale=.8]{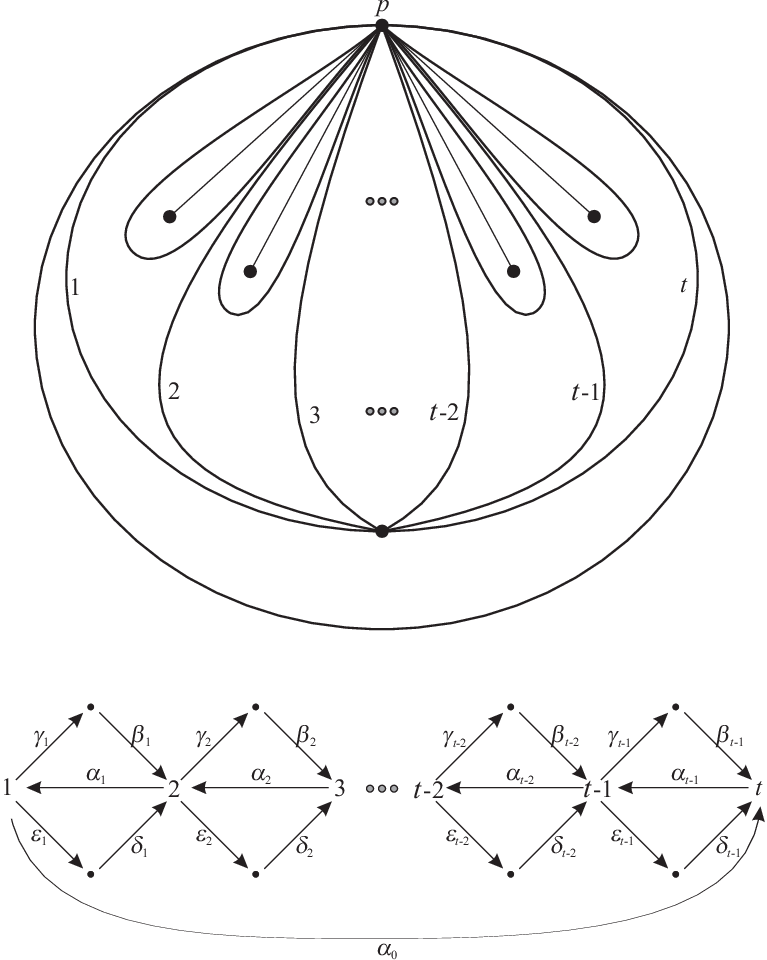}
\caption{A skewed-gentle triangulation for a monogon
$\surf$ with $t \ge 4$ punctures.}
\label{Fig:new_monogon_clannish.eps}
\end{figure}
\begin{figure}[!htb]
                \centering
                \includegraphics[scale=.5]{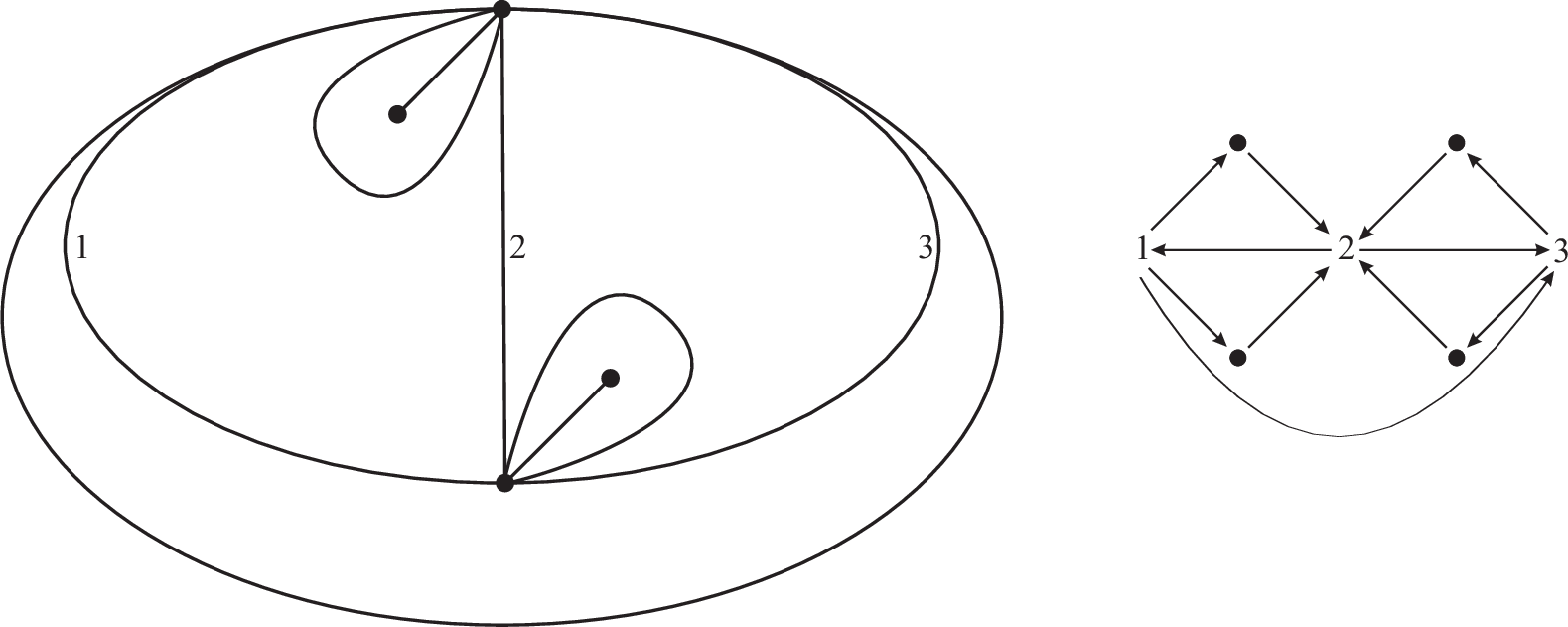}
\caption{A skewed-gentle triangulation for a monogon
$\surf$ with $3$ punctures.}
\label{Fig:new_monogon_3puncts.eps}
\end{figure}
%

\subsubsection{Digons}
Let $\surf$ be a digon with $t+1 \ge 2$ punctures.
Figure~\ref{Fig:new_digon_clannish.eps} shows a skewed-gentle
triangulation $\tau$ of $\surf$ together
with the quiver $Q(\tau)$.
\begin{figure}[!htb]
                \centering
                \includegraphics[scale=.4]{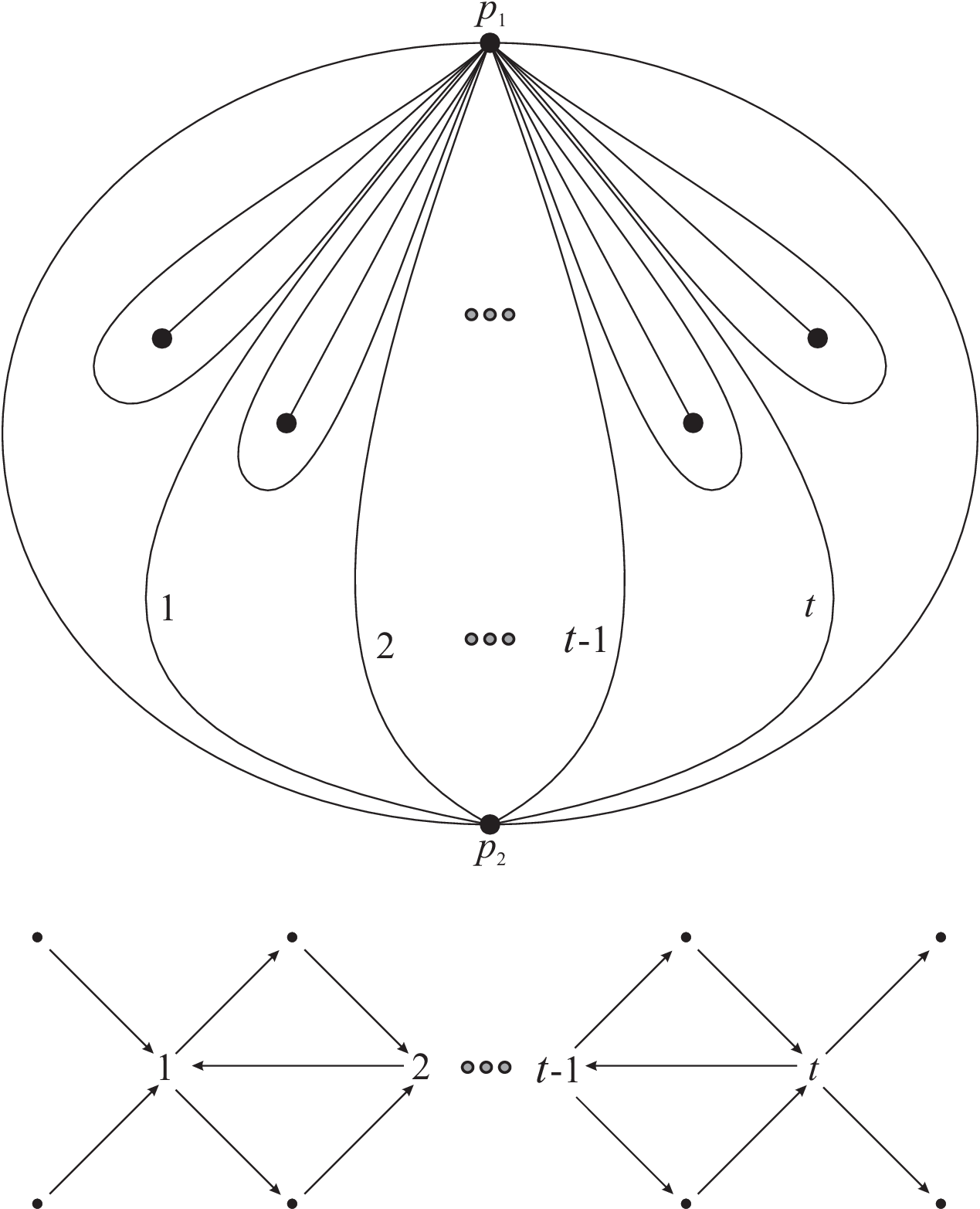}
\caption{
A skewed-gentle triangulation for a digon
$\surf$ with $t+1 \ge 2$ punctures.}
\label{Fig:new_digon_clannish.eps}
\end{figure}
%

\subsubsection{Triangles}
Let $\surf$ be a triangle with $t \ge 1$ punctures.
Figure~\ref{Fig:new_triangle_clannish.eps} shows a skewed-gentle
triangulation $\tau$ of $\surf$ together
with the quiver $Q(\tau)$.
\begin{figure}[!htb]
                \centering
                \includegraphics[scale=.4]{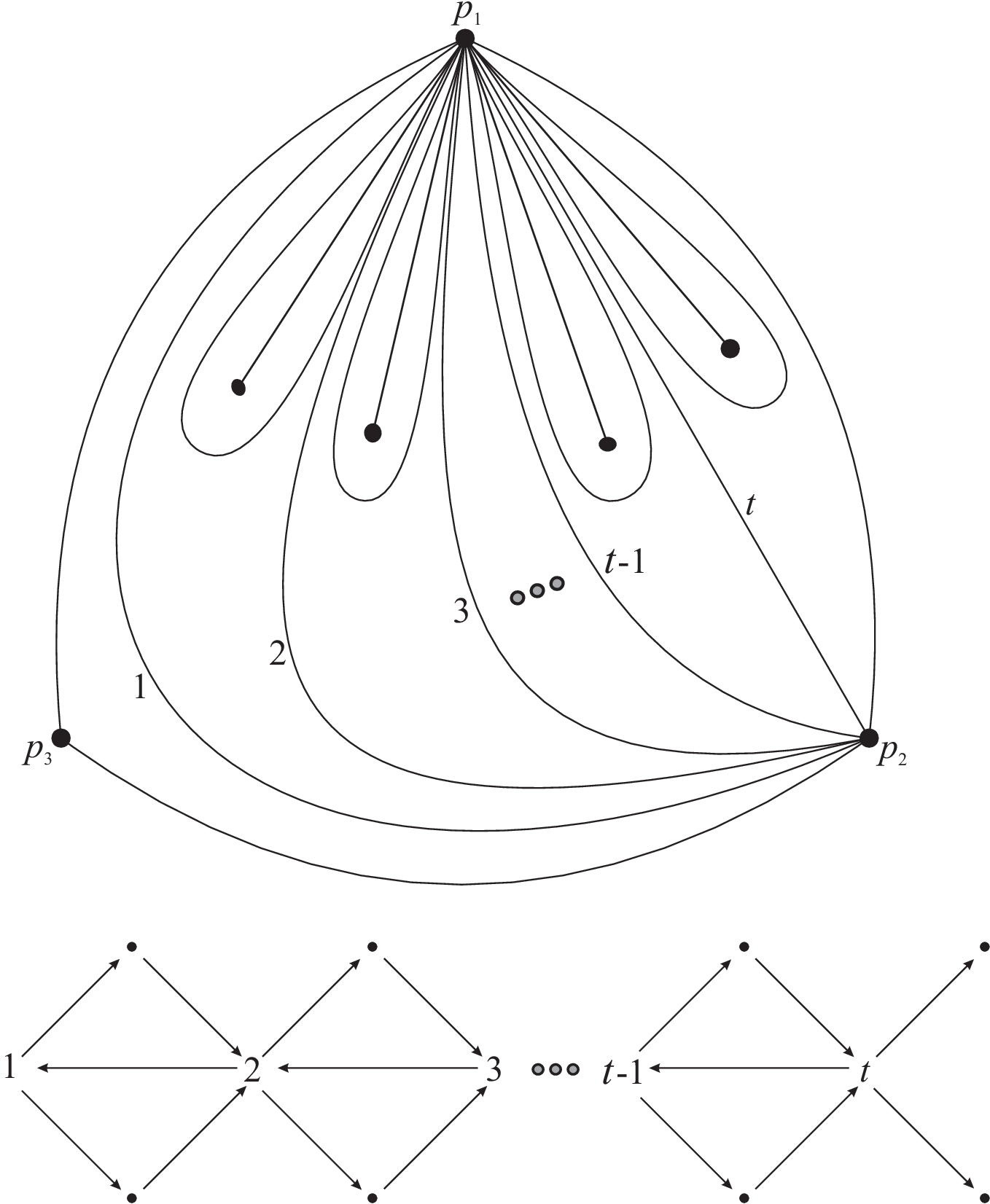}
\caption{
A skewed-gentle triangulation for a triangle
$\surf$ with $t \ge 1$ punctures.}
\label{Fig:new_triangle_clannish.eps}
\end{figure}
%

\subsection{Skewed-gentle triangulations for punctured
annuli $\surf$
with $|\marked \setminus \punct| = 2$}
Let $\surf$ be an annulus with $|\marked \setminus \punct| = 2$
and $|\punct| = t-1 \ge 1$ punctures.
Figure~\ref{Fig:new_annulus_clannish.eps} shows a skewed-gentle
triangulation $\tau$ of $\surf$ together
with the quiver $Q(\tau)$.
\begin{figure}[!htb]
                \centering
                \includegraphics[scale=0.8]{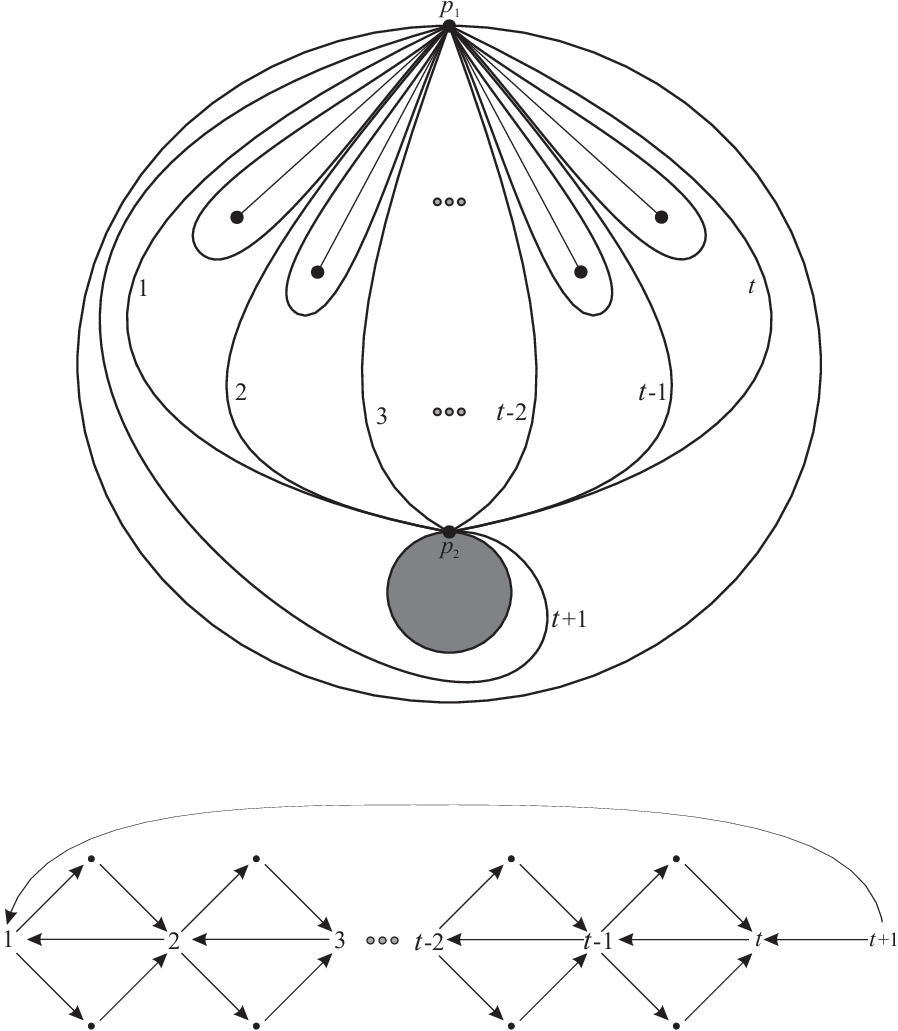}
\caption{
A skewed-gentle triangulation for a punctured annulus
$\surf$ with $|\marked \setminus \punct| = 2$. 
In the figure, the number of arcs is $t+1=|\punct|+2\geq 3$.}
\label{Fig:new_annulus_clannish.eps}
\end{figure}
%

\subsection{Skewed-gentle triangulations for a punctured torus $\surf$
with $|\marked \setminus \punct| = 1$}
Let $\surf$  be a torus with $|\marked \setminus \punct| = 1$
and $|\punct| = t \ge 1$ punctures.
Figure~\ref{Fig:torus_1bound_manypuncts.eps} shows a skewed-gentle
triangulation $\tau$ of $\surf$ together
with the quiver $Q(\tau)$.
\begin{figure}[!htb]
                \centering
                \includegraphics[scale=0.4]{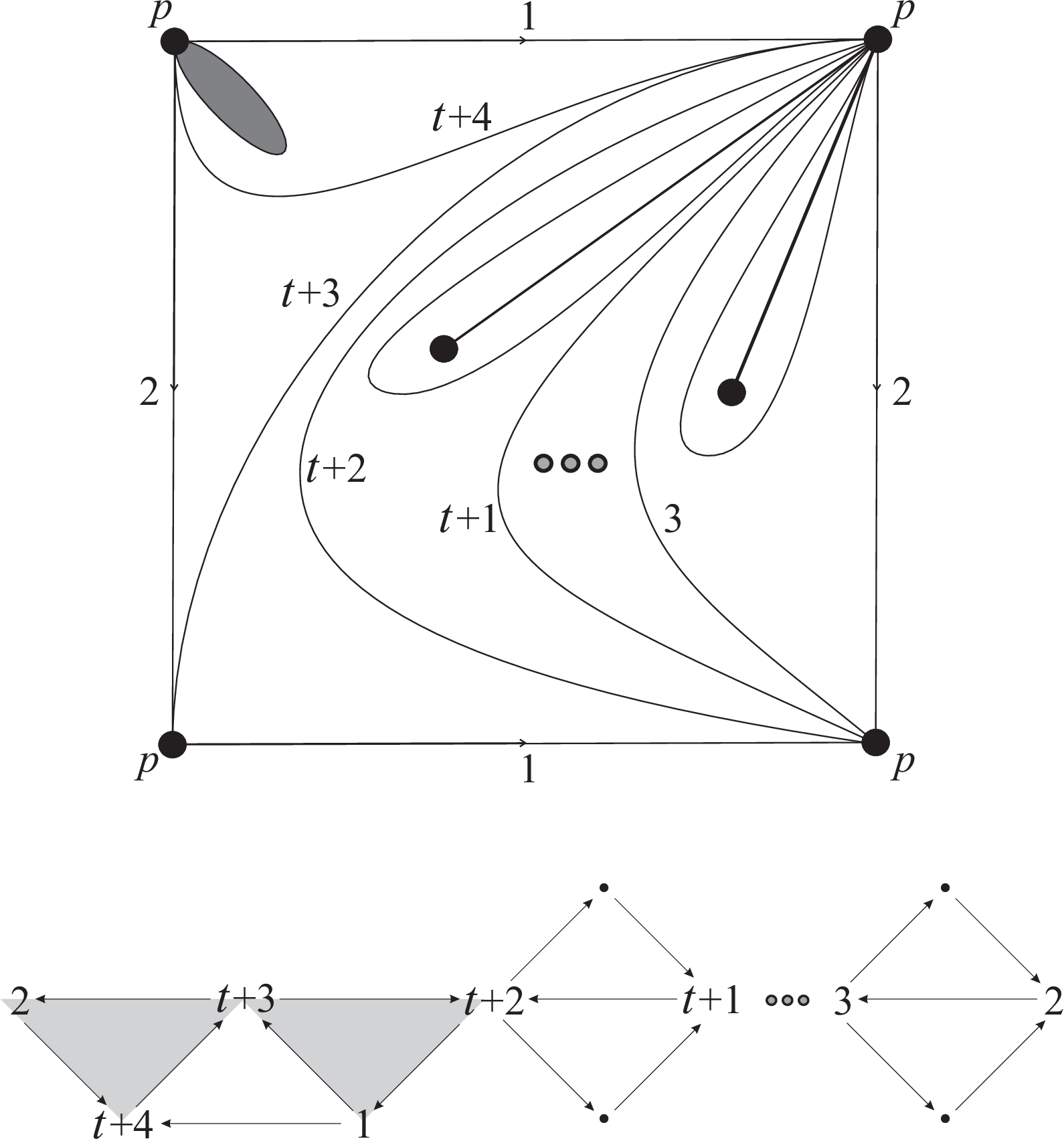}
\caption{
A skewed-gentle triangulation for a punctured torus $\surf$
with $|\marked \setminus \punct| = 1$.}
\label{Fig:torus_1bound_manypuncts.eps}
\end{figure}
%

\subsection{Triangulations of unpunctured surfaces}
\label{unpunctriang}
Let $Q$ be a 2-acyclic quiver.
A potential $S$ on $Q$ is called a $3$-\emph{cycle potential}
if $S = w_1 + \cdots + w_t$, where $w_1,\ldots,w_t$ is
a complete set of representatives of rotation equivalence classes
of 3-cycles in $Q$.
Note that for $Q = Q(\tau)$ and $\tau$ a gentle triangulation
of some marked surface $\surf$, the rotation equivalence classes of
$3$-cycles in $Q$ correspond bijectively to the interior triangles
of $\tau$.


\begin{prop}\label{uniquehelp1}
Let $(\Sigma,\marked)$ be an unpunctured marked surface, and
let $\tau$ be a triangulation of $\surf$ such that $Q(\tau)$ has
no double arrows.
Let $S$ be a $3$-cycle potential on $Q(\tau)$.
Then the following hold:
\begin{itemize}

\item[(g)]
$\tau$ is a gentle triangulation;

\item[(u)]
Each cycle $w$ of length at least $4$ in $Q(\tau)$ is rotation equivalent
to a cycle of the form $c\partial_\alp(S)$
for some arrow $\alp$ appearing in exactly
one summand of $S$ and some path $c$ of length at least $2$ in
$Q(\tau)$.

\end{itemize}
\end{prop}

The condition that an arrow $\alp$ appears in exactly one summand
of a potential $S$ on a quiver $Q$
can be rephrased by saying that the cyclic derivative
$\partial_\alp(S)$ is a scalar
multiple of a single path in $Q$.

\subsection{Skewed-gentle triangulations for surfaces with
non-empty boundary}
The following result will be needed in Section~\ref{sec10}
for the classification of non-degenerate potentials for
surfaces with non-empty boundary.

\begin{thm}\label{uniquehelp2}
Let $\surf$ be a marked surface with non-empty boundary.
Assume that $\surf$ is neither a monogon, nor an annulus with exactly two marked points, nor a torus with $|\marked|=1$.
Then there exist a triangulation $\tau$ of $\surf$, and a $3$-cycle potential $S\in\CQ{Q(\tau)}$, such
that the following hold:
\begin{itemize}

\item[(sg)]
$\tau$ is a skewed-gentle triangulation;

\item[(u)]
Each cycle $w$ of length at least $4$ in $Q(\tau)$ is rotation equivalent
to a cycle of the form $c\partial_\alp(S)$
for some arrow $\alp$ appearing in exactly
one summand of $S$
and some path $c$ of length at least $2$ in $Q(\tau)$.

\end{itemize}
\end{thm}

\begin{proof}
For $\surf$ a digon or a triangle, the triangulations shown in
Figures~\ref{Fig:new_digon_clannish.eps} and \ref{Fig:new_triangle_clannish.eps}, respectively, satisfy the properties
(sg) and (u).
If $\surf$ is a punctured annulus with $|\marked \setminus \punct| = 2$,
then the triangulation in Figure~\ref{Fig:new_annulus_clannish.eps} satisfies (sg) and (u).

Now assume that $\surf$ is not a digon, triangle or annulus with
$|\marked \setminus \punct| = 2$. We distinguish two cases, namely, either $\surf$ is different from a punctured torus with exactly one marked point on the boundary, or it is not.

Suppose first that $\surf$ is not a punctured torus with exactly one marked point on the boundary, and let $\marked_0=\marked\setminus\punct$. Then $\unpunctsurf$ is an unpunctured surface different from a torus with $|\marked|=1$. Let
$\tau_0$ be any triangulation of $\unpunctsurf$ such that
$Q(\tau_0)$ has no double arrows. (Such a triangulation exists
by Theorem~\ref{thmgoodtriang}.)
By Proposition~\ref{uniquehelp1} we know that $\tau_0$
is a gentle triangulation satisfying condition (u).
In particular, $\tau_0$ satisfies (sg).

Since $\partial \surfnoM \neq \varnothing$, the triangulation $\tau_0$ must have a
non-interior triangle, say $\triangle_0$.
At least one side of the triangle $\triangle_0$ belongs
to the boundary.
Let $p_1$ and $p_2$ be its (possibly identical) end points.
Put $t$ punctures $q_1,\ldots,q_t$ inside $\triangle_0$.
As illustrated in Figure~\ref{Fig:new_clan_triangulation},
we draw $t$ arcs $b_1,\ldots,b_t$ from $p_1$ to $p_2$, and for each $1 \le i \le t$
an arc $c_i$ from $p_1$ to $q_i$, and $t$ arcs $d_i$ from $p_1$ to itself,
such that $c_i$ and $d_i$ form a self-folded triangle for each $i$.
Together with the arcs form the triangulation $\tau_0$ of $\unpunctsurf$,
we obtain a triangulation $\tau_t$
of the marked surface $(\Sigma,\marked_t)$
where $\marked_t := \marked_0 \cup \{ q_1,\ldots,q_t \}$.
\begin{figure}[!htb]
                \centering
                \includegraphics[scale=.5]{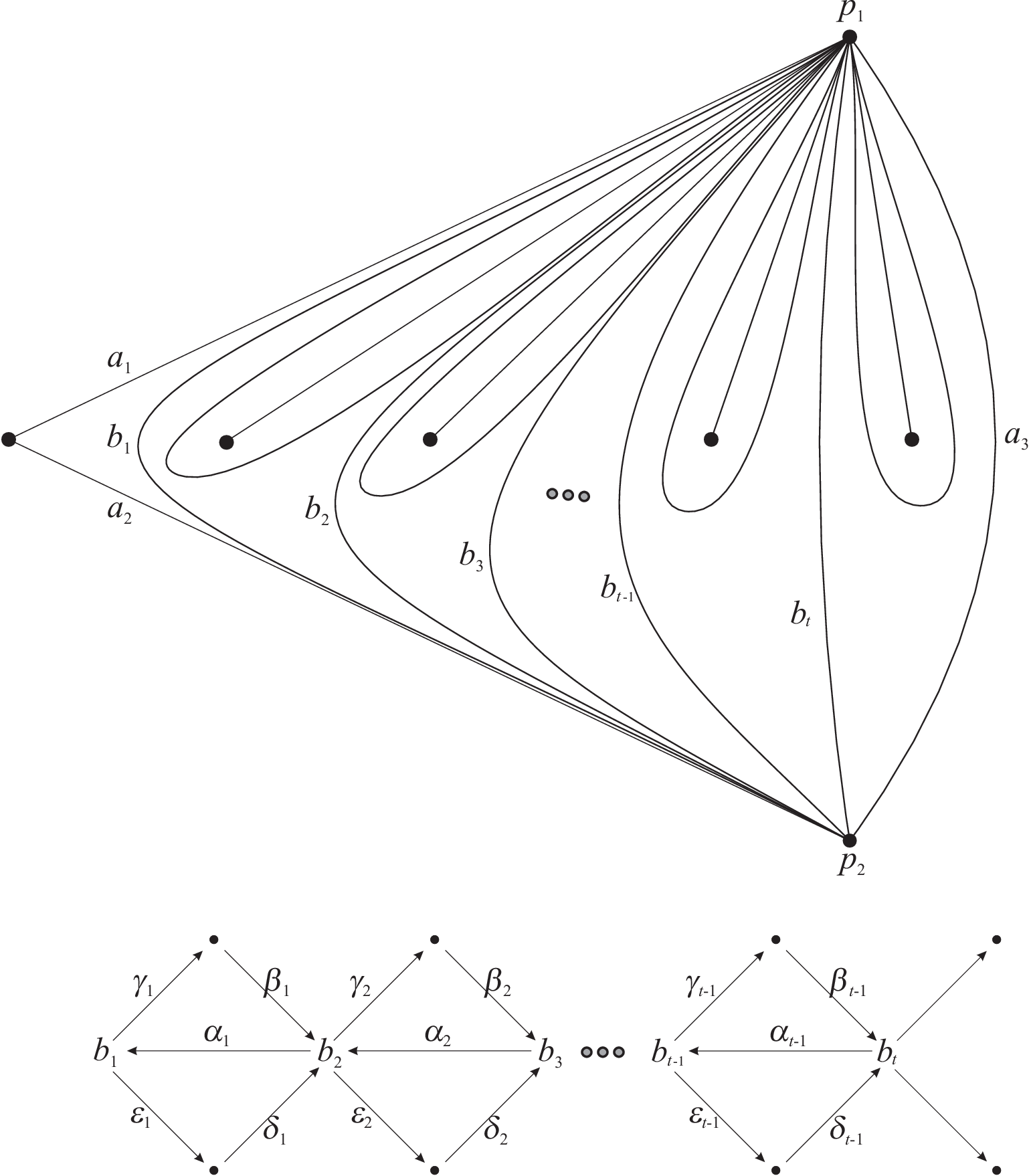}
\caption{
Construction of the triangulation $\tau_t$.}
\label{Fig:new_clan_triangulation}
\end{figure}
Let $C_t$ be the quiver shown in the lower part of Figure~\ref{Fig:new_clan_triangulation}.
The side $a_3$ of $\triangle_0$ belongs to the boundary of $(\Sigma,\marked_t)$, and
at most one of the sides $a_1$ or $a_2$ might also be part of
the boundary.
We obtain $Q(\tau_t)$ from  the disjoint union of the
quivers $Q(\tau_0)$ and $C_t$ by adding an
arrow $a_1 \to b_1$, provided $a_1$ does not belong to the boundary,
and an arrow $b_1 \to a_2$, provided $a_2$ is not part of the boundary.
Now,
using the fact that $\tau_0$ is a gentle triangulation satisfying (u),
one easily checks that $\tau_t$ is a skewed-gentle triangulation
also satisfying (u).
To be more precise, given a
cycle $w$ of length at least $4$ in $Q(\tau_t)$, one can always find a
summand $\alp\beta\gamma$ of any given 3-cycles potential $S$
on $Q(\tau_t)$
such that one arrow $\eta$ of
$\alpha\beta\gamma$ appears only once in
$S$ and $w' = c\partial_\eta(S)$ for some path $c$ and some rotation
$w'$ of $w$.

Now, suppose $\surf$ is a punctured torus with exactly one marked point on the boundary. Figure \ref{Fig:torus_1bound_manypuncts.eps} shows a triangulation satisfying (sg) and (u). This finishes the proof.
\end{proof}

It is easy to see
that for the triangulation $\tau_t$ constructed in the proof of
Theorem~\ref{uniquehelp2} the
Jacobian algebra $\cP(Q(\tau_t),S(\tau_t))$ is a
skewed-gentle algebra.
(This follows directly by combining \cite[Theorem~2.7]{ABCP} and the example in Section~\ref{block4clannish}, compare also the recent
preprint \cite{QZ}.)

\subsection{Triangulations without double arrows}\label{secnodouble}

\begin{prop}\label{thmnodoublearrows}
Let $Q$ be a $2$-acyclic quiver of finite mutation type which is not equal to
one of the quivers $T_1$, $T_2$ or $K_m$ with $m \ge 2$ (see Figure \ref{figureL3}).
Then $Q$ is mutation equivalent to a quiver without double arrows.
\end{prop}

\begin{proof}
Let $Q$ be a quiver as in the statement of the proposition. Then, by Felikson-Shapiro-Tumarkin's classification, $Q$ is mutation-equivalent to one of the following quivers:
\begin{enumerate}
\item a quiver of type $E_6$, $E_7$, $E_8$, $\widetilde{E}_6$, $\widetilde{E}_7$ or $\widetilde{E}_8$;
\item one of the quivers $E_6^{(1,1)}$, $E_7^{(1,1)}$, $E_8^{(1,1)}$, $X_6$ and $X_7$, displayed in Figure \ref{figureL2};
\item the quiver of a triangulation of a surface different from a torus with exactly one marked point, and from an unpunctured annulus with exactly two marked points.
\end{enumerate}

None of the quivers of type $E_6$, $E_7$, $E_8$, $\widetilde{E}_6$, $\widetilde{E}_7$ or $\widetilde{E}_8$ has double arrows.

Each of the quivers $E_6^{(1,1)}$, $E_7^{(1,1)}$, $E_8^{(1,1)}$, $X_6$ and $X_7$, displayed in Figure \ref{figureL2}, is one mutation away from a quiver without double arrows (in Figure \ref{figureL2}, mutate at the black vertex of the corresponding quiver).

If $\surf$ is not a torus with exactly one marked point, nor an unpunctured annulus with exactly two marked points, one can use Theorems \ref{thmgoodtriang} and \ref{uniquehelp2} to show that $\surf$ has a triangulation $\tau$ whose associated quiver $Q(\tau)$ does not have double arrows. Indeed, the spheres with four and five punctures obviously have triangulations without double arrows (spheres are always assumed to have at least four punctures), and, noticing that condition (gl3) of the definition of gentle triangulation implies that the corresponding quiver does not have double arrows, one can use Theorem \ref{thmgoodtriang} to show that any other empty-boundary surface different from a once-puncture torus has a triangulation whose quiver does note have double arrows. Similarly, using Subsection \ref{sec6.4.1} and Theorem \ref{uniquehelp2}, and recalling that condition (gl3) is required also in the definition of skewed-gentle triangulation, one can deduce that, with the only exceptions of the annulus with exactly two marked points and the unpunctured torus with exactly one boundary component and exactly one marked point, every surface with non-empty boundary different has a triangulation whose quiver does note have double arrows.
\end{proof}

The quivers $T_1$, $T_2$ and $K_m$ in the statement of Proposition~\ref{thmnodoublearrows} are displayed in Figure~\ref{figureL3}.
Recall that
the mutation equivalence class of each of these three quivers
contains just one quiver up to isomorphism.

\begin{figure}[H]
\begin{tabular}{lllll}
$\xymatrix@-0.5pc{
K_m \\
\circ \ar@<0.9ex>[r]_>>>>\cdots\ar@<-0.9ex>[r]
&\circ
}$
&&
$
\xymatrix@-1.5pc{
T_1
&& \circ \ar@<0.5ex>[dddrr]\ar@<-0.5ex>[dddrr]\\
&&&&\\
&&&&\\
\circ \ar@<0.5ex>[uuurr]\ar@<-0.5ex>[uuurr]
&&&& \circ \ar@<0.5ex>[llll]\ar@<-0.5ex>[llll]
}
$
&&
$
\xymatrix@-1.7pc{
T_2 && \circ \ar@<0.5ex>[dddd]\ar@<-0.5ex>[dddd]\\
&&&&\\
\circ \ar[uurr] \ar[rrrr]
&&&& \circ \ar[uull]\\
&&&&\\
&& \circ \ar[uurr]\ar[uull]
}
$
\end{tabular}
\caption{Quivers of finite mutation type
with non-removable double arrows.}
\label{figureL3}
\end{figure}
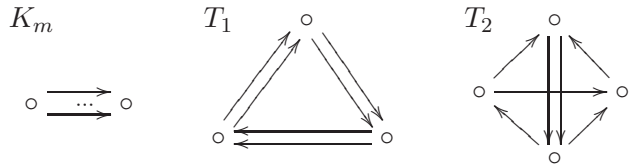

The proof of Proposition~\ref{thmnodoublearrows} contains a proof of the following result.

\begin{prop}\label{thm:no-double-arrows-existence}
Let $\surf$ be a marked surface which is not equal to
a torus $\surf$ with $|\marked| = 1$, or to an annulus
with $|\marked| = 2$.
Then $\surf$ has a triangulation $\tau$ such that $Q(\tau)$ has no double arrows.
\end{prop}

We were informed by Sefi Ladkani that he obtained a proof of Proposition \ref{thm:no-double-arrows-existence} independently.


\section{Classification of non-degenerate potentials: Regular cases}\label{sec10}


\subsection{Construction of non-degenerate potentials
for quivers associated to triangulations of surfaces}\label{labardinipotential}
Let $\surf$ be a marked surface.
As before, let $\punct$ be its set of punctures.

For technical reasons, we introduce some quivers that are obtained from signed adjacency quivers by adding some 2-cycles in specific situations.
Let $\tau$ be a triangulation of $\surf$.
For each puncture $p$ incident to exactly two arcs of $\tau$, we add to $\qtau$ a 2-cycle that connects those arcs and call the resulting quiver the \emph{unreduced signed adjacency quiver} $\unredqtau$.
It is clear that $\qtau$ can always be obtained from $\unredqtau$ by deleting 2-cycles.

Now  let $\mathbf{x}=(x_p)_{p\in\punct}$ be a choice of non-zero
scalars $x_p \in \ka$.
Following \cite{Labardini1}
we recall the definition of a potential $S(\tau,\bx)$ on $Q(\tau)$.
\begin{itemize}

\item[(i)]
Each interior non-self-folded triangle $\triangle$ of $\tau$ gives rise to a unique (up to rotation equivalence) oriented $3$-cycle $\alpha^\triangle\beta^\triangle\gamma^\triangle$ in $\unredqtau$. Define
\[
\widehat{S}^\triangle(\tau,\mathbf{x}) := 
\alpha^\triangle\beta^\triangle\gamma^\triangle.
\]

\item[(ii)]
If the interior non-self-folded triangle $\triangle$ with sides $j$, $k$ and $l$, is adjacent to two self-folded triangles as displayed in
Figure~\ref{adsftriangs},
define (up to rotation equivalence)
\[
\widehat{U}^\triangle(\tau,\mathbf{x}) :=
x_p^{-1}x_q^{-1}\alp\bet\gam,
\]
where $p$ and $q$ are the punctures enclosed in the self-folded triangles 
adjacent to $\triangle$.
\begin{figure}[!htb]
                \centering
                \includegraphics[scale=.5]{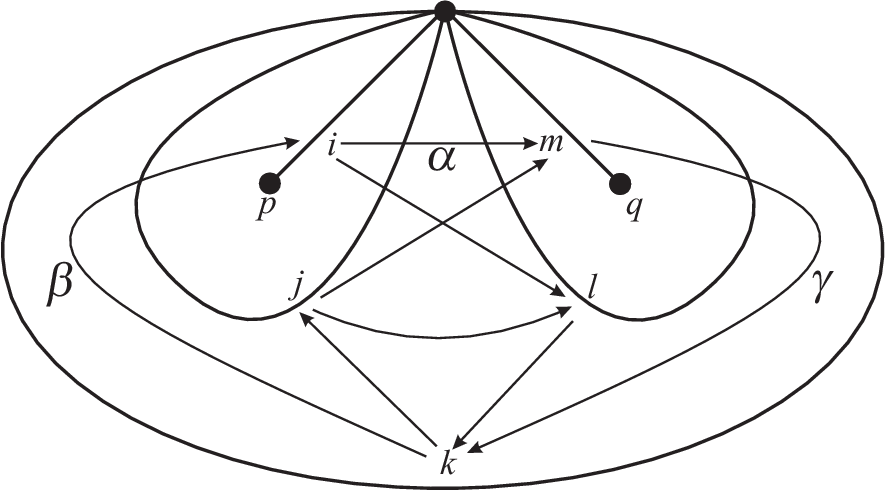}
\caption{
Definition of $\widehat{U}^\triangle(\tau,\mathbf{x})$.}
\label{adsftriangs}
\end{figure}
Otherwise, if $\triangle$ is adjacent to less than two self-folded triangles, define $\widehat{U}^\triangle(\tau,\mathbf{x}) := 0$.

\item[(iii)]
If a puncture $p$ is adjacent to exactly one arc $\arc$ of $\tau$, then $\arc$ is the folded side of a self-folded triangle of $\tau$ and around $\arc$ we have the configuration shown in Figure~\ref{sftriangle}.
\begin{figure}[!htb]
                \centering
                \includegraphics[scale=.5]{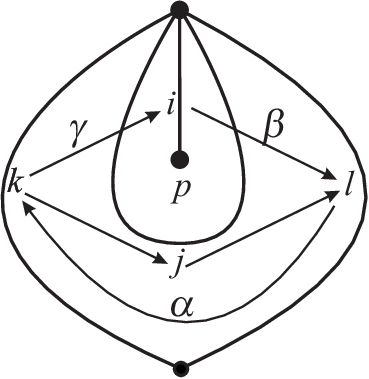}
\caption{
Definition of $\widehat{S}^p(\tau,\mathbf{x})$ (case (ii)).}\label{sftriangle}
\end{figure}
In case both $k$ and $l$ are indeed arcs of $\tau$ (and not part of the boundary of $\surfnoM$), we define (up to rotation equivalence)
\[
\widehat{S}^p(\tau,\mathbf{x}) :=
-x_p^{-1}\alp\bet\gam.
\]
Otherwise, if either $k$ or $l$ is a boundary segment, we define $\widehat{S}^{p}(\tau,\mathbf{x}) := 0$.

\item[(iv)]
If a puncture $p$ is adjacent to more than one arc, delete all the loops incident to $p$ that enclose self-folded triangles.
The arrows between the remaining arcs adjacent to $p$ give rise
to a unique (up to rotation equivalence)
cycle $\alpha^p_1\ldots \alpha^p_{d_p}$ (without repeated arrows) around $p$ in the counterclockwise orientation (defined by the orientation of $\surfnoM$).
Note that some of the remaining arcs might still be loops.
We define (up to rotation equivalence)
\[
\widehat{S}^p(\tau,\mathbf{x}) := x_p\alpha^p_1\ldots \alpha^p_{d_p}.
\]
The definition of $\alpha^p_1\ldots \alpha^p_{d_p}$ is illustrated
in Figure~\ref{Fig:new_typical_cycle_around_puncture.eps}.
\begin{figure}[!htb]
                \centering
        \includegraphics[scale=.6]{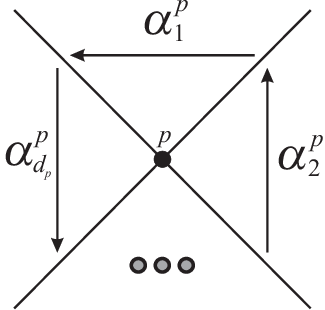}
\caption{
The cycle $\alpha^p_1\ldots \alpha^p_{d_p}$ around a puncture $p$.}
\label{Fig:new_typical_cycle_around_puncture.eps}
\end{figure}

\item[(v)]
The \emph{unreduced potential}
$\unredstaux \in \CQ{\unredqtau}$ associated to $\tau$ and $\mathbf{x}$ is defined as
\[
\unredstaux :=
\sum_\triangle\left(\widehat{S}^\triangle(\tau)+
\widehat{U}^\triangle(\tau,\mathbf{x})\right) +
\sum_{p\in\punct}\left(\widehat{S}^p(\tau,\mathbf{x})\right)
\]
where the first sum runs over all interior non-self-folded triangles of $\tau$.

\item[(vi)]
We define $\qstaux$ to be the reduced part of $\unredqstaux$.

\item[(vii)]
Define
\[
S(\tau) := S(\tau,\bx)
\]
provided $x_p = 1$ for all $p \in \punct$.

\end{itemize}
If $\tau$  is a triangulation such that every puncture of $\surf$ is incident to at least three arcs of $\tau$, then $\unredqtau$ is already 2-acyclic, hence equal to $\qtau$, and therefore $\unredstaux=\staux$. Moreover, the triangulation $\tau$ does  not have self-folded triangles.
This implies that $\staux$ takes a simpler form, namely we have
\[
\staux=\sum_\triangle\left(\alpha^\triangle
\beta^\triangle\gamma^\triangle\right) +
\sum_{p\in\punct}\left(x_p\alp^p_1\ldots \alp^p_{d_p}\right).
\]
The only situation where one needs to apply reduction to $(\unredqtau,\unredstaux)$ in order to obtain $S(\tau,\bx)$ is when there is some puncture incident to exactly two arcs of $\tau$.
The reduction is done explicitly in \cite[Section 3]{Labardini1}.

For a tagged triangulation $\tau$,
the definition of $S(\tau,\bx)$ can be
found in \cite{Labardini4}.

\begin{thm}[{\cite{Labardini4}}]\label{thm:flip-is-QP-mut}
Let $\surf$ be a marked surface which is not equal to a sphere
with $|\marked| = 4,5$.
If $\tau$ and $\sigma$ are tagged triangulations of $\surf$ related by the flip of a tagged arc $k$, then $\mu_k(\qstaux)$ is right equivalent to $\qssigmax$.
\end{thm}

Theorem \ref{thm:flip-is-QP-mut} has the following immediate consequence.

\begin{coro}[{\cite{Labardini4}}]\label{thm:staux-always-nondeg}
Let $\surf$ be a marked surface which is not
a sphere with $|\marked| = 4,5$.
Then for any tagged triangulation $\tau$ of $\surf$, the
potentials $\staux$ are non-degenerate.
\end{coro}

In the case of a sphere with five punctures, only a weaker version of Theorem \ref{thm:flip-is-QP-mut} has been proved.
Namely, ideal triangulations related by a flip have QPs related by the corresponding QP-mutation, see \cite{Labardini1}, but the tagged version of this statement has not yet been proved for this case.

\begin{prop}[{\cite{Labardini4}}]
\label{prop:non-empty=>scalars-irrelevant}
Suppose $\surf$ is a marked surface with non-empty boundary.
Then for any two choices $\mathbf{x}=(x_p)_{p\in\punct}$ and $\mathbf{y}=(y_p)_{p\in\punct}$ of non-zero scalars
$x_p,y_p \in \ka$,
the QPs $(\qtau,S(\tau,\mathbf{x}))$ and $(\qtau,S(\tau,\mathbf{y}))$ are right equivalent.
\end{prop}

\subsection{Potentials for surfaces with empty boundary}\label{secemptybdpot}
In this section we prove the following classification theorem.

\begin{thm}\label{thmemptybdpot}
Let
$\surf$ be a marked surface with empty boundary.
Assume that
\[
|\marked| \geq
\begin{cases}
6 & \text{if $\Sigma$ is a sphere},\\
3 & \text{otherwise}.
\end{cases}
\]
Then for
any tagged triangulation $\tau$ of $\surf$ the quiver $\qtau$ admits
only one non-degenerate potential up to weak right equivalence.
\end{thm}

We conjecture that Theorem~\ref{thmemptybdpot}
holds as well for all marked surfaces $\surf$ with
genus $g(\Sigma) \ge 1$ and $|\marked| = 2$.
We will see in Section~\ref{sec11b} that for surfaces $\surf$
with empty boundary and $|\marked| =1$ there are at
least two non-degenerate potentials up to weak right equivalence.

We now prove a series of lemmas leading to a proof of
Theorem~\ref{thmemptybdpot} .

\begin{lemma}\label{lemma:Staux=lambdaStauy}
Suppose $\surf$ is a marked surface with empty boundary different from a sphere with $|\marked| = 4,5$.
Let $\bx = (x_p)_{p \in \punct}$ and $\by = (y_p)_{p \in \punct}$
be arbitrary choices of non-zero scalars.
Then for every tagged triangulation $\sigma$ of $\surf$, the QPs $(Q(\sigma),S(\sigma,\bx))$ and $(Q(\sigma),S(\sigma,\by))$ are weakly right equivalent.
\end{lemma}

\begin{proof}
We start with a couple of general observations.
Let $Q$ be any 2-acyclic quiver (not necessarily arising from a triangulation).
Then the following hold:
\begin{enumerate}

\item
If $S_1,S_2,S_3$ are potentials on $Q$ such that $S_1$ is weakly right equivalent to $S_2$, and $S_2$ is weakly right equivalent to $S_3$, then $S_1$ is weakly right equivalent to $S_3$;

\item
If $S_1$ and $S_2$ are weakly right equivalent potentials on $Q$, then $\mu_k(Q,S_1)$ and $\mu_k(Q,S_2)$ are weakly right equivalent for any vertex $k\in Q_0$. This follows from the following obvious facts:
\begin{itemize}

\item
for any non-zero scalar $\lambda$, the potential $[\lambda S_2]+\Delta_k(Q)\in\CQ{\widetilde{\mu}_k(Q)}$ is right equivalent to $\lambda([S_2]+\Delta_k(Q))=\lambda\widetilde{\mu}_k(S_2)$, and hence, $[S_1]+\Delta_k(Q)=\widetilde{\mu}_k(S_1)$ is right equivalent to $\lambda\widetilde{\mu}_k(S_2)$;

\item
if $\varphi\colon \CQ{\widetilde{\mu}_k(Q)}\rightarrow\CQ{\widetilde{\mu}_k(Q)}$
is a right equivalence between $(\widetilde{\mu}_k(Q),\widetilde{\mu}_k(S_2))$ and
    $(\widetilde{\mu}_k(Q)_{\operatorname{red}},\widetilde{\mu}_k(S_2)_{\operatorname{red}})\oplus
    (\widetilde{\mu}_k(Q)_{\operatorname{triv}},\widetilde{\mu}_k(S_2)_{\operatorname{triv}})$,
then $\varphi$ is also a right equivalence between
    $(\widetilde{\mu}_k(Q),\lambda\widetilde{\mu}_k(S_2))$ and
    $(\widetilde{\mu}_k(Q)_{\operatorname{red}},\lambda\widetilde{\mu}_k(S_2)_{\operatorname{red}})\oplus
    (\widetilde{\mu}_k(Q)_{\operatorname{triv}},\lambda\widetilde{\mu}_k(S_2)_{\operatorname{triv}})$;
\item for any non-zero scalar $\lambda$, the QP $(\widetilde{\mu}_k(Q)_{\operatorname{red}},\lambda\widetilde{\mu}_k(S_2)_{\operatorname{red}})$ is reduced and the QP
    $(\widetilde{\mu}_k(Q)_{\operatorname{triv}},\lambda\widetilde{\mu}_k(S_2)_{\operatorname{triv}})$ is trivial.
\end{itemize}
\end{enumerate}
These facts, together with Theorem~\ref{thm:flip-is-QP-mut}, imply that, in order to prove Lemma \ref{lemma:Staux=lambdaStauy}, it suffices to show the mere existence of a triangulation $\tau$ such that $(Q(\tau),S(\tau,\mathbf{1}))$ is weakly right equivalent to a scalar multiple of $(Q(\tau),S(\tau,\bx))$.

Suppose $\tau$ is a triangulation such that every puncture
$p \in \punct$ has valency $\val_\tau(p) \ge 3$.
It follows that
there is a puncture $p$ with $\val_\tau(p) \ge 4$.
For the existence of such a triangulation $\tau$ we refer to \cite{L2}.
Now choose any puncture $q$.
A minor modification of the proof of \cite[Proposition 11.2]{Labardini4} proves that $S(\tau,\mathbf{x})$ is right equivalent to $S(\tau,\mathbf{w})$, where $\mathbf{w}=(w_p)_{p \in \mathbb{P}}$ is defined by $w_q=\prod_{p \in \mathbb{P}} x_p$, and $w_p=1$ for $p \neq q$.

Let $a$ be the number of arrows of $Q(\tau)$, and $r$ be the number of punctures of $\surf$.
Note that the integer $a-3r$ is positive (this follows from the fact that all punctures have valency at least 3,
and at least one puncture has valency at least 4).
Let $\xi \in \ka$ be an $(a-3r)$-root of $w_q$. For each puncture $p$, let $z_p=\xi^{a_p-3}$, where $a_p$ is the number of arrows in the cycle that surrounds $p$. Note that $\prod_{p \in \mathbb{P}}z_p = w_q$. As in the previous paragraph, this implies that $S(\tau,\mathbf{z})$ is right equivalent to $S(\tau,\mathbf{w})$.

We see that, in order to prove that $S(\tau,\mathbf{1})$ is right equivalent to a scalar multiple of $S(\tau,\mathbf{x})$, it is enough to show that $S(\tau,\mathbf{1})$ is right equivalent to a scalar multiple of $S(\tau,\mathbf{z})$. But this is easy, let $\varphi$ be the $R$-algebra automorphism of $R\langle\langle Q(\tau)\rangle\rangle$ given by $\varphi\colon \beta \mapsto \xi\beta$ for every arrow $\beta$ of $Q$. Then $\varphi$ is a right equivalence between $S(\tau,\mathbf{1})$ and $\xi^3S(\tau,\mathbf{z})$.
\end{proof}

Let $\surf$ be a surface with empty boundary, and let $\tau$ be a triangulation of $\surf$ such that every puncture has valency at least 3. Let $Q=Q(\tau)$ be the quiver of $\tau$.
Following Ladkani \cite{L2} we
define two maps $f,g\colon Q_1 \to Q_1$ as follows. Each triangle $\triangle$ of $\tau$ gives rise to an oriented 3-cycle $\alpha^{\triangle}\beta^{\triangle}\gamma^{\triangle}$ on $Q(\tau)$. We set $f\alpha^\triangle=\gamma^\triangle$, $f\beta^\triangle=\alpha^\triangle$ and $f\gamma^\triangle=\beta$. Now, given any arrow $\alpha$ of $Q(\tau)$, the quiver $Q(\tau)$ has exactly two arrows starting at the terminal vertex of $\alpha$. One of these two arrows is $f\alpha$. We define $g\alpha$ to be the other arrow.

Note that the map $f$ (resp. $g$) splits the arrow set of $Q(\tau)$ into $f$-orbits (resp. $g$-orbits). The set of $f$-orbits is in one-to-one correspondence with the set of triangles of $\tau$. All $f$-orbits have exactly three elements. The set of $g$-orbits is in one-to-one correspondence with the set of punctures of $\surf$. For every arrow $\alpha$ of $Q(\tau)$, we denote by $m_\alpha$ the size of the $g$-orbit of $\alpha$. Note that $(g^{m_\alpha-1}\alpha)(g^{m_\alpha-2}\alpha)\ldots(g\alpha)\alpha$ is a cycle surrounding the puncture corresponding to the $g$-orbit of $\alpha$.

For Lemmas \ref{lemma:3-types-of-cycles}, \ref{lemma:killing-type-II}, \ref{lemma:killing-type-I} and \ref{lemma:killing-type-III}, let $\{\alpha_1,\ldots,\alpha_\ell\}$ (resp. $\{\beta_1,\ldots,\beta_{|\punct|}\}$) be a system of representatives  for the action of the map $f$ (resp. $g$) on the set of arrows of $Q(\tau)$.
Thus, each $f$-orbit (resp. each $g$-orbit) has exactly one representative in the set $\{\alpha_1,\ldots,\alpha_\ell\}$ (resp. $\{\beta_1,\ldots,\beta_{|\punct|}\}$). We also use the following notations
\begin{equation}\label{eq:A-and-B}
A:= \sum_{k=1}^\ell (f^2\alpha_k)(f\alpha_k)(\alpha_k) \ \ \ \text{and} \ \ \ B:= \sum_{t=1}^{|\punct|} x_t\left((g^{m_{\beta_t}-1}\beta_t)(g^{m_{\beta_t}-2}\beta_t)\ldots (g\beta_t)(\beta_t)\right).
\end{equation}
Note that, with these notations, we have $S(\tau,\bx)=A+B$.

\begin{defi}\label{def:3-types-of-cycles} Let $\surf$ be a marked surface with empty boundary and let $\tau$ be an ideal triangulation of $(\Sigma,\mathbb{M})$ for which every puncture has valency at least three. We will say that a cycle $\xi$ in $Q(\tau)$ is:
\begin{itemize}
\item \emph{of type I} if $\xi=((f^2\alpha)(f\alpha)\alpha)^n$ for some arrow $\alpha$ and some $n>1$;
\item \emph{of type II} if $\xi=((g^{m_\beta-1}\beta)(g^{m_\beta-2}\beta)\ldots (g\beta) \beta )^n$ for some arrow $\beta$ and some $n>1$;
\item \emph{of type III} if $\xi=(f^2\beta)(f\beta)\lambda(f^2\alpha)(f\alpha)\rho$ for some arrows $\alpha,\beta$, and some paths $\lambda,\rho$, such that $\lambda=(g^{-1}f\beta)\lambda'$ or $\rho=\rho'(gf^2\beta)$.
\end{itemize}
Given a potential $S$ on $Q(\tau)$, we will say that $S$ \emph{involves only cycles of type I} (resp. \emph{II}, \emph{III}) if $S$ is a possibly infinite linear combination of cycles of type I (resp. II, III).
\end{defi}

Of course, not every cycle in $Q(\tau)$ is necessarily of type I, type II or type III.

\begin{lemma}\label{lemma:3-types-of-cycles}
Let $\surf$ be a marked surface with empty boundary.
Suppose $\tau$ is a tagged triangulation of $\surf$ such that
the following hold:
\begin{eqnarray}
\label{eq:valency>3}
&&\text{Every puncture has valency at least four;}\\
\label{eq:tau-has-no-loops}
&&\text{None of the arcs in $\tau$ is a loop;}\\
\label{eq:Q(tau)-has-no-double-arrows}
&&\text{$Q(\tau)$ has no double arrows.}
\end{eqnarray}
Then every cycle in $Q(\tau)$ that is rotationally disjoint from $S(\tau,\bx)$, is rotationally equivalent to a cycle of type I, II or III.
%
%
%
%
\end{lemma}

\begin{proof}
Let $\xi=\alpha_1\ldots\alpha_r$ be any cycle on $Q(\tau)$. Denote $\alpha_{r+1}=\alpha_1$, and notice that for every $\ell=1,\ldots,r$, we have either $\alpha_\ell=f\alpha_{\ell+1}$ or $\alpha_{\ell}=g\alpha_{\ell+1}$. Let $\mathbf{fg}_\xi$ be the length-$r$ sequence of $f$s and $g$s that has an $f$ at the $\ell^{\operatorname{th}}$ place if $\alpha_\ell=f\alpha_{\ell+1}$ and a $g$ otherwise.

If $\mathbf{fg}_\xi$ consists only of $f$s, then $\xi$ is rotationally equivalent to $((f^2\alpha)(f\alpha)\alpha)^n$ for some arrow $\alpha$ and some $n\geq1$. If $\mathbf{fg}_\xi$ consists only of $g$s, then $\xi$ is rotationally equivalent to $((g^{m_\beta-1}\beta)(g^{m_\beta-2}\beta)\ldots (g\beta) \beta )^n$ for some arrow $\beta$ and some $n\geq1$. Therefore, if $\xi$ is rotationally disjoint from $S(\tau,\bx)$ and $\mathbf{fg}_\xi$ involves only $f$s or only $g$s, then $\xi$ is rotationally equivalent to a cycle of type (I) or (II).

Suppose that at least one $f$ and at least one $g$ appear in $\mathbf{fg}_\xi$. Rotating $\xi$ if necessary, we can assume that $\mathbf{fg}_\xi$ starts with an $f$ followed by a $g$, i.e., $\mathbf{fg}_\xi=(f,g,\ldots)$. This means that $\alpha_1=f\alpha_2$ and $\alpha_2=g\alpha_3$. In particular, the arrows $\alpha_1$ and $\alpha_2$ are contained in a common triangle $\triangle$ of $\tau$. Since no arc in $\tau$ is a loop, the vertices of $\triangle$ are three different punctures. Hence the puncture associated to $\alpha_2$ is not incident to the arc in $\tau$ which is opposite to $\alpha_2$ in $\triangle$.

Consider the path $\alpha_3\ldots\alpha_r$. It starts precisely at the arc in $\tau$ which is opposite to $\alpha_2$ in $\triangle$, and it ends at an arc incident to the puncture associated to $\alpha_2$. Since all elements of $\{g^{n}\alpha_3\mid n\in\mathbb{Z}\}$ are arrows connecting arcs that are incident to the puncture associated to $\alpha_2$, we deduce that there is some $\ell\in\{3,\ldots,r-1\}$ such that $\alpha_{\ell}=f\alpha_{\ell+1}$. This means that $\xi$ is rotationally equivalent to a cycle of type (III).

Lemma \ref{lemma:3-types-of-cycles} is proved.
\end{proof}

It is easy to check that a triangulation $\tau$ satisfying the
assumptions \eqref{eq:valency>3}, \eqref{eq:tau-has-no-loops} and \eqref{eq:Q(tau)-has-no-double-arrows} in
Lemma~\ref{lemma:3-types-of-cycles} is a gentle
triangulation, and that $Q(\tau) =  {\rm glue}(B_1,\ldots,B_t;g)$
with $B_1,\ldots,B_t$ blocks of type II.

\begin{lemma}\label{lemma:killing-type-II}
Let $\surf$ be a marked surface with empty boundary,
and let $\tau$ be a tagged triangulation of $\surf$ satisfying \eqref{eq:valency>3}, \eqref{eq:tau-has-no-loops} and \eqref{eq:Q(tau)-has-no-double-arrows}. If $S'\in\CQ{Q(\tau)}$ is a potential rotationally disjoint from $S(\tau,\bx)$, then $(Q(\tau),S(\tau,\bx)+S')$ is right equivalent to $(Q(\tau),S(\tau,\bx)+W)$ for some potential $W\in \CQ{Q(\tau)}$ that can be written as $W=W_I+W_{III}$ for a potential $W_I$ (resp. $W_{III}$) which involves only cycles of type I (resp. type III).
\end{lemma}

\begin{proof} By Lemma \ref{lemma:3-types-of-cycles}, up to cyclical equivalence we can write $S'=S'_\I+S'_\II+S'_\III$, with
\begin{align*}
S'_\I &=  \sum_{k=1}^\ell \left((f^2\alpha_k)(f\alpha_k) \left(\sum_{n=1}^\infty y_{k,n}\alpha_k((f^2\alpha_k)(f\alpha_k)\alpha_k)^{n}\right)\right),\\
S'_\II &= \sum_{t=1}^{|\punct|}\left((g^{m_{\beta_t}-1}\beta_t)(g^{m_{\beta_t}-2}\beta_t)\ldots (g\beta_t)
\left(\sum_{n\geq 1}z_{t,n}\beta_t((g^{m_{\beta_t}-1}\beta_t)(g^{m_{\beta_t}-2}\beta_t)\ldots (g\beta_t)\beta_t)^n\right)\right),\\
S'_\III &=  \sum_{\alpha\in Q_1(\tau)} (f^2\alpha)(f\alpha)\nu_\alpha,
\end{align*}
where, for each $\alpha\in Q_1(\tau)$, $\nu_\alpha$ is a possibly infinite 
linear combination of paths of the form $\lambda(f^2\beta)(f\beta)\rho$ as
in the  description of cycles of type III given in 
Definition~\ref{def:3-types-of-cycles} (hence $S'_\III$ involves only cycles of
type III). Define an $R$-algebra automorphism $\psi$ of $\CQ{Q(\tau)}$ 
according to the rule
\[
\beta_t \mapsto \beta_t-
\sum_{n\geq 1} x_t^{-1}z_{t,n}\beta_t\left((g^{m_{\beta_t}-1}\beta_t)(g^{m_{\beta_t}-2}\beta_t)\ldots (g\beta_t)\beta_t\right)^n
\]
for $1 \le t \le |\punct|$.
It is clear that $\psi$ is unitriangular of positive depth. Moreover, we have $\depth(\psi)\geq\short(S'_\II)-r$, where
$r := \max\{ \val_\tau(p) \mid p \in \punct\}$.

Observe that, with the notation of \eqref{eq:A-and-B},
\begin{eqnarray*}
&&\psi(S(\tau,\bx)+S') \\ & = &
\psi(A)+\psi(\sum_{t=1}^{|\punct|} x_t\left((g^{m_{\beta_t}-1}\beta_t)(g^{m_{\beta_t}-2}\beta_t)\ldots (g\beta_t)(\beta_t)\right))
+\psi(S_{\I}')+\psi(S_{\II}')+\psi(S_{\III}')\\
& = &
\psi(A)\\
&&
+\sum_{t=1}^{|\punct|} x_t\left((g^{m_{\beta_t}-1}\beta_t)(g^{m_{\beta_t}-2}\beta_t)\ldots (g\beta_t)\left(\beta_t-
\sum_{n\geq 1} x_t^{-1}z_{t,n}\beta_t\left((g^{m_{\beta_t}-1}\beta_t)(g^{m_{\beta_t}-2}\beta_t)\ldots (g\beta_t)\beta_t\right)^n\right)\right)\\
&&
+\psi(S_{\I}')+\psi(S_{\II}')+\psi(S_{\III}')\\
& = &
\psi(A)\\
&&
+\sum_{t=1}^{|\punct|} x_t\left((g^{m_{\beta_t}-1}\beta_t)(g^{m_{\beta_t}-2}\beta_t)\ldots (g\beta_t)\beta_t\right)\\
&&
-
\sum_{t=1}^{|\punct|} x_t\left((g^{m_{\beta_t}-1}\beta_t)(g^{m_{\beta_t}-2}\beta_t)\ldots (g\beta_t)\sum_{n\geq 1} x_t^{-1}z_{t,n}\beta_t\left((g^{m_{\beta_t}-1}\beta_t)(g^{m_{\beta_t}-2}\beta_t)\ldots (g\beta_t)\beta_t\right)^n\right)\\
&&
+\psi(S_{\I}')+\psi(S_{\II}')+\psi(S_{\III}')\\
& = &
A-A+\psi(A)+B-S_{\II}'+\psi(S_{\I}')+\psi(S_{\II}')+\psi(S_{\III}')\\
& = &
S(\tau,\bx)+\psi(A)-A+\psi(S'_\I)+\psi(S'_\II)-S'_\II+\psi(S'_\III).
\end{eqnarray*}

We claim that
\begin{eqnarray}
\label{eq:lemma-killing-type-II-psi(A)-A}
&&
\text{$\psi(A)-A$ is cyclically equivalent to a potential that involves only 
cycles of type III;}\\
\label{eq:lemma-killing-type-II-psi(S'_I)}
&&
\text{$\psi(S'_\I)$ is cyclically equivalent to the sum of a potential that 
involves only cycles of type I}\\
 \nonumber
&&
\text{with a potential that involves only cycles of type III;}\\
\label{eq:lemma-killing-type-II-psi(S'_II)-S'_II}
&&
\text{$\psi(S'_\II)-S'_\II$ is cyclically equivalent to a potential that 
involves only cycles of type II,}\\
\nonumber
&&
\text{and $\short(\psi(S'_\II)-S'_\II)>\short(S'_\II)$;}\\
\label{eq:lemma-killing-type-II-psi(S'_III)}
&&
\text{$\psi(S'_\III)$ is cyclically equivalent to a potential that involves 
only cycles of type III.}
\end{eqnarray}
To prove these claims we start with some elementary calculations. 
For a fixed $k\in\{1,\ldots,\ell\}$, regardless of whether $(f^2\alpha_k)$, 
$(f\alpha_k)$ or $(\alpha_k)$ belong or not to the set 
$\{\beta_1,\ldots,\beta_{|\punct|}\}$, we can certainly write each 
$\delta\in\{\psi(\alpha_k), \psi(f\alpha_k), \psi(f^2\alpha_k)\}$ 
as the 2-term sum of $\delta$ with a possibly infinite linear combination of 
paths as follows:
\begin{eqnarray}\label{eq:lemma-killing-type-II-psi(f^ialpha_k)}
\ \ \ \psi(\alpha_k) &=&
\alpha_k+
\sum_{n\geq 1} x_{0,n}\alpha_k\left((g^{m_{\alpha_k}-1}\alpha_k)(g^{m_{\alpha_k}-2}\alpha_k)\ldots (g\alpha_k)\alpha_k\right)^n,\\
\nonumber
\ \ \ \psi(f\alpha_k) &=&
(f\alpha_k)+
\sum_{n\geq 1} x_{1,n}(f\alpha_k)\left((g^{m_{f\alpha_k}-1}f\alpha_k)(g^{m_{f\alpha_k}-2}f\alpha_k)\ldots (gf\alpha_k)(f\alpha_k)\right)^n,\\
\nonumber
\ \ \ \psi(f^2\alpha_k) &=&
(f^2\alpha_k)+
\sum_{n\geq 1} x_{2,n}(f^2\alpha_k)\left((g^{m_{f^2\alpha_k}-1}f^2\alpha_k)(g^{m_{f^2\alpha_k}-2}f^2\alpha_k)\ldots (gf^2\alpha_k)(f^2\alpha_k)\right)^n,
\end{eqnarray}
with $x_{0,n},x_{1,n},x_{2,n}\in\mathbb{C}$ for all $n\in\mathbb{Z}_{>0}$. 
Since $\psi\left((f^2\alpha_k)(f\alpha_k)\alpha_k\right)  = 
\psi(f^2\alpha_k)\psi(f\alpha_k)\psi(\alpha_k)$, we can 
use~\eqref{eq:lemma-killing-type-II-psi(f^ialpha_k)} to calculate 
$\psi\left((f^2\alpha_k)(f\alpha_k)\alpha_k\right)$ by plain substitution. 
This substitution expresses $\psi\left((f^2\alpha_k)(f\alpha_k)\alpha_k\right)$ 
as the product of three 2-summand terms. 
When we expand this product, we obtain an expression of 
$\psi\left((f^2\alpha_k)(f\alpha_k)\alpha_k\right)$ as a sum of 8 terms; to 
calculate the $r^{\operatorname{th}}$ power of this 8-term sum for $r\geq 1$, 
we have to form all ordered $r$-tuples of the 8 terms, and for each of these 
tuples, take the product of the tuple's members in the order under which they 
appear in the tuple. There is a total of $8^{r}$ ordered $r$-tuples. Exactly 
one of them consists of the term $(f^2\alpha_k)(f\alpha_k)\alpha_k$ repeated $r$
times. Take any of the remaining $8^{r}-1$ tuples, say $(c_1,\ldots,c_r)$; we 
claim that the product $c_1\ldots c_r$ is a potential cyclically equivalent to 
a potential that involves only cycles of type III. To prove this claim, first 
notice that according to \eqref{eq:lemma-killing-type-II-psi(f^ialpha_k)}, for 
any $\delta\in\{\alpha_k, (f\alpha_k), (f^2\alpha_k)\}$ we can write 
$\psi(f^2\delta)=(f^2\delta)+(f^2\delta)\xi_{\delta,1}$ and 
$\psi(f\delta)=f(\delta)+\xi_{\delta,2}(f\delta)$ for some 
$\xi_{\delta,1},\xi_{\delta_2}\in\CQ{Q(\tau)}$. Now, for at least one index 
$\ell\in\{1,\ldots,r\}$ we have $c_\ell\neq (f^2\alpha_k)(f\alpha_k)\alpha_k$, 
which, using the observation we have just made, means that for some 
$\delta\in\{\alpha_k, (f\alpha_k), (f^2\alpha_k)\}$, the potential 
$c_1\ldots c_r$ is cyclically equivalent to
\begin{eqnarray}\nonumber
&&
(f^2\delta)\xi(f\delta)\left(\sum_{n\geq 1} y_{n}\delta\left((g^{m_{\delta}-1}\delta)(g^{m_{\delta}-2}\delta)\ldots (g\delta)\delta\right)^n\right)\\
\nonumber
&=&
\sum_{n\geq 1} y_{n}(f^2\delta)\xi(f\delta)\delta\left((g^{m_{\delta}-1}\delta)(g^{m_{\delta}-2}\delta)\ldots (g\delta)\delta\right)^n\\
\nonumber
&\sim_{\operatorname{cyc}}&
\sum_{n\geq 1} y_{n}(f\delta)\delta\left((g^{m_{\delta}-1}\delta)(g^{m_{\delta}-2}\delta)\ldots (g\delta)\delta\right)^n(f^2\delta)\xi\\
\label{eq:Sept-referee-revision-showing-type-III}
&=&
\sum_{n\geq 1} y_{n}\ 
\blue{(f\delta)\delta}
\left((g^{m_{\delta}-1}\delta)(g^{m_{\delta}-2}\delta)\ldots 
(g\delta)\delta\right)^{n-1}(g^{m_{\delta}-1}\delta)(g^{m_{\delta}-2}\delta)\ldots 
(g\delta)\ 
\blue{\delta(f^2\delta)}
\xi
\end{eqnarray}
for some $\xi\in\CQ{Q(\tau)}$ and some scalars $y_n\in\mathbb{C}$. Noticing 
that $g^{m_{\delta}-1}\delta=g^{-1}\delta$, we see that the potential 
in~\eqref{eq:Sept-referee-revision-showing-type-III} is a possibly infinite 
linear combination of cycles of type III (the third item in 
Definition~\ref{def:3-types-of-cycles} is verified with $\beta=f^{-1}\delta$ 
and $\alpha=f\delta$ --see the highlighted (in blue) factors 
in~\eqref{eq:Sept-referee-revision-showing-type-III}). This shows that 
$c_1\ldots c_r$ is indeed cyclically equivalent to a potential that involves 
only cycles of type III.

We deduce that:
\begin{itemize}
\item $\psi\left((f^2\alpha_k)(f\alpha_k)\alpha_k\right)-(f^2\alpha_k)(f\alpha_k)\alpha_k$ 
is cyclically equivalent to a potential that involves only cycles of type III 
(take $r=1$ in the expansion above). Therefore, $\psi(A)-A$ is cyclically 
equivalent to a potential that involves only cycles of type III;
\item for $r>1$, 
$\psi\left(\left((f^2\alpha_k)(f\alpha_k)\alpha_k\right)^{r}\right)$ is 
cyclically equivalent to the sum of a potential that involves only cycles of 
type I with a potential that involves only cycles of type III. Therefore, 
$\psi(S'_{\I})$ is cyclically equivalent to the sum of a potential that 
involves only cycles of type I with a potential that involves only cycles of 
type III.
\end{itemize}
This proves our claims \eqref{eq:lemma-killing-type-II-psi(A)-A} 
and~\eqref{eq:lemma-killing-type-II-psi(S'_I)}.

Let us tackle the claim made 
in~\eqref{eq:lemma-killing-type-II-psi(S'_II)-S'_II}. The inequality 
$\short(\psi(S'_\II)-S'_\II)>\short(S'_\II)$ is an immediate consequence of the 
fact that $\psi$ is unitriangular of positive depth, for the unitriangularity 
of $\psi$ implies $\psi(S'_\II)-S'_\II\in\maxid^{\short(S'_{\II})+\depth(\psi)}$ 
(see~\cite[Equation (2.4)]{DWZ1}). For the other part of the claim, fix 
$t\in\{1,\ldots,|\punct|\}$ and $r\in\mathbb{Z}_{>0}$. Then
\begin{eqnarray*}
&&
\psi\left(((g^{m_{\beta_t}-1}\beta_t)(g^{m_{\beta_t}-2}\beta_t)\ldots (g\beta_t)\beta_t)^r\right)\\
&=&
\left((g^{m_{\beta_t}-1}\beta_t)(g^{m_{\beta_t}-2}\beta_t)\ldots (g\beta_t)\left(\beta_t-
\sum_{n\geq 1} x_t^{-1}z_{t,n}\beta_t\left((g^{m_{\beta_t}-1}\beta_t)(g^{m_{\beta_t}-2}\beta_t)\ldots (g\beta_t)\beta_t\right)^n\right)\right)^r\\
&=&
\left((g^{m_{\beta_t}-1}\beta_t)(g^{m_{\beta_t}-2}\beta_t)\ldots (g\beta_t)\beta_t-
\sum_{n\geq 2} x_t^{-1}z_{t,n}\left((g^{m_{\beta_t}-1}\beta_t)(g^{m_{\beta_t}-2}\beta_t)\ldots (g\beta_t)\beta_t\right)^{n}\right)^r
\end{eqnarray*}
From this expression we deduce that, for $r\in\mathbb{Z}_{>0}$, the potential 
$\psi\left(((g^{m_{\beta_t}-1}\beta_t)(g^{m_{\beta_t}-2}\beta_t)\ldots (g\beta_t)\beta_t)^r\right)$ 
is equal to a possibly infinite linear combination of cycles of the form 
$((g^{m_{\beta_t}-1}\beta_t)(g^{m_{\beta_t}-2}\beta_t)\ldots (g\beta_t)\beta_t)^s$ 
with $s\geq r$. Since $S'_{\II}$ is equal to a possibly infinite linear 
combination of cycles of the form 
$((g^{m_{\beta_t}-1}\beta_t)(g^{m_{\beta_t}-2}\beta_t)\ldots (g\beta_t)\beta_t)^r$ 
with $r\geq 2$, it follows that $\psi(S'_{\II})-S'_{\II}$ is a possibly infinite 
linear combination of cycles of the form 
$((g^{m_{\beta_t}-1}\beta_t)(g^{m_{\beta_t}-2}\beta_t)\ldots (g\beta_t)\beta_t)^s$ 
with $s\geq 2$. This proves~\eqref{eq:lemma-killing-type-II-psi(S'_II)-S'_II}.

Let us prove~\eqref{eq:lemma-killing-type-II-psi(S'_III)}. Suppose $\alpha$ 
and $\beta$ are arrows, and $\lambda$, and $\rho$ are paths. Regardless of 
whether the arrows $f^2\beta$, $f\beta$, $g^{-1}f\beta$,
$f^2\alpha$ and $f\alpha$ belong or not to the set 
$\{\beta_1,\ldots,\beta_{|\punct|}\}$, for 
$\gamma\in\{f^2\beta, f\beta, g^{-1}f\beta, f^2\alpha,f\alpha\}$ 
it is possible to write
\[
\psi(\gamma) = \gamma + \gamma\zeta_\gamma\gamma
\]
for some $\zeta_\gamma\in e_{t(\gamma)}\CQ{Q}e_{h(\gamma)}$. Hence,
\begin{eqnarray}\nonumber
&&
\psi((f^2\beta)(f\beta)(g^{-1}f\beta)\lambda(f^2\alpha)(f\alpha)\rho)\\
\nonumber
& = & \left((f^2\beta)+(f^2\beta)\zeta_{f^2\beta}(f^2\beta)\right) \cdot 
\blue{\left((f\beta)+(f\beta)\zeta_{f\beta}(f\beta)\right)}
\cdot
\left((g^{-1}f\beta)+(g^{-1}f\beta)\zeta_{g^{-1}f\beta}(g^{-1}f\beta)\right) \cdot \\
\label{eq:lemma-killing-type-II-5-term-product}
&&
\cdot \ \psi(\lambda) \cdot \left((f^2\alpha)+(f^2\alpha)\zeta_{f^2\alpha}(f^2\alpha)\right)\cdot
\left((f\alpha)+(f\alpha)\zeta_{f\alpha}(f\alpha)\right)\cdot\psi(\rho)
\end{eqnarray}
When we expand this product, we find that 
$\psi((f^2\beta)(f\beta)(g^{-1}f\beta)\lambda(f^2\alpha)(f\alpha)\rho)$ 
is the sum of $2^5=32$ specific terms. These 32 terms arise from the ordered 
$5$-tuples that can be formed by taking, for each factor in the 
product~\eqref{eq:lemma-killing-type-II-5-term-product}, one of the two 
summands of the factor. Indeed, each of the 32 terms is the result of 
multiplying the members of the corresponding 5-tuple. Thus, the referred 32 
terms can be split into two groups: The first group (resp. second group) 
consists of the 16 terms that result from choosing the summand $(f\beta)$ 
(resp. $(f\beta)\zeta_{f\beta}(f\beta)$) of the second factor 
of~\eqref{eq:lemma-killing-type-II-5-term-product} 
(we have highlighted in blue the 
alluded second factor 
in~\eqref{eq:lemma-killing-type-II-5-term-product}). 
The terms in the first group clearly are potentials cyclically equivalent to 
potentials that involve only cycles of type III. Let us show that the terms in 
the second group are as well. These terms are precisely those that appear when 
we expand the product
\begin{eqnarray}\label{eq:lemma-killing-type-II-psi(S'_III)-nonobvious-expansion-1}
&  & \left((f^2\beta)+(f^2\beta)\zeta_{f^2\beta}(f^2\beta)\right) \cdot \left((f\beta)\zeta_{f\beta}(f\beta)\right) \cdot
\left((g^{-1}f\beta)+(g^{-1}f\beta)\zeta_{g^{-1}f\beta}(g^{-1}f\beta)\right) \cdot \\
\nonumber
&&
\cdot \ \psi(\lambda) \cdot \left((f^2\alpha)+(f^2\alpha)\zeta_{f^2\alpha}(f^2\alpha)\right)\cdot
\left((f\alpha)+(f\alpha)\zeta_{f\alpha}(f\alpha)\right)\cdot\psi(\rho)
\end{eqnarray}
Setting $\gamma=f\beta$, we have $\zeta_{\gamma}\in e_{t(\gamma)}\CQ{Q}e_{h(\gamma)}$,
which means that if we take a path $\theta$ that appears with non-zero scalar 
coefficient in the expression of $\zeta_\gamma$ as a possibly infinite linear 
combination of paths, we can write
\begin{equation}\label{eq:lemma-killing-type-II-psi(S'_III)-path-theta}
\theta= (f^{-1}\gamma)(f^{-2}\gamma)\ldots(f^{-s_\theta}\gamma)\theta'
\end{equation}
for some $s_\theta\in\{0,\ldots,\length(\theta)\}$ and some path $\theta'$ that 
either has length zero or can be written as
\[
\theta'=(g^{-1}f^{-s_\theta}\gamma)\theta''
\]
for some path $\theta''$
(the integer $s_\theta$ and the path $\theta'$ can be found by looking at the 
maximum value of $s\in\mathbb{Z}_{\geq 0}$ for which the length-$s$ path 
$(f^{-1}\gamma)(f^{-2}\gamma)\ldots(f^{-s}\gamma)$ appears as a left-most subpath
of $\theta$; note that we are using the fact that for each vertex of 
$Q=Q(\tau)$ there are precisely two arrows of $Q=Q(\tau)$ whose head is the 
given vertex; note also that $s_\theta=0$ means that the arrow $f^{-1}\gamma$ 
does not appear as the left-most arrow in the path $\theta$ and hence 
$g^{-1}\gamma$ does, being $\theta=\theta'$). Therefore, for a fixed $\theta$, 
the potential
\begin{eqnarray}\nonumber
\omega_\theta & = & \left((f^2\beta)+(f^2\beta)\zeta_{f^2\beta}(f^2\beta)\right) \cdot (f\beta)\cdot\theta\cdot(f\beta) \cdot
\left((g^{-1}f\beta)+(g^{-1}f\beta)\zeta_{g^{-1}f\beta}(g^{-1}f\beta)\right) \cdot \\
\label{eq:lemma-killing-type-II-psi(S'_III)-one-term-of-expansion}
&&
\cdot \ \psi(\lambda) \cdot \left((f^2\alpha)+(f^2\alpha)\zeta_{f^2\alpha}(f^2\alpha)\right)\cdot
\left((f\alpha)+(f\alpha)\zeta_{f\alpha}(f\alpha)\right)\cdot\psi(\rho)
\end{eqnarray}
is either equal to
\begin{eqnarray*}
&  & \text{{\small$\left((f^2\beta)+(f^2\beta)\zeta_{f^2\beta}(f^2\beta)\right) \cdot (f\beta)\cdot(f^{-1}\gamma)(f^{-2}\gamma)\ldots(f^{-s_\theta}\gamma)\cdot(f\beta)\cdot$}}\\
&&
\text{{\small$\cdot
\left((g^{-1}f\beta)+(g^{-1}f\beta)\zeta_{g^{-1}f\beta}(g^{-1}f\beta)\right)
\cdot \ \psi(\lambda) \cdot \left((f^2\alpha)+(f^2\alpha)\zeta_{f^2\alpha}(f^2\alpha)\right)\cdot
\left((f\alpha)+(f\alpha)\zeta_{f\alpha}(f\alpha)\right)\cdot\psi(\rho)$}}
\end{eqnarray*}
(if the path $\theta'$ 
in~\eqref{eq:lemma-killing-type-II-psi(S'_III)-path-theta} has length zero) or 
it is equal to
\begin{eqnarray*}
&  & \text{{\small$\left((f^2\beta)+(f^2\beta)\zeta_{f^2\beta}(f^2\beta)\right) \cdot (f\beta)\cdot(f^{-1}\gamma)(f^{-2}\gamma)\ldots(f^{-s_\theta}\gamma)(g^{-1}f^{-n_\theta}\gamma)\theta''\cdot(f\beta)\cdot$}}\\
&&
\text{{\small$\cdot
\left((g^{-1}f\beta)+(g^{-1}f\beta)\zeta_{g^{-1}f\beta}(g^{-1}f\beta)\right)
\cdot \ \psi(\lambda) \cdot \left((f^2\alpha)+(f^2\alpha)\zeta_{f^2\alpha}(f^2\alpha)\right)\cdot
\left((f\alpha)+(f\alpha)\zeta_{f\alpha}(f\alpha)\right)\cdot\psi(\rho)$}}.
\end{eqnarray*}
In either case, if we recall that we have set $\gamma=f\beta$, that $\theta$ 
is a path going from $h(f\beta)$ to $t(f\beta)$, and that $f^3\delta=\delta$ 
for every arrow $\delta$, we deduce that the 
potential~\eqref{eq:lemma-killing-type-II-psi(S'_III)-one-term-of-expansion} 
is cyclically equivalent to a potential that involves only cycles of type III.

If we write $\zeta_{f\beta}$ as the possibly infinite linear combination of paths
$\sum_{\theta}u_\theta\theta$, with $u_\theta\in\mathbb{C}\setminus\{0\}$ for every
path $\theta$, then the 
product~\eqref{eq:lemma-killing-type-II-psi(S'_III)-nonobvious-expansion-1} is 
equal to the convergent series $\sum_{\theta}u_\theta\omega_\theta$, where 
$\omega_\theta$ is given 
by~\eqref{eq:lemma-killing-type-II-psi(S'_III)-one-term-of-expansion} for each 
$\theta$. Therefore, the 
product~\eqref{eq:lemma-killing-type-II-psi(S'_III)-nonobvious-expansion-1} is 
cyclically equivalent to a potential that involves only cycles of type III. We 
deduce that 
$\psi((f^2\beta)(f\beta)(g^{-1}f\beta)\lambda(f^2\alpha)(f\alpha)\rho)$ is 
cyclically equivalent to a potential that involves only cycles of type III.

The fact that 
$\psi((f^2\beta)(f\beta)\lambda(f^2\alpha)(f\alpha)\rho(gf^2\beta))$ is also 
cyclically equivalent to a potential that involves only cycles of type III can 
be proved in a similar algebraic-combinatorial fashion. Our 
claim~\eqref{eq:lemma-killing-type-II-psi(S'_III)} follows.

Let us summarize. Given $S'=S'_\I +S'_\II +S'_\III$, we can guarantee the 
existence of an $R$-algebra automorphism $\psi$ of $\CQ{Q(\tau)}$ such that
\begin{itemize}
\item $\psi$ is unitriangular and $\depth(\psi)\geq \short(S'_\II)-r$;
\item $\psi(S(\tau,\bx)+S')$ is cyclically equivalent to $S(\tau,\bx)+S''$ for 
some potential $S''$ with the property that, when written in the form 
$S''=S''_\I+S''_\II+S''_\III$, it satisfies $\short(S''_\II)>\short(S'_\II)$.
\end{itemize}
Therefore, the lemma follows from a limit process.
\end{proof}

\begin{lemma}\label{lemma:killing-type-I}
Let $\surf$ be a marked surface with empty boundary, and let
$\tau$ be a tagged triangulation of $\surf$ satisfying \eqref{eq:valency>3}, 
\eqref{eq:tau-has-no-loops} and~\eqref{eq:Q(tau)-has-no-double-arrows}. 
If $W\in \CQ{Q(\tau)}$ is a potential involving only cycles of types I and III,
then $(Q(\tau),S(\tau,\bx)+W)$ is right equivalent to $(Q(\tau),S(\tau,\bx)+U)$ 
for some potential $U$ involving only cycles of type III.
\end{lemma}

\begin{proof}
By Definition \ref{def:3-types-of-cycles}, we can write $W = W_\I + W_\III$ with
\begin{eqnarray*}
W_\I & = & \sum_{k=1}^\ell \left((f^2\alpha_k)(f\alpha_k) \left(\sum_{n=1}^\infty y_{k,n}\alpha_k((f^2\alpha_k)(f\alpha_k)\alpha_k)^{n}\right)\right)\\
W_\III & = & \sum_{\alpha\in Q_1(\tau)} (f^2\alpha)(f\alpha)\nu_\alpha,
\end{eqnarray*}
such that for each $\alpha\in Q_1(\tau)$, $\nu_\alpha$ is a possibly infinite 
linear combination of paths of the form $\lambda(f^2\beta)(f\beta)\rho$ as in 
the above description of cycles of type III.
Define an $R$-algebra automorphism $\varphi$ of $\CQ{Q(\tau)}$ according to 
the rule
\[
\alpha_k\mapsto \alpha_k-\sum_{n=1}^\infty y_{k,n}\alpha_k\left((f^2\alpha_k)(f\alpha_k)\alpha_k\right)^{n}
\]
for $1 \le k \le \ell$.
It is clear that $\varphi$ is unitriangular with $\depth(\varphi) = \short(W_\I)-3>0$.

Direct computation yields
\begin{eqnarray*}
\varphi(S(\tau,\bx)+W) & \sim_{\operatorname{cyc}} &
S(\tau,\bx)+\varphi(B)-B+\varphi(W_\I)-W_\I+\varphi(W_\III).
\end{eqnarray*}
Note that
\begin{itemize}

\item
$\varphi(B)-B$ involves only cycles of type III;

\item
$\varphi(W_\I)-W_\I$ involves only cycles of type I, and 
$\short(\varphi(W_\I)-W_\I)>\short(W_\I)$ since $\varphi$ has positive depth;

\item
$\varphi(W_\III)$ involves only cycles of type III.

\end{itemize}

Summarizing, given $W = W_\I+W_\III$, we can guarantee the existence of an 
$R$-algebra automorphism $\varphi$ of $\CQ{Q(\tau)}$ such that
the following hold:
\begin{itemize}

\item
$\varphi$ is unitriangular and $\depth(\varphi)= \short(W_\I)-3$;

\item
$\varphi(S(\tau,\bx)+W)$ is cyclically equivalent to $S(\tau,\bx)+W'$ for some 
potential $W'$ with the property that, when written in the form
$W' = W'_\I+W'_\II+W'_\III$, it satisfies $W'_\II=0$ and 
$\short(W'_\I)>\short(W_\I)$.
\end{itemize}
Therefore, the lemma follows from a limit process.
\end{proof}

\begin{lemma}\label{lemma:killing-type-III} Let $\surf$ be a surface with empty 
boundary, and let $\tau$ be a (tagged) triangulation of $\surf$ 
satisfying~\eqref{eq:valency>3}, \eqref{eq:tau-has-no-loops} 
and~\eqref{eq:Q(tau)-has-no-double-arrows}. If $U\in \CQ{Q(\tau)}$ is a 
potential involving only cycles of type III, then $(Q(\tau),S(\tau,\bx)+U)$ is 
right equivalent to $(Q(\tau),S(\tau,\bx))$.
\end{lemma}

\begin{proof} Suppose $U\in \CQ{Q(\tau)}$ is a potential involving only cycles 
of type III. Then Definition~\ref{def:3-types-of-cycles} tells us that
\begin{itemize}

\item[($\star$)]
every cycle appearing in $U$ is a cycle of the form
\[
(f^2\beta)(f\beta)\lambda(f^2\alpha)(f\alpha)\rho
\]
for some arrows $\alpha,\beta$,
and some paths $\lambda,\rho$ such that $\lambda=(g^{-1}f\beta)\lambda'$ or 
$\rho=\rho'(gf^2\beta)$.

\end{itemize}
For each cycle appearing in $U$ we choose exactly one such expression 
$(f^2\beta)(f\beta)\lambda(f^2\alpha)(f\alpha)\rho$. Once this choice has been 
made, we see that
\begin{eqnarray*}
U & \sim_{\operatorname{cyc}} & \sum_{\alpha\in Q_1(\tau)} (f^2\alpha)(f\alpha)\nu_\alpha,
\end{eqnarray*}
where, for each $\alpha\in Q_1(\tau)$, $\nu_\alpha$ is a possibly infinite linear
combination of paths of the form $\rho(f^2\beta)(f\beta)\lambda$ as above. 
Define an $R$-algebra automorphism $\eta$ of $\CQ{Q(\tau)}$ according to the 
rule
\[
\alpha\mapsto \alpha-\nu_\alpha, \ \ \ \alpha\in Q_1(\tau).
\]
It is clear that $\eta$ is unitriangular with $\depth(\eta)=\short(U)-3>0$.

The key observation here is that $\eta(A)\sim_{\operatorname{cyc}}A-U+U'$, with $U'$
a potential involving only cycles of type III, and such that 
$\short(U')>\short(U)$. To see this, note that
\begin{eqnarray}\nonumber
\sum_{\alpha\in Q_1(\tau)} (f^2\alpha)(f\alpha)\nu_\alpha & \sim_{\operatorname{cyc}} & \sum_{k=1}^\ell\left(
(f^2\alpha_k)(f\alpha_k)(\nu_{\alpha_k})
+(f^2\alpha_k)(\nu_{f\alpha_k})(\alpha_k)
+(\nu_{f^2\alpha_k})(f\alpha_k)(\alpha_k)
\right)
\end{eqnarray}
and
\begin{align*}
\eta(A)&=\sum_{k=1}^\ell (f^2\alpha_k-\nu_{f^2\alpha_k})(f\alpha_k-\nu_{f\alpha_k})(\alpha_k-\nu_{\alpha_k})\\
\nonumber
&=\sum_{k=1}^\ell\left(
(f^2\alpha_k)(f\alpha_k)(\alpha_k)\right)\\
\nonumber
&-\sum_{k=1}^\ell\left(
(f^2\alpha_k)(f\alpha_k)(\nu_{\alpha_k})
+(f^2\alpha_k)(\nu_{f\alpha_k})(\alpha_k)
+(\nu_{f^2\alpha_k})(f\alpha_k)(\alpha_k)
\right)\\
\nonumber
&+\sum_{k=1}^\ell\left(
(f^2\alpha_k)(\nu_{f\alpha_k})(\nu_{\alpha_k})
+(\nu_{f^2\alpha_k})(f\alpha_k)(\nu_{\alpha_k})
+(\nu_{f^2\alpha_k})(\nu_{f\alpha_k})(\alpha_k)
-(\nu_{f^2\alpha_k})(\nu_{f\alpha_k})(\nu_{\alpha_k})
\right).
\end{align*}
By ($\star$) and Lemma \ref{lemma:3-types-of-cycles}, the potential
\[
\sum_{k=1}^\ell\left(
(f^2\alpha_k)(\nu_{f\alpha_k})(\nu_{\alpha_k})
+(\nu_{f^2\alpha_k})(f\alpha_k)(\nu_{\alpha_k})
+(\nu_{f^2\alpha_k})(\nu_{f\alpha_k})(\alpha_k)
-(\nu_{f^2\alpha_k})(\nu_{f\alpha_k})(\nu_{\alpha_k})
\right)
\]
involves only cycles that are rotationally equivalent to cycles of type III.

Direct computation yields
\begin{eqnarray*}
\eta(S(\tau,\bx)+U) & \sim_{\operatorname{cyc}}
S(\tau,\bx)+\eta(B)-B+\eta(U)-U+U'.
\end{eqnarray*}
Note that
\begin{itemize}
\item $\eta(B)-B$ involves only cycles of type III (this follows from
($\star$) and Lemma~\ref{lemma:3-types-of-cycles} when we expand $\eta(B)$ in 
a way similar to the way we expanded $\eta(A)$ above), and 
$\short(\eta(B)-B)\geq\short(B)+\depth(\eta)=\short(B)+\short(U)-3>\short(U)$ 
since $\tau$ satisfies~\eqref{eq:valency>3};
\item $\eta(U)-U$ involves only cycles of type III (this follows from
($\star$) and Lemma~\ref{lemma:3-types-of-cycles} when we expand $\eta(U)$ in 
a way similar to the way we expanded $\eta(A)$ above),
and $\short(\eta(U)-U)>\short(U)$ since the depth of $\eta$ is positive.
\end{itemize}

Summarizing, given $U=U_3$, we can guarantee the existence of an $R$-algebra 
automorphism $\eta$ of $\CQ{Q(\tau)}$ such that
\begin{itemize}
\item $\eta$ is unitriangular and $\depth(\eta)= \short(U)-3$;
\item $\varphi(S(\tau,\bx)+U)$ is cyclically equivalent to $S(\tau,\bx)+U'$ 
for some potential $U'$ with the property that, when written in the form 
$U'=U'_1+U'_2+U'_3$, it satisfies $U'_1=0=U'_2$ and $\short(U')>\short(U)$.
\end{itemize}
Therefore, the lemma follows from a limit process.
\end{proof}

The following result is a direct consequence of
Lemmas \ref{lemma:killing-type-II}, \ref{lemma:killing-type-I} 
and~\ref{lemma:killing-type-III}.

\begin{prop}\label{prop:killing-S'}
Let $\surf$ be a marked surface with empty boundary, and let $\tau$ be a tagged 
triangulation of $\surf$ satisfying \eqref{eq:valency>3}, 
\eqref{eq:tau-has-no-loops} and~\eqref{eq:Q(tau)-has-no-double-arrows}.
Then, for any potential $S'\in\CQ{Q(\tau)}$ rotationally disjoint from 
$S(\tau,\bx)$, the QP $(Q(\tau),S(\tau,\bx) + S')$ is right equivalent to 
$(Q(\tau),S(\tau,\bx))$.
\end{prop}

We now give a sufficient condition on $\surf$ for the existence of a 
triangulation $\tau$ satisfying \eqref{eq:valency>3}, 
\eqref{eq:tau-has-no-loops} and~\eqref{eq:Q(tau)-has-no-double-arrows}.

\begin{lemma}\label{lemma8.11}
Let $\surf$ be a marked surface with empty boundary.
Assume that
\[
|\marked| \geq
\begin{cases}
6 & \text{if $\Sigma$ is a sphere},\\
3 & \text{otherwise}.
\end{cases}
\]
Then there exists a triangulation $\tau$ of $\surf$
satisfying the conditions~\eqref{eq:valency>3}, 
\eqref{eq:tau-has-no-loops} and~\eqref{eq:Q(tau)-has-no-double-arrows}.
\end{lemma}

\begin{proof}
If $\surf$ is a sphere with $|\marked| \ge 6$, the triangulation
obtained from Figure~\ref{Fig:new_sphere_6puncts.eps} in combination
with the procedure of adding punctures as described in
Figure~\ref{Fig:adding_puncture} satisfies~\eqref{eq:valency>3}, 
\eqref{eq:tau-has-no-loops} and~\eqref{eq:Q(tau)-has-no-double-arrows}.

In the case of positive-genus surfaces, we shall prove the lemma by induction 
on the genus of $\Sigma$. In Figure~\ref{Fig:2nd_version_torus_3puncts_noloops} 
we see a triangulation $\sigma$ of the torus with exactly three punctures 
$(\mathbb{T},\{p_1,p_2,p_3\})$.
\begin{figure}[!h]
                \centering
                \includegraphics[scale=.65]{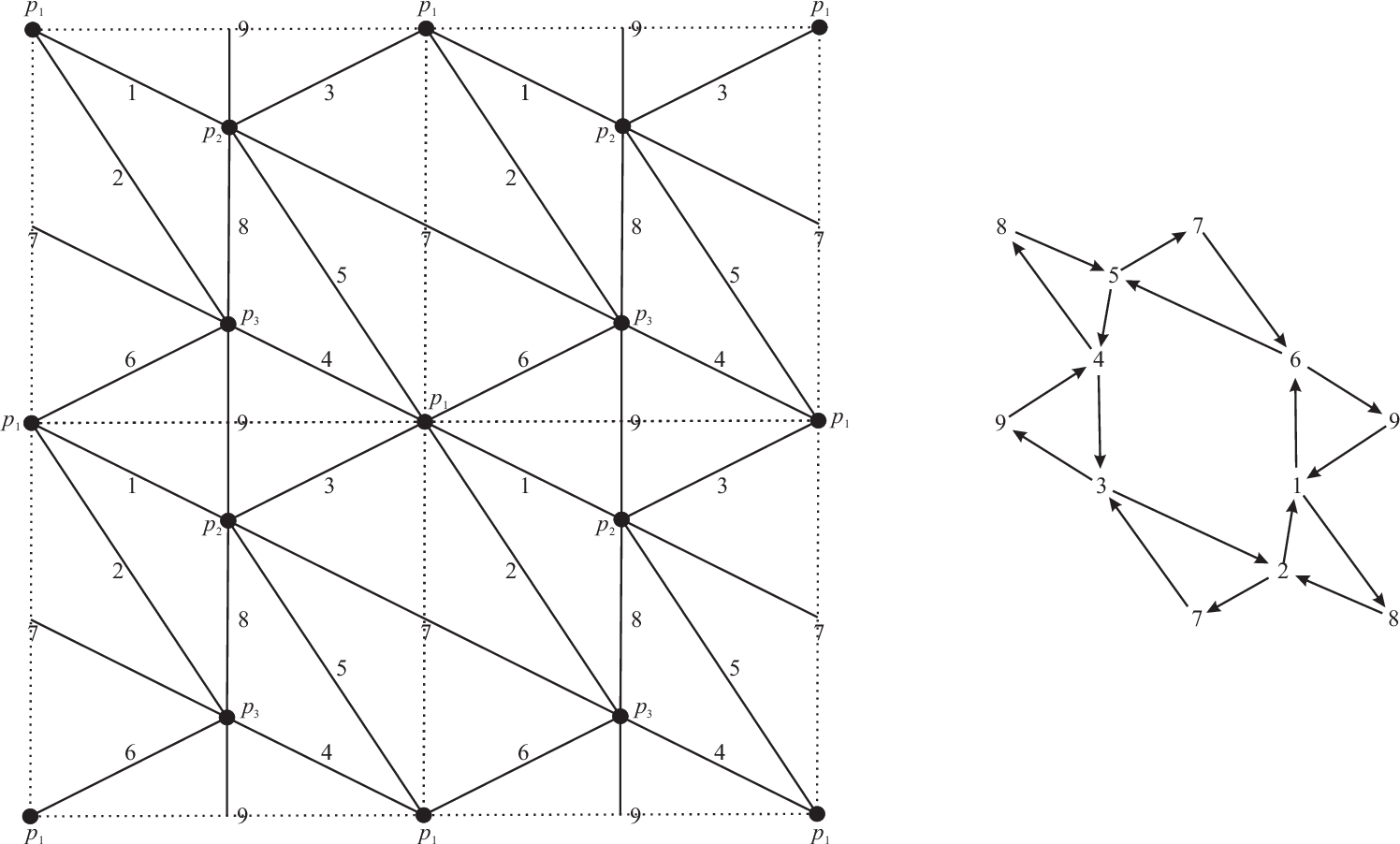}\caption{Triangulation of the 3-punctured torus satisfying  \eqref{eq:valency>3},  \eqref{eq:tau-has-no-loops} and  \eqref{eq:Q(tau)-has-no-double-arrows}}\label{Fig:2nd_version_torus_3puncts_noloops}
        \end{figure}
Straightforward inspection shows that every puncture is incident to at least 4 
arcs in $\sigma$, i.e. that $\sigma$ satisfies~\eqref{eq:valency>3}. It is also 
clear that every arc in $\sigma$ connects punctures $p_i$, $p_j$, with 
$p_i\neq p_j$, i.e. that $\sigma$ satisfies~\eqref{eq:tau-has-no-loops}. 
Clearly, the quiver $Q(\sigma)$, drawn in 
Figure~\ref{Fig:2nd_version_torus_3puncts_noloops} as well, does not have 
double arrows, i.e. $\sigma$ satisfies \eqref{eq:Q(tau)-has-no-double-arrows}. 
Taking this $\sigma$ as starting point, if we now begin adding punctures and 
arcs as indicated in Figure \ref{Fig:adding_puncture}, we obtain a triangulation
satisfying~\eqref{eq:valency>3}, \eqref{eq:tau-has-no-loops} 
and~\eqref{eq:Q(tau)-has-no-double-arrows} for the closed torus with $n\geq3$ 
punctures. The assertion of the lemma is thus proved when the genus of $\Sigma$ 
is 1.

Now, let $(\Sigma',\marked')$ be a positive-genus surface with empty boundary 
and at least three punctures, and suppose that $(\Sigma',\marked')$ has a 
triangulation $\tau'$ satisfying \eqref{eq:valency>3}, 
\eqref{eq:tau-has-no-loops} and \eqref{eq:Q(tau)-has-no-double-arrows}. 
Let $\triangle_1$ be any ideal triangle of $\tau'$ (recall that, by definition,
$\triangle_1$ is the topological closure in $\Sigma'$ of a connected component 
of the complement in $\Sigma'$ of the union of all arcs in $\tau'$). Since no 
arc in $\tau'$ is a loop, $\triangle_1$ is homeomorphic to a closed disc. 
Similarly, if we take a triangle $\triangle_2$ of the triangulation $\sigma$ 
defined in the first paragraph of the ongoing proof, then $\triangle_2$ is 
homeomorphic to a closed disc.

The sides of $\triangle_1$ (resp. $\triangle_2$) are three different arcs of 
$\tau'$ (resp. $\sigma$) that are not loops. Let $j_1$, $j_2$ and $j_3$ 
(resp. $k_1$, $k_2$ and $k_3$) be the arcs in $\tau'$ (resp. $\sigma$) that are 
the sides of $\triangle_1$ (resp. $\triangle_2$), ordered along the clockwise 
orientation of $\triangle_1$ (resp. the counterclockwise orientation of 
$\triangle_2$).
Let $\triangle_1^{\circ}$ (resp. $\triangle_2^\circ$) be the interior of 
$\triangle_1$ in $\Sigma'$ (resp. of $\triangle_2$ in the torus), and glue 
$\Sigma'\setminus(\triangle_1^\circ)$ to 
$\mathbb{T}'\setminus(\triangle_2^\circ)$ by glueing $j_1$ with $k_1$, $j_2$ with
$k_2$, and $j_3$ with $k_3$. The result of this glueing is the well-known 
connected sum $\Sigma'\#\mathbb{T}$, which is a surface whose genus exceeds by 
one the genus of $\Sigma'$.

Since $\marked'\subseteq\Sigma'\setminus(\triangle_1^\circ)$ we can view 
$\marked'$ as a set of marked points in $\Sigma:=\Sigma'\#\mathbb{T}$ via the 
canonical inclusion $(\Sigma'\setminus\triangle_1^\circ)\hookrightarrow\Sigma$. 
Note that under the inclusion 
$(\mathbb{T}\setminus\triangle_2^\circ)\hookrightarrow\Sigma$, the marked points 
$p_1,p_2,p_3$ of $(\mathbb{T},\{p_1,p_2,p_3\})$ are mapped to points that lie 
already in the image of $\marked'$ under the inclusion 
$\Sigma'\hookrightarrow\Sigma$.

Under the inclusions of $\Sigma'\setminus\triangle_1^\circ$ and 
$\mathbb{T}\setminus\triangle_2^\circ$ in $\Sigma$, we can view both $\tau'$ and 
$\sigma$ as sets of arcs on $\Sigma$. Their union $\tau:=\tau'\cup\sigma$ is 
then a triangulation of $\Sigma$ and satisfies~\eqref{eq:valency>3}, 
\eqref{eq:tau-has-no-loops} and~\eqref{eq:Q(tau)-has-no-double-arrows}, as the 
reader can easily verify. This finishes the inductive step of our proof.
\end{proof}

\begin{lemma}\label{lemma8.12}
Let $\surf$ be a marked surface with empty boundary, and let $\tau$ be a 
tagged triangulation of $\surf$  satisfying~\eqref{eq:valency>3}, 
\eqref{eq:tau-has-no-loops} and~\eqref{eq:Q(tau)-has-no-double-arrows}.
Each non-degenerate potential $S$ on $Q(\tau)$ is of the form 
$S = S(\tau,\bx) + S'$, where $S'$ is rotationally disjoint from
$S(\tau,\bx)$.
\end{lemma}

\begin{proof}
For $p \in \punct$ let $Q_p(\tau)$ be the full subquiver of $Q(\tau)$
whose vertices correspond to the arcs in $\tau$ incident to $p$.
It is clear that $Q_p(\tau)$ contains the
canonical cycle
$S_p := \alpha_1^p \alpha_2^p \cdots \alpha_{d_p}^p$
around $p$.
Since $\tau$ does not contain any loops, we have
$d_p = \val_\tau(p)$.
Now suppose that there is another arrow $\alpha$ in $Q_p(\tau)$.
This implies that two of the arcs incident to $p$, say $k$ and $l$,
are both incident to a puncture $q$ different from $p$.
Each triangle of $\tau$ is an interior triangle, since $\Sigma$ has
an empty boundary.
Thus $k$, $l$ and a third arc $m$ form the sides of a triangle.
But now $m$ has to be a loop, a contradiction.
It follows that $Q_p(\tau)$ is isomorphic to a cyclically oriented
quiver of type $\widetilde{A}_{d_p-1}$.
Now let $S$ be a non-degenerate potential on $Q(\tau)$.
Since $\tau$ is a gentle triangulation, we know that $S$ contains
all $3$-cycles in $Q(\tau)$.
Furthermore, by Proposition~\ref{proptcycle}, the cycles $S_p$ with
$p \in \punct$ also appear in
$S$.
Now the result follows from the definition of $S(\tau,\bx)$.
\end{proof}

Combining Proposition~\ref{prop:killing-S'} and Lemmas~\ref{lemma8.11} 
and~\ref{lemma8.12} yields Theorem~\ref{thmemptybdpot}.

\subsection{The sphere with five punctures}\label{sec5sphere}

\begin{prop}\label{prop:sp5-uniqueness-of-potentials}
Let $\sigma$ be a triangulation of a
sphere $\surf$ with $|\marked| = 5$.
Then $Q(\sigma)$ admits only one non-degenerate potential up to
weak right equivalence.
\end{prop}

Since QP-mutation maps weakly right equivalent QPs to QPs that are again weakly 
right equivalent, to prove Proposition \ref{prop:sp5-uniqueness-of-potentials} 
it suffices to show the existence of a triangulation $\tau$ whose quiver 
$Q(\tau)$ admits only one non-degenerate potential up to weak right equivalence.
We choose $\tau$ to be the triangulation depicted in 
Figure~\ref{Fig:new_sp5_uniqueness_of_potentials.eps}.
\begin{figure}[!htb]
                \centering
          \includegraphics[scale=.7]{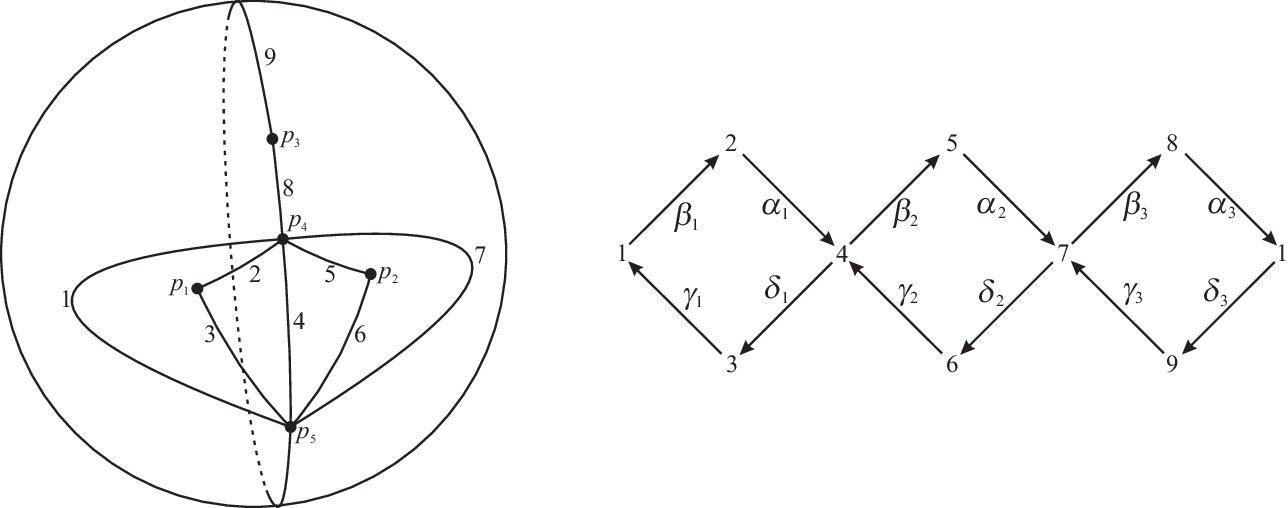}
\caption{
A triangulation of a sphere with 5 punctures.}
\label{Fig:new_sp5_uniqueness_of_potentials.eps}
\end{figure}
Let $\bx=(x_1,x_2,x_3,x_4,x_5)$ be a tuple of non-zero scalars. Then we have
\[
S(\tau,\bx)=-x_1^{-1}\alpha_1\beta_1\gamma_1\delta_1
-x_2^{-1}\alpha_2\beta_2\gamma_2\delta_2
-x_3^{-1}\alpha_3\beta_3\gamma_3\delta_3
+x_4\alpha_3\beta_3\alpha_2\beta_2\alpha_1\beta_1
+x_5\gamma_3\delta_3\gamma_2\delta_2\gamma_1\delta_1.
\]

The fact that $Q(\tau)$ admits only one non-degenerate potential up to weak 
right equivalence is a consequence of the following lemma, whose proof can be 
achieved in steps that are similar to the proofs of 
Lemmas~\ref{lemma:3-types-of-cycles}, \ref{lemma:killing-type-II}, 
\ref{lemma:killing-type-I}, \ref{lemma:killing-type-III} and~\ref{lemma8.12}.

\begin{lemma}
\begin{enumerate}
\item
Let
$A=-x_1^{-1}\alpha_1\beta_1\gamma_1\delta_1
-x_2^{-1}\alpha_2\beta_2\gamma_2\delta_2
-x_3^{-1}\alpha_3\beta_3\gamma_3\delta_3$.
Every cycle on $Q(\tau)$ rotationally disjoint from $S(\tau,\bx)$ is rotation 
equivalent to a cycle of one of the following types:
\begin{itemize}

\item[(I)]
$(\alpha_i\beta_i\gamma_i\delta_i)^n$ for some $n>1$ and some $i=1,2,3$;

\item[(II)]
$(\alpha_3\beta_3\alpha_2\beta_2\alpha_1\beta_1)^n$ or 
$(\gamma_3\delta_3\gamma_2\delta_2\gamma_1\delta_1)^n$ for some $n>1$;

\item[(III)]
$\partial_{\varepsilon_1}(A)\lambda\partial_{\varepsilon_2}(A)\rho$ for some arrows 
$\varepsilon_1,\varepsilon_2$, and some paths $\lambda,\rho$, such that 
$\lambda=\eta\lambda'$ or $\rho=\rho'\eta$ for some arrow 
$\eta\neq\varepsilon_1$.
\end{itemize}
\item
If $S'\in\CQ{Q(\tau)}$ is a potential rotationally disjoint from $S(\tau,\bx)$, 
then $(Q(\tau),S(\tau,\bx)+S')$ is right equivalent to $(Q(\tau),S(\tau,\bx)+W)$
for some potential $W\in \CQ{Q(\tau)}$ involving only cycles of types I and III.
\item
If $W\in \CQ{Q(\tau)}$ is a potential involving only cycles of types I and III, 
then $(Q(\tau),S(\tau,\bx)+W)$ is right equivalent to $(Q(\tau),S(\tau,\bx)+U)$ 
for some potential $U$ involving only cycles of type III.
\item
If $U\in \CQ{Q(\tau)}$ is a potential involving only cycles of type III, then 
$(Q(\tau),S(\tau,\bx)+U)$ is right equivalent to $(Q(\tau),S(\tau,\bx))$.
\item Every non-degenerate potential on $Q(\tau)$ is weakly right equivalent to 
$S(\tau,\bx)+S'$ for some potential $S'$ rotationally disjoint from 
$S(\tau,\bx)$.
\end{enumerate}
\end{lemma}

For the following corollary, we remind the reader that, by definition, the 
quiver $Q(\tau)$ of a tagged triangulation $\tau$ is isomorphic to the quiver 
of certain ideal triangulation $\tau^\circ$ that can be obtained from $\tau$ by 
a careful ``deletion of notches". Under such quiver isomorphism, given a choice 
$\bx=(x_p)_{p\in\punct}$ of non-zero scalars, the potential $S(\tau,\bx)$ can be 
seen to correspond to the potential $S(\tau^\circ,\by)$ for a choice 
$\by=(y_p)_{p\in\punct}$ obtained from $\bx$ by suitably multiplying some of the 
scalars $x_p$ by $-1$. See \cite[Definition 3.2]{Labardini4}.

\begin{coro}\label{coro:flip<->QP-mut-for-sp5}
Let $\tau$ and $\sigma$ be tagged triangulations of a sphere $\surf$
with $|\marked| = 5$.
If $\tau$ and $\sigma$ are related by the flip of a tagged arc $i\in\tau$, 
then the QPs $\mu_i(Q(\tau),S(\tau,\bx))$ and $(Q(\sigma),S(\sigma,\bx))$ are 
weakly right equivalent, where $\bx$ is an a priori fixed tuple of non-zero 
scalars attached to the five punctures of $\surf$.
Consequently, for any tagged triangulation $\tau$ of the five-punctured sphere, 
the QP $(Q(\tau),S(\tau,\bx))$ is non-degenerate.
\end{coro}

\begin{proof}
Let $T$ be the ideal triangulation depicted in 
Figure~\ref{Fig:new_sp5_uniqueness_of_potentials.eps}. By 
Proposition~\ref{prop:sp5-uniqueness-of-potentials}, any non-degenerate 
potential on $Q(T)$ is weakly right equivalent to $S(T,\bx)$. Since $Q(T)$ 
admits a non-degenerate potential over $\ka$ (cf. \cite{DWZ1}), and since 
potentials weakly right equivalent to non-degenerate ones are themselves 
non-degenerate, this implies that $S(T,\bx)$ is non-degenerate.

Suppose $T'$ is an ideal triangulation of the sphere with five punctures. Then 
$T'$ and $T$ are connected by a sequence of flips involving only ideal 
triangulations. Since ideal triangulations related by a flip have QPs related 
by QP-mutation, and since $S(T,\bx)$ is non-degenerate, we deduce that 
$S(T',\bx)$ is non-degenerate. Given the definition of the QPs associated to 
tagged triangulations (cf. \cite[Definition 3.2]{Labardini4}), this implies that
both $S(\tau,\bx)$ and $S(\sigma,\bx)$ are non-degenerate (recall that $\tau$ 
and $\sigma$ have been assumed to be tagged triangulations).

Let us denote the underlying potential of the QP $\mu_i(Q(\tau),S(\tau,\bx))$ 
by $\mu_i(S(\tau,\bx))$. Since $S(\tau,\bx)$ is non-degenerate, 
$\mu_i(S(\tau,\bx))$ is a non-degenerate potential on the quiver 
$\mu_i(Q(\tau))=Q(\sigma)$. By 
Proposition~\ref{prop:sp5-uniqueness-of-potentials}, $\mu_i(S(\tau,\bx))$ and 
$S(\sigma,\bx)$ are weakly right equivalent. 
Corollary~\ref{coro:flip<->QP-mut-for-sp5} is proved.
\end{proof}

In the situation of Corollary~\ref{coro:flip<->QP-mut-for-sp5}, we conjecture 
that the QPs $\mu_i(Q(\tau),S(\tau,\bx))$ and $(Q(\sigma),S(\sigma,\bx))$ are 
right equivalent and not only weakly right equivalent.

It was proved in \cite{Labardini1} that for ideal triangulations $\tau$ and 
$\sigma$ of the 5-punctured sphere that are related by the flip of an arc 
$i\in\tau$, the QPs $\mu_i(Q(\tau),S(\tau,\bx))$ and $(Q(\sigma),S(\sigma,\bx))$
are right equivalent (hence weakly right equivalent). However, \cite{Labardini1}
does not deal with the case of flips involving arbitrary tagged triangulations 
of the sphere with five punctures.
This case is not dealt with either in the more recent paper~\cite{Labardini4}.

\subsection{A uniqueness criterion for non-degenerate potentials}
\label{sec9}
The main result of this subsection is
Theorem~\ref{prop:only-one-potential}, which gives a criterion for a
potential to be the only non-degenerate potential, up to right equivalence,
on a given quiver.

For the following lemma, we use the convention that $\maxid^\infty=0$.

\begin{lemma}[{\cite[Equation (2.4)]{DWZ1}}]
\label{lemma:depth-is-short-cycle-friendly}
If $\vph$ is a unitriangular $R$-algebra automorphism of
$\CQ{Q}$, then for every $n\geq 0$ and every
$u\in\maxid^n$ we have $\vph(u)-u\in\maxid^{n+\depth(\vph)}$.
\end{lemma}

The following lemma strengthens~\cite[Lemma 2.4]{Labardini4}, which in turn
appears implicitly in \cite{DWZ1}.

\begin{lemma}\label{lemma:well-defined-limit-automorphism}
Let $Q$ be a
quiver, and let $(\psi_n)_{n>0}$ be a sequence of unitriangular $R$-algebra
automorphisms of $\CQ{Q}$. Suppose that $\lim_{n\to\infty}\depth(\psi_n)=\infty$.
Then the following hold:
\begin{itemize}

\item[(i)]
The limit
\[
\psi := \lim_{n\to\infty}\psi_n\psi_{n-1}\ldots \psi_2\psi_1
\]
is a well-defined  unitriangular $R$-algebra automorphism of $\CQ{Q}$.

\item[(ii)]
If $S$ and $(S_n)_{n>0}$ are a potential and a
sequence of potentials, respectively, on $Q$ such that
$\lim_{n\to\infty} \short(S_n) = \infty$, and
$\psi_n$ is a right equivalence $(Q,S+S_n)\to(Q,S+S_{n+1})$ for all $n>0$, 
then $\psi$  is a right equivalence $(Q,S+S_1)\to(Q,S)$.
\end{itemize}
\end{lemma}

\begin{proof} We can suppose, without loss of generality, that
$\depth(\psi_n)<\infty$ for all $n>0$ (this is because the only unitriangular
automorphism of $\CQ{Q}$ that has infinite depth is
the identity). Since $\lim_{n\to\infty}\depth(\psi_n)=\infty$, there exists
$m_2>0$ such that $\depth(\psi_k)>\depth(\psi_1)$ for all $k\geq m_2$. And
once we have a positive integer $m_n$, we can find $m_{n+1}>m_n$ such that
$\depth(\psi_k)>\depth(\psi_{m_n})$ for all $k \geq m_{n+1}$.
Set
$\vph_n := \psi_{m_{n}}\psi_{m_{n}-1}\ldots\psi_{m_{n-1}+1}$, with the convention
that $m_1 = 1$ and $m_0 = 0$.
Then $(\depth(\vph_n))_{n>0}$ is a strictly increasing
sequence of positive integers.

Note that the limit
$\lim_{n\to\infty}\vph_n\vph_{n-1}\ldots \vph_2\vph_1$ is equal to the
limit $\lim_{n\to\infty}\psi_{n}\psi_{n-1}\ldots \psi_2\psi_1$. In order to show
that $\vph := \lim_{n\to\infty}\vph_n\vph_{n-1}\ldots \vph_2\vph_1$ is
well-defined, it suffices to show that for every $u \in \maxid$ and every $d>0$
the sequence $(u_n^{(d)})_{n>0}$ formed by the components
of degree $d$ of the
elements $u_n := \vph_n\ldots\vph_1(u)$ eventually stabilizes as
$n\to\infty$.
Note that $\depth(\vph_n)\geq n-1$ for all $n>0$.
From this
and Lemma \ref{lemma:depth-is-short-cycle-friendly} we deduce that, for a
given $u\in\maxid$, there exists a sequence $(v_n)_{n>0}$ such that
$v_n\in\maxid^n$ and $u_n=u+\sum_{j=1}^nv_j$ for all $n>0$. From this we see that
$u_n^{(d)}=u_{n+1}^{(d)}$ for $n>d$, so the sequence $(u_n^{(d)})_{n>0}$ stabilizes.
This proves (i).

As before, let $V := V_Q$
be the
$\m$-adic closure of
the $\ka$-vector subspace of $\CQ{Q}$ generated by all
elements of the form $\alp_1\alp_2\ldots\alp_d-\alp_2\ldots\alp_d\alp_1$
with $\alp_1\ldots \alp_d$ a cycle in $Q$.

Since compositions of right equivalences are again right equivalences,
$\vph_n$ is a right equivalence $(Q,S+S_{m_{n}})\to(Q,S+S_{m_{n+1}})$ for all
$n>0$. We deduce that
\[
\vph_n\vph_{n-1}\ldots\vph_2\vph_1(S+S_1)-(S+S_{m_{n+1}})
\]
is contained in $V$.
The fact that both sequences
$(\vph_n\vph_{n-1}\ldots\vph_2\vph_1(S+S_1))_{n>0}$ and
$(-(S+S_{m_{n+1}}))_{n>0}$ are convergent in $\CQ{Q}$
implies that
\[
(\vph_n\vph_{n-1}\ldots\vph_2\vph_1(S+S_1)-(S+S_{m_{n+1}}))_{n>0}
\]
converges as well. Since $V$ is closed, we get that
\[
\vph(S+S_1) - S =
\lim_{n\to\infty}(\vph_n\vph_{n-1}\ldots
\vph_2\vph_1(S+S_1)-(S+S_{m_{n+1}}))
\]
is an element of $V$. Thus (ii) follows.
\end{proof}

\begin{lemma}\label{lemma:killing-S'} Let $S$ be a finite potential and
$m\geq \longest(S)$. If every cycle $\xi$ in $Q$ of length greater than $m$
is cyclically equivalent to an element of the form
$\sum_{\eta\in Q_1}u_\eta\partial_\eta(S)$ with
$\short(u_\eta)+\short(\partial_\eta(S))\geq\length(\xi)$ for all
$\eta\in Q_1$, then the following hold:
\begin{enumerate}

\item
For every non-zero potential $S'$ such that $\short(S')>m$ there exist a
potential $S''$ with $\short(S'')>\short(S')$ and a right equivalence
$\psi\colon (Q,S+S')\rightarrow(Q,S+S'')$ which is unitriangular and satisfies
$\depth(\psi)\geq\short(S')-\longest(S)$;

\item
For every non-zero potential $S'$ such that $\short(S')>m$,
the QP $(Q,S+S')$ is right equivalent to $(Q,S)$.

\end{enumerate}
\end{lemma}

\begin{proof}
Let $S'$ be a non-zero potential such that $\short(S')>m$. Up to
cyclic equivalence, we can assume that
\[
S'=\sum_{\eta\in Q_1}u_\eta\partial_\eta(S)
\]
for some elements
$u_\eta \in e_{t(\eta)}\CQ{Q} e_{s(\eta)}$,
with
\[
\short(u_\eta)+\short(\partial_\eta(S))\geq\short(S')>m\geq\longest(S).
\]
Note that $\short(u_\eta)>1$ for every $\eta\in Q_1$. Define a unitriangular
$R$-algebra automorphism $\psi$ of $\CQ{Q}$ by the
rule $\eta\mapsto\eta-u_\eta$ (that $\psi$ is indeed an automorphism and
indeed unitriangular follows from the fact that $\short(u_\eta)>1$ for every
$\eta\in Q_1$).
Then we have
\[
\depth(\psi)=\min\{\short(u_\eta)-1\mid \eta\in Q_1\}.
\]
Let $\alp\in Q_1$ be an arrow such that $\short(u_\alp)-1=\depth(\psi)$. Then
\[
\short(S')-\longest(S)\leq\short(S')-\short(\partial_\alp(S))-1\leq
\short(u_\alp)-1=\depth(\psi).
\]
Write
\[
S=\sum_{k=1}^tx_k\eta_{{k,1}}\eta_{{k,2}}\ldots \eta_{{k,m_k}}
\]
for some non-zero scalars $x_1,\ldots,x_t \in \ka$ and arrows 
$\eta_{k,i} \in Q_1$ for $1 \le i \le m_k$ and $1 \le k \le t$. Then we have
\[
\psi(S)=\sum_{k=1}^tx_k(\eta_{{k,1}}-u_{\eta_{{k,1}}})(\eta_{{k,2}}-u_{\eta_{{k,2}}})\ldots (\eta_{k,m_k}-u_{\eta_{{k,m_k}}}).
\]
Expanding each product
$(\eta_{{k,1}}-u_{\eta_{{k,1}}})(\eta_{{k,2}}-u_{\eta_{{k,2}}})\ldots (\eta_{k,m_k}-u_{\eta_{{k,m_k}}})$,
we see that it is possible to write it as $T_{k,1}+T_{k,2}+T_{k,3}$, where
\begin{enumerate}
\item $T_{k,1}=\eta_{{k,1}}\eta_{{k,2}}\ldots \eta_{{k,m_k}}$,
\item $T_{k,2}=-\sum_{l=1}^{m_k}\eta_{{k,1}}\ldots \eta_{k,l-1}u_{\eta_{{k,l}}}\eta_{k,l+1}\ldots  \eta_{{k,m_k}}$,
\item $T_{k,3}$ consists of all the summands that involve at least two factors
of the form $u_{\eta_{{k,l}}}$.
\end{enumerate}
Hence, we get
\[
\psi(S) = \sum_{k=1}^tx_kT_{k,1} + \sum_{k=1}^tx_kT_{k_2} + \sum_{k=1}^tx_kT_{k,3}.
\]
We obviously have $S=\sum_{k=1}^tx_kT_{k,1}$. Furthermore, $\sum_{k=1}^tx_kT_{k_2}$
is cyclically equivalent to $-S'$.
Therefore, $\psi(S+S')$ is cyclically equivalent to
\[
S+\sum_{k=1}^tx_kT_{k,3}+(\psi(S')-S').
\]
Since we clearly have $\short(\sum_{k=1}^tx_kT_{k,3})>\short(S')$, and since
$\short(\psi(S')-S')\geq \short(S')+\depth(\psi)>\short(S')$, setting
$S''=\sum_{k=1}^tx_kT_{k,3}+(\psi(S')-S')$ we have $\short(S'')>\short(S')$.
This finishes the proof the first statement.

An inductive use of the first statement yields a sequence $(S_n)_{n>0}$ of
potentials and a sequence $\psi_n$ of unitriangular $R$-algebra automorphisms
of $\CQ{Q}$, with the following properties:
\begin{itemize}
\item $S_1=S'$;
\item $\short(S_{n+1})>\short(S_n)$;
\item $\depth(\psi_n)\geq \short(S_n)-\longest(S)$;
\item $\psi_n$ is a right equivalence $(Q,S_n)\rightarrow (Q,S_{n+1})$.
\end{itemize}
The second property implies $\lim_{n\to\infty}S_n=0$. The second and third
properties imply that
$\lim_{n\to\infty}\depth(\psi_n)=\infty$. We can thus apply
Lemma~\ref{lemma:well-defined-limit-automorphism} and conclude that the
limit $\psi = \lim_{n\to\infty}\psi_n\psi_{n-1}\ldots\psi_1$ is a well-defined
right equivalence $(Q,W)\rightarrow (Q,S)$. This proves the second statement
and hence finishes the proof of Lemma \ref{lemma:killing-S'}.
\end{proof}

\begin{thm}\label{prop:only-one-potential}
Suppose $(Q,S)$ is a QP that satisfies the following three properties:
\begin{itemize}

\item[(i)]
$S$ is a finite potential;

\item[(ii)]
Every cycle $\xi$ in $Q$ of length greater than $\longest(S)$ is
cyclically equivalent to an element of the form
\[
\sum_{\eta\in Q_1}u_\eta\partial_\eta(S)
\]
with
\[
\short(u_\eta)+\short(\partial_\eta(S))\geq\length(\xi)
\]
for all $\eta\in Q_1$;

\item[(iii)]
Every non-degenerate potential on $Q$ is right equivalent to $S+S'$
for some potential $S'$ with $\short(S')>\longest(S)$.

\end{itemize}
Then $S$ is non-degenerate and every non-degenerate potential on $Q$ is
right equivalent to $S$.
\end{thm}

\begin{proof} Let $W$ be any non-degenerate potential on $Q$. 
By condition (iii), we can assume that $W = S+S'$ for some potential $S'$ with
$\short(S')>\longest(S)$.
By conditions (i) and (ii) we can apply Lemma~\ref{lemma:killing-S'} to deduce 
that $(Q,W)$ is right equivalent to $(Q,S)$.

Since $Q$ admits at least one non-degenerate potential (see \cite{DWZ1}), and 
since, as we just showed, any non-degenerate potential on $Q$ is right 
equivalent to $S$, we see that $S$ must be non-degenerate.
This finishes the proof.
\end{proof}

\subsection{Potentials for surfaces with non-empty boundary}
The main aim of this subsection is to prove the following theorem.
Its proof relies on Theorem~\ref{uniquehelp2}.

\begin{thm}\label{thm:uniqueness-of-potentials}
Suppose
$\surf$ is a marked surface with non-empty boundary, and
$\surf$ is not a torus with $|\marked| = 1$.
For any tagged triangulation $\tau$ of $\surf$ the quiver $\qtau$ admits only
one non-degenerate potential up to right equivalence.
\end{thm}

\begin{proof}
Let $\tau$ be a tagged triangulation of some arbitrary
marked surface $\surf$, and
suppose that $\sigma$ is another tagged triangulation
related to $\tau$ by a flip at $k$.
By \cite{FST} we know that $\mu_k(Q(\tau)) = Q(\sigma)$.
We know from \cite{DWZ1} that a QP-mutation at $k$ induces
a bijection between the sets of right equivalence classes
of non-degenerate potentials on $Q(\tau)$ and on $Q(\sigma)$.
Thus, to prove a uniqueness result as stated in the theorem,
it is enough to exhibit a single triangulation of $\surf$ whose quiver
admits only one non-degenerate potential up to right equivalence.

Now assume that
$\surf$ is a marked surface with non-empty boundary, and
assume that $\surf$ is not a torus with $|\marked| = 1$.
First, we assume additionally that $\surf$ is not a monogon.
For $\surf$ a digon, a triangle, or an annulus with 
$|\marked \setminus \punct| = 2$, let $\tau$ be the triangulation displayed in
Figure~\ref{Fig:new_digon_clannish.eps},  \ref{Fig:new_triangle_clannish.eps}, 
or~\ref{Fig:new_annulus_clannish.eps}, respectively.
Otherwise, let $\tau := \tau_t$ be the triangulation constructed
in the proof of Theorem~\ref{uniquehelp2}.
Then $\tau$ is a skewed-gentle triangulation.

Now let $S$ be a $3$-cycle potential on $Q(\tau)$.
(Note that $S = S(\tau,1)$.)
We claim that $S$ is, up to right equivalence, the unique non-degenerate 
potential on $Q(\tau)$.
To prove this claim, we show that the potential $S$ satisfies the conditions 
(i), (ii) and (iii) in Theorem~\ref{prop:only-one-potential}.
It clearly satisfies (i), and by Theorem~\ref{uniquehelp2}, (u) it also
satisfies condition (ii).
To show (iii),
let $W$ be any non-degenerate potential on $Q(\tau)$.
Without loss of generality, we can assume that the cycles appearing in $W$ are 
pairwise different up to rotation equivalence.
Since $Q(\tau)$ does not have double arrows, every 3-cycle in $Q(\tau)$ 
appears in $W$, up to rotation equivalence, see Corollary~\ref{prop3cycle}.
For each $3$-cycle $w$ appearing in $W$ there is an arrow
$\eta$ that only occurs in $w$ and in none of the other $3$-cycles
appearing in $W$, see Theorem~\ref{uniquehelp2}, (u).
Thus, applying a suitable rescaling of
arrows, we can assume without loss of generality, that each of these
$3$-cycles occurs with a coefficient equal to 1.
In other words, we can assume that
$W = S+S'$, where $S$ is our chosen $3$-cycle potential, and all
cycles appearing in $S'$ have length at least 4. Thus we have 
$\short(S') > \longest(S) = 3$.
So $S$ satisfies condition (iii) in Theorem~\ref{prop:only-one-potential}.
It follows that $S$ is the unique non-degenerate potential on $Q(\tau)$,
up to right equivalence.

Finally, let $\surf$ be a monogon with $t \ge 2$ punctures.
For $t=2$ one easily constructs a triangulation $\tau$ of $\surf$ such that
$Q(\tau)$ is acyclic.
Thus there is just one (non-degenerate) potential on $Q(\tau)$.
For $t \ge 4$ (respectively $t=3$) let $\tau$ be the triangulation
of $\surf$ displayed in Figure~\ref{Fig:new_monogon_clannish.eps}
(resp. Figure~\ref{Fig:new_monogon_3puncts.eps}).
Using the same techniques as the ones explained in Sections~\ref{secemptybdpot} 
and~\ref{sec5sphere} one shows that $Q(\tau)$ admits only one non-degenerate 
potential up to right equivalence.
\end{proof}


\section{Exceptional cases}\label{sec11b}


\subsection{}
In this section we classify the non-degenerate potentials on a
list of exceptional quivers of finite mutation type, and we determine
the representation type of the corresponding Jacobian algebras.
Summarizing, we will prove the following statement.

\begin{thm}
Let $Q$ be a $2$-acyclic quiver.
\begin{itemize}

\item[(i)]
If $Q$  is mutation equivalent to one of the quivers
$E_m$, $\widetilde{E}_m$ or $E_m^{(1,1)}$ with $m \ge 3$,
then $Q$ is Jacobi-tame.
Furthermore,
there exists only one non-degenerate potential on $Q$ up to right equivalence.

\item[(ii)]
If $Q$  is mutation equivalent to one of the quivers
$X_6$, $X_7$ or $K_m$ with $m \ge 3$, then $Q$ is Jacobi-wild.
Furthermore, for the cases $X_6$ and $K_m$
there exists only one non-degenerate potential on $Q$ up to right equivalence.

\item[(iii)]
If $Q$ is (mutation equivalent to) one of the quivers
$T_1$ or $T_2$, then $Q$ is Jacobi-irregular.
In particular, there are at least two non-degenerate potentials
on $Q$ up to weak right equivalence.

\end{itemize}
\end{thm}

Furthermore, we classify all non-degenerate potentials for the sphere with
$4$ punctures, and we show that for surfaces $\surf$ with empty boundary
and $|\marked| = 1$ there are at least two non-degenerate potentials up
to weak right equivalence.

For some exceptional quivers $Q$ of finite mutation type there are 
non-degenerate potentials $S$ on $Q$ such that the Jacobian algebra 
$\cP(Q,S)$ is wild. This is proved with the help of Galois coverings.
For convenience we recall the relevant definitions and results in 
Section~\ref{galoisreminder}.

\subsection{Acyclic quivers}

\subsubsection{Classification of non-degenerate potentials}

\begin{lemma}\label{lemma10.2}
Assume that $Q$ is mutation equivalent to an acyclic quiver.
Then there exists only one non-degenerate potential on $Q$.
\end{lemma}

\begin{proof}
An acyclic quiver $Q$ has only one potential, namely the zero potential
$S=0$, and $\cP(Q,S)$ is isomorphic to the path algebra $\kQ{Q}$.
The potential $S$ has to be non-degenerate by~\cite[Corollary~7.4]{DWZ1}.
Since QP-mutation induces a bijection on right equivalence classes of 
non-degenerate potentials~\cite[Theorem~5.7]{DWZ1}, the result follows.
\end{proof}

\subsubsection{Representation type}
The following result is well known, see for example~\cite{R}.

\begin{thm}\label{acycreptype}
For an acyclic quiver $Q$ the following hold:
\begin{itemize}

\item[(i)]
$\kQ{Q}$ is representation-finite if and only if $Q$ is a Dynkin quiver.

\item[(ii)]
$\kQ{Q}$ is tame if and only if $Q$ is a Euclidean quiver.

\item[(iii)]
In all other cases, $\kQ{Q}$ is wild.
\end{itemize}
\end{thm}

As a consequence of Theorems~\ref{thmmutationinv} and
Lemma~\ref{lemma10.2} we get the following result.

\begin{prop}
Assume that $Q$ is mutation equivalent to an acyclic quiver $Q'$.
Then $Q$ is Jacobi-wild if and only if $\kQ{Q'}$ is wild.
\end{prop}

In particular, the quivers $K_m$ with $m \ge 3$ are Jacobi-wild, and
the quivers mutation equivalent to one of the quivers
$E_m$ or $\widetilde{E}_m$ are Jacobi-tame.

\subsection{The quivers $E_m^{(1,1)}$}

\subsubsection{Classification of non-degenerate potentials}

\begin{lemma}
Let $Q$ be mutation equivalent to one of the quivers $E_m^{(1,1)}$.
Then there exists only one non-degenerate potential $S$ on $Q$ up to 
right equivalence.
\end{lemma}

\begin{proof}
Let $E$ be the quiver
\[
\xymatrix@-0.4pc{
&1 \ar[dl]\ar[dr]\ar[drr] \\
2 \ar[dr] && 3 \ar[dl]& 4 \ar[dll]\\
&5 \ar@<0.5ex>[uu]\ar@<-0.5ex>[uu]
}
\]
with arrows $\alp_i\df 1 \to i+1$ and $\bet_i\df i+1 \to 5$ for $i =1,2,3$
and $\gam_j\df 5 \to 1$ for $j=1,2$.
The quiver $\mu_5\mu_4\mu_3\mu_2(E)$ is acyclic. Thus up to right equivalence 
there is only one non-degenerate potential $S$ on $E$.
It is straightforward  to check that up to right equivalence we have
\[
S = \gam_1(\bet_1\alp_1 + \bet_2\alp_2) + \gam_2(\bet_2\alp_2 + \bet_3\alp_3).
\]
Potentials of this type can already be found in \cite[Section~1.2]{L1}.
Now let $Q$ be one of the quivers $E_m^{(1,1)}$ for $m=6,7,8$.
Clearly, $E$ is a full subquiver of $Q$, and every cycle in $Q$
is also a cycle in $E$. Thus, if $W$ is a non-degenerate potential on $Q$, 
then $W$ itself is a non-degenerate potential on $E$ by 
Proposition~\ref{restriction}, and hence there exists a right-equivalence 
$\varphi:(E,W)\rightarrow(E,S)$. This right-equivalence $\varphi$ can be 
extended to a right-equivalence between $(Q,W)$ and $(Q,S)$ by sending every 
arrow $\del\in Q_1\setminus E_1$ to itself. We see that $S$ is, up to right 
equivalence, the only non-degenerate potential on $Q$.
\end{proof}

\subsubsection{Representation type}

\begin{prop}\label{tameEcases}
Let $Q$ be mutation equivalent to one of the quivers
$E_m^{(1,1)}$.
Then $Q$ is Jacobi-tame.
\end{prop}

\begin{proof}
As shown above, there is only one non-degenerate potential $S$ on $Q$ up to 
right equivalence. The tameness of $\cP(Q,S)$ is proved by relating 
$\cP(Q,S)$ to Ringel's~\cite{R} tubular algebras. 
A detailed proof can be found in \cite{GeGo}.
\end{proof}

\subsection{The quiver $X_6$}

\subsubsection{Classification of non-degenerate potentials}
Let $X_6$ be the quiver displayed below
\[
X_6 := \vcenter{\xymatrix@-1.2pc{
&&
5 \ar[dd]^(.3){\del}\\
 3 \ar[rrd]^{\bet_1}&&&& 2 \ar@<0.5ex>[dd]^{\alp_2'}\ar@<-0.5ex>[dd]_{\alp_2}\\
&&6 \ar[rru]^{\gam_2}\ar[lld]^{\gam_1}&&\\
1 \ar@<0.5ex>[uu]^{\alp_1'}\ar@<-0.5ex>[uu]_{\alp_1}&&&& 4 \ar[llu]^{\bet_2}\\
}}
\]
and set
\[
W := \gam_1\bet_1\alp_1+\gam_2\bet_2\alp_2 + \gam_1\bet_2\alp_2'\gam_2\bet_1\alp_1'.
\]
Up to right equivalence $W$ is the only non-degenerate potential on $X_6$. 
Indeed, let $Q$ be the quiver obtained from $X_6$ by deleting the vertex $5$ 
and the arrow $\delta$. Then $(Q,W)$ is the QP associated to a triangulation of 
an annulus with two marked points on the boundary and one puncture, hence it is 
a non-degenerate QP. Moreover, $Q$ admits only one non-degenerate potential up 
to right equivalence by Theorem \ref{thm:uniqueness-of-potentials}. If $S$ is a 
non-degenerate potential on $X_6$, then $(Q,S)$ is non-degenerate by 
Proposition~\ref{restriction}, hence there exists a right-equivalence 
$\varphi(Q,S)\rightarrow(Q,W)$. This right equivalence can be extended to a 
right equivalence between $(X_6,S)$ and $(X_6,W)$ by sending the arrow $\delta$ 
to itself.

\subsubsection{Representation type}

\begin{lemma}
Assume that $Q$ is mutation equivalent to the quiver $X_6$. 
Then $Q$ is Jacobi-wild.
\end{lemma}

\begin{proof}
The Jacobian algebra $\LL_2:=\cP(X_6,W)$ with $X_6$ and $W$ as defined above
is obviously finite-dimensional.
Note, that with the $\ZZ \times \ZZ$-grading on $\CQ{X_6}$ defined by
$\abs{\alp'_1}=(1,0)$, $\abs{\alp_1}=(1,1)=\abs{\alp_2}$, $\abs{\alp'_2}=(0,1)$
and $\abs{\bet_i}=\abs{\gam_i}=\abs{\del}=0$ the potential $W \in \CQ{X_6}$
becomes homogeneous (of degree $(1,1)$).
This yields a $\ZZ\times\ZZ$-grading on $\LL_2=\cP(X_6,W)$.
We highlight in the  corresponding Galois-covering $\tilde{\LL}_2$
a convex hypercritical subcategory which belongs to the
frame $\doubletilde{$\mathsf{E}$}_7$. Thus $\LL_2$ is wild.
\[
\tilde{\LL}_2:=\qquad\vcenter{\xymatrix{
&&&\ar@<-1ex>@{.}[rr]&&\\
&6\ar[d]^{\gam_2}&\ar[l]^{\bet_1}3&\ar[l]^{\alp'_1}\ar[lddd]_(0.3){\alp_1}1&\ar[l]^{\gam_1}6\ar[d]^{\gam_2}&\ar[l]^{\bet_1}3&\ar[l]^{\alp'_1}\ar[lddd]_(0.3){\alp_1}1&\ar[l]^{\gam_1}   6\ar[d]^{\gam_2}\\
&2\ar[d]^{\alp'_2}&             &5\ar[ru]_(0.4){\del} &\ar[llld]^(0.2){\alp_2}  \rV{2}\rA{d}^{\alp'_2}&              &          5\ar[ru]_(0.4){\del} &\ar[llld]^(0.2){\alp_2}2\ar[d]^{\alp'_2}\\
\ar@{.}[dd]&4\ar[d]^{\bet_2}&              &                                     &        \rV{4}\rA{d}^{\bet_2}&               &                                      &                4\ar[d]^{\bet_2}&\ar@{.}[dd]\\
&6\ar[d]^{\gam_2}&\ar[l]^{\bet_1}3&\ar[l]^{\alp'_1}\ar[lddd]_(0.3){\alp_1}\rV{1}&\rA{l}^{\gam_1}\rV{6}\rA{d}^{\gam_2}&\rA{l}^{\bet_1}\rV{3}&\rA{l}^{\alp'_1}\ar[lddd]_(0.3){\alp_1}\rV{1}&\ar[l]^{\gam_1}   6\ar[d]^{\gam_2}&\\
&2\ar[d]^{\alp'_2}&          &\rV{5}\rA{ru}_(0.4){\del} &\ar[llld]^(0.2){\alp_2}\rV{2}\rA{d}^{\alp'_2}&              &           5\ar[ru]_(0.4){\del}&\ar[llld]^(0.2){\alp_2}2\ar[d]^{\alp'_2}&\\
&4\ar[d]^{\bet_2}&              &                                            & \rV{4}\ar[d]^{\bet_2}&               &                                       &                4\ar[d]^{\bet_2}\\
&6             &\ar[l]^{\bet_1}3&\ar[l]^{\alp_1}                      1&\ar[l]^{\gam_1}6              &\ar[l]^{\bet_1}3&\ar[l]^{\alp_1}                  1       &\ar[l]^{\gam_1}   6\\
&&&\ar@<1ex>@{.}[rr]&&
}}
\]
Notice, that $\tilde{\LL}$ is defined by a $\ZZ \times \ZZ$-periodic quiver
with all possible commutativities and the only zero-relations
$\gam_i\bet_i$ for $i\in\{1,2\}$ wherever applicable.
\end{proof}

\subsection{The quiver $X_7$}

\subsubsection{Classification of non-degenerate potentials}

For quivers mutation equivalent to $X_7$, the classification of
non-degenerate potentials up to right equivalence is still an open
problem.
We conjecture that there are at least two potentials up to right equivalence.

\subsubsection{Representation type}

\begin{lemma}
Assume that $Q$ is mutation equivalent to the quiver $X_7$. 
Then $Q$ is Jacobi-wild.
\end{lemma}

\begin{proof}
Suppose that $S$ is a non-degenerate potential on $X_7$.
An obvious QP-restriction yields a non-degenerate potential
$W$ on $X_6$ and a surjective algebra homomorphism
$f\df \cP(X_7,S) \to \cP(X_6,W)$.
This yields an exact embedding $\md(\cP(X_6,W)) \to \md(\cP(X_7,S))$.
Thus, if $\cP(X_6,W)$ is wild, then $\cP(X_7,S)$ is also wild.
We proved already above that $\cP(X_6,W)$ is wild.
This finishes the proof.
\end{proof}

\subsection{Torus $\surf$ with empty boundary
and $|\marked|=1$}

\subsubsection{Classification of non-degenerate potentials}
Let
$Q = Q(\tau)$ for some triangulation $\tau$ of a torus $(\Sigma,\marked)$ with 
empty boundary and $|\marked| = 1$.
Then $Q$ is isomorphic to the quiver $T_1$ displayed below.
\[
T_1:=\quad
\vcenter{\xymatrix@-1.2pc{
&&2 \ar@<0.4ex>[dddrr]^{\bet_1}\ar@<-0.4ex>[dddrr]_{\bet_2}\\
&&&&\\
&&&&\\
1 \ar@<0.4ex>[rruuu]^{\alp_1}\ar@<-0.4ex>[rruuu]_{\alp_2} &&&&3 \ar@<0.4ex>[llll]^{\gam_1}\ar@<-0.4ex>[llll]_{\gam_2}
}}
\]
The non-degenerate potentials on $T_1$ have been classified by 
Geuenich~\cite{Geu} in the following sense. He shows that a potential $S$ on 
$T_1$ is non-degenerate if and only if $\Smin$ is non-degenerate.
Furthermore, Geuenich proves that the potential $\Smin$ of
a non-degenerate potential $S$ is either right equivalent to
the potential $S_{\rm tame} := \alp_1\bet_1\gam_1 + \alp_2\bet_2\gam_2$,
or to the potential $S_{\rm wild} :=
\gam_2\bet_2\alp_1+\gam_2\bet_1\alp_2+\gam_1\bet_2\alp_2$.

\subsubsection{A tame non-degenerate potential on $T_1$}
For
\[
S := S_{\rm tame} = \alp_1\bet_1\gam_1 + \alp_2\bet_2\gam_2
\]
the Jacobian algebra $\cP(T_1,S)$ is a gentle algebra and therefore tame.

\subsubsection{A wild non-degenerate potential on $T_1$}
Let
\[
W := S_{\rm wild} =  \gam_2\bet_2\alp_1+\gam_2\bet_1\alp_2+\gam_1\bet_2\alp_2.
\]
The Jacobian algebra $\cP(T_1,W)$ is infinite-dimensional.
However, we
will see that the finite-dimensional quotient
$\LL_3:=\cP(T_1,W)/\langle\gam_1\bet_1\alp_1\gam_1\bet_1\rangle$ is wild.
Note that  $W\in\CQ{T_1}$ is homogeneous with respect to the
$\ZZ\times\ZZ$-grading of $\CQ{T_1}$ given by
$\abs{\alp_1}=\abs{\bet_1}=\abs{\gam_1}=(1,0)$ and
$\abs{\alp_2}=\abs{\bet_2}=\abs{\gam_2}=(0,1)$.
This yields on $\cP(T_2,W)$ and $\LL_3$ a
$\ZZ\times\ZZ$-grading. We highlight in the corresponding Galois-covering
$\tilde{\LL}_3$ (which is strongly simply connected) a convex hypercritical
subcategory which belongs to one of the frames which are concealed o of
type $\doubletilde{$\mathsf{E}$}_7$,  see~\cite[p.~150]{Ung90}.
It follows that $\LL_3$ is wild, and thus also $\cP(T_1,W)$ is wild.
\[
\tilde{\LL}_3:=\quad
\vcenter{\xymatrix@-.3pc{
                &              &                                &            \ar@{.}@<-2ex>[rrr] &&&\\
                &3\ar[d]^{\gam_2}&\ar[l]^{\bet_1}    2 \ar[d]^{\bet_2}&\ar[l]^{\alp_1}    1 \ar[d]^{\alp_2}&\ar[l]^{\gam_1}    3 \ar[d]^{\gam_2}&\ar[l]^{\bet_1}    2 \ar[d]^{\bet_2}&\ar[l]^{\alp_1}    1 \ar[d]^{\alp_2}&\ar[l]^{\gam_1}    3 \ar[d]^{\gam_2}&\ar[l]^{\bet_1} 2\ar[d]^{\bet_2}\\
\ar@{.}@<2ex>[d]&1\ar[d]^{\alp_2}&\ar[l]^{\gam_1}\rV{3}\ar[d]^{\gam_2}&\rA{l}^{\bet_1}\rV{2}\ar[d]^{\bet_2}&\rA{l}^{\alp_1}\rV{1}\rA{d}^{\alp_2}&\rA{l}^{\gam_1}\rV{3}\rA{d}^{\gam_2}&\rA{l}^{\bet_1}\rV{2}\rA{d}^{\bet_2}&\ar[l]^{\alp_1}    1 \ar[d]^{\alp_2}&\ar[l]^{\gam_1} 3\ar[d]^{\gam_2}&\ar@{.}@<-2ex>[d]\\
                &2\ar[d]^{\bet_2}&\ar[l]^{\alp_1}    1 \ar[d]^{\alp_2}&\ar[l]^{\gam_1}    3 \ar[d]^{\gam_2}&\ar[l]^{\bet_1}\rV{2}\ar[d]^{\bet_2}&\rA{l}^{\alp_1}\rV{1}\ar[d]^{\alp_2}&\rA{l}^{\gam_1}\rV{3}\ar[d]^{\gam_2}&\rA{l}^{\bet_1}\rV{2}\ar[d]^{\bet_2}&\ar[l]^{\alp_1} 1\ar[d]^{\alp_2}&\\
                &3             &\ar[l]^{\bet_1}    2              &\ar[l]^{\alp_1}    1              &\ar[l]^{\gam_1}    3              &\ar[l]^{\bet_1}    2              &\ar[l]^{\alp_1}    1              &\ar[l]^{\gam_1}    3              &\ar[l]^{\bet_1} 2              \\
                &              &                                &              \ar@{.}@<2ex>[rrr]&&&
}}
\]
Notice, that $\tilde{\LL}$ is defined by a $\ZZ\times\ZZ$-periodic quiver
with all possible commutativities and the  zero-relations
$\bet_2\alp_2$, $\gam_2\bet_2$, $\alp_2\gam_2$ and
$\gam_1\bet_1\alp_1\gam_1\bet_1$ wherever applicable.

\subsection{Torus $\surf$ with non-empty boundary
and $|\marked|=1$}

\subsubsection{Classification of non-degenerate potentials}
Let $Q = Q(\tau)$ for some triangulation $\tau$ of a torus $(\Sigma,\marked)$ 
with non-empty boundary and $|\marked| = 1$.
Then $Q$ is isomorphic to the quiver $T_2$ displayed below.
\[
T_2 := \qquad\vcenter{\xymatrix{
         &&1\ar@<-.8ex>[dd]_{\alp_1}\ar@<.8ex>[dd]^{\alp_2}&\\
3\ar[rru]^(.6){\bet_1}\ar@<1.8ex>@/^{4pc}/[rrrr]^{\delta}&&      &&4\ar[llu]_(.6){\bet_2} \\
         &&2\ar[llu]^{\gam_1}\ar[rru]_{\gam_2}
}}
\]

\begin{prop}\label{prop:two-potentials-on-torus-with-boundary} 
The quiver $Q := T_2$ admits exactly two non-degenerate potentials up to 
right equivalence. More specifically, the potentials
\[
S=\alp_1\bet_1\gam_1+\alp_2\bet_2\gam_2 \qquad\text{and}\qquad
W=\alp_1\bet_1\gam_1+\alp_1\bet_2\gam_2 +\alp_2\bet_2\del\gam_1
\]
are both non-degenerate and not right equivalent to each other,
and any non-degenerate potential on $Q$ is right equivalent to $S$ or to $W$.
\end{prop}

\begin{proof}
Note that $e_2\mathfrak{m}^3e_2\subseteq J(S)$. Were $(Q,S)$
and $(Q,W)$ right equivalent, there would exist an $R$-algebra automorphism
$\varphi$ of $\CQ{Q}$ such that $\varphi(J(S))=J(W)$. We would therefore have
$e_2\mathfrak{m}^3e_2=\varphi(e_2\mathfrak{m}e_2)^3\subseteq\varphi(J(S))=J(W)$.
Hence, in order to show that $(Q,S)$ and $(Q,W)$ are not right equivalent it
is enough to show that $e_2\mathfrak{m}^3e_2$ is not contained in $J(W)$. We
claim that the element $\alp_2\bet_2\gam_2$ of $e_2\mathfrak{m}^3e_2$ does not
belong to $J(W)$. To prove this claim it suffices to exhibit a nilpotent
representation $M$ of $Q$ which is annihilated by the cyclic derivatives of
$W$ and such that $\alp_2\bet_2\gam_2$ does not act as the zero map on $M$.
A straightforward check shows that the representation
\[
M =\qquad
\vcenter{\xymatrix{
         &&\ka\ar@<-.8ex>[dd]_{0}\ar@<.8ex>[dd]^{\left(\begin{smallmatrix} 0\\1 \end{smallmatrix}\right) }&\\
\ka\ar[rru]^(.6){-1}\ar@<1.8ex>@/^{4pc}/[rrrr]^{0}&&      &&\ka\ar[llu]_(.6){1} \\
         &&\ka^2\ar[llu]^{(1,0)}\ar[rru]_{(1,0)} \\
}}\]
satisfies all these requirements. We conclude that $(Q,S)$ and $(Q,W)$ are not
right equivalent.

To show that, up to right equivalence, there are no other non-degenerate
potentials on $Q$ besides $S$ and $W$, we first show:

\begin{lemma}\label{lemma:S-and-W-on-torus}\begin{enumerate}\item For any
potential $S'$ on $Q$ such that $\short(S')>3$, the QP $(Q,S+S')$ is
right equivalent to $(Q,S)$.
\item For any potential $W'$ on $Q$ such that $\short(W')>4$, the QP
$(Q,W+W')$ is right equivalent to $(Q,W)$.
\end{enumerate}
\end{lemma}

\begin{proof} The potential $S$ obviously satisfies the hypothesis of
Lemma~\ref{lemma:killing-S'} with $m=3$. To see that $W$ satisfies such
hypothesis with $m=4$, first notice that any cycle in $Q$ passing through the
arrow $\del$ is rotationally equivalent to a cycle that has the path
$\bet_2\del\gam_1=\partial_{\alp_1}(W)$ as a factor. Furthermore, any cycle on
$Q$ of length greater than 4 not passing through $\del$ is rotationally equivalent
to a cycle that has one of the following paths as a factor:
\begin{itemize}
\item
$\alp_1\bet_1\gam_1\alp_1\bet_1\gam_1=
\alp_1\bet_1\partial_{\bet_1}(W)\bet_1\gam_1$
\item
$\alp_1\bet_1\gam_1\alp_1 \bet_2\gam_2=
\alp_1\bet_1\partial_{\bet_1}(W) \bet_2\gam_2$
\item
$\alp_1\bet_1\gam_1 \alp_2\bet_1\gam_1=
\alp_1\bet_1\gam_1 \alp_2\bet_1\gam_1+\alp_1\bet_1\gam_1 \alp_2\bet_2\gam_2-
\alp_1\bet_1\gam_1 \alp_2\bet_2\gam_2=
\alp_1\bet_1\gam_1 \alp_2\partial_{\alp_1}(W)-\alp_1\bet_1\partial_{\del}(W)\gam_2$
\item
$\alp_1\bet_1\gam_1 \alp_2\bet_2\gam_2=\alp_1\bet_1\partial_{\del}(W)\gam_2$
\item
$\alp_1 \bet_2\gam_2\alp_1\bet_1\gam_1=
\partial_{\gam_2}(W)\gam_2\alp_1\bet_1\gam_1$
\item
$\alp_1 \bet_2\gam_2\alp_1 \bet_2\gam_2=
\partial_{\gam_2}(W)\gam_2\alp_1 \bet_2\gam_2$
\item
$\alp_1 \bet_2\gam_2 \alp_2\bet_1\gam_1=
\partial_{\gam_2}(W)\gam_2 \alp_2\bet_1\gam_1$
\item
$\alp_1 \bet_2\gam_2 \alp_2\bet_2\gam_2=
\partial_{\gam_2}(W)\gam_2 \alp_2\bet_2\gam_2$
\item
$\alp_2\bet_1\gam_1\alp_1\bet_1\gam_1=
\alp_2\bet_1\gam_1\alp_1\bet_1\gam_1+\alp_2\bet_1\gam_1 \alp_2\bet_2\del\gam_1-
\alp_2\bet_1\gam_1 \alp_2\bet_2\del\gam_1=
\alp_2\bet_1\gam_1\partial_{\gam_1}(W)\gam_1-
\alp_2\bet_1\partial_{\del}(W)\del\gam_1$
\item
$\alp_2\bet_1\gam_1\alp_1\bet_2\gam_2=
\alp_2\bet_1\gam_1\partial_{\gam_2}(W)\gam_2$
\item
$\alp_2\bet_1\gam_1 \alp_2\bet_1\gam_1=
\alp_2\bet_1\gam_1\alp_2\bet_1\gam_1+\alp_2\bet_1\gam_1 \alp_2\bet_2\gam_2-
\alp_2\bet_1\gam_1\alp_2\bet_2\gam_2=
\alp_2\bet_1\gam_1\alp_2\partial_{\alp_1}(W)-\alp_2\bet_1\partial_{\del}(W)\gam_2$
\item
$\alp_2\bet_1\gam_1 \alp_2\bet_2\gam_2=\alp_2\bet_1\partial_{\del}(W)\gam_2$
\item
$\alp_2\bet_2\gam_2\alp_1\bet_1\gam_1=
\alp_2\bet_2\gam_2\alp_1\bet_1\gam_1+\alp_2\bet_2\del\gam_1 \alp_2\bet_1\gam_1-
\alp_2\bet_2\del\gam_1\alp_2\bet_1\gam_1=
\alp_2 \bet_2\partial_{\bet_2}(W)\bet_1\gam_1-
\alp_2\partial_{\alp_2}(W)\alp_2\bet_1\gam_1$

\item
$\alp_2\bet_2\gam_2\alp_1 \bet_2\gam_2=
 \alp_2\bet_2\gam_2\partial_{\gam_2}(W)\gam_2$

\item
$\alp_2\bet_2\gam_2 \alp_2\bet_1\gam_1 =
\alp_2 \bet_2\gam_2 \alp_2\bet_1\gam_1 + \alp_2 \bet_1\gam_1 \alp_2\bet_1\gam_1 -
\alp_2 \bet_1\gam_1 \alp_2\bet_1\gam_1 - \alp_2 \bet_1\gam_1 \alp_2\bet_2\gam_2 +
\alp_2 \bet_1\gam_1 \alp_2\bet_2\gam_2 =
\alp_2 \partial_{\alp_1}(W) \alp_2\bet_1\gam_1 -
\alp_2 \bet_1\gam_1 \alp_2\partial_{\alp_1}(W) +
\alp_2 \bet_1\partial_{\del}(W)\gam_2$

\item
$\alp_2 \bet_2\gam_2 \alp_2\bet_2\gam_2 =
\alp_2 \bet_2\gam_2 \alp_2\bet_2\gam_2 + \alp_2 \bet_1\gam_1 \alp_2\bet_2\gam_2 -
\alp_2 \bet_1\gam_1 \alp_2\bet_2\gam_2 =
\alp_2 \partial_{\alp_1}(W) \alp_2\bet_2\gam_2 -
\alp_2 \bet_1\partial_{\del}(W)\gam_2$.

\end{itemize}

This readily implies that $W$ satisfies the hypothesis of
Lemma~\ref{lemma:killing-S'} with $m=4$.
\end{proof}

Let $T$ be any non-degenerate potential on $Q$. Up to cyclic equivalence,
we can write it as
\begin{eqnarray}\nonumber
T & = & x_1\alp_1\bet_1\gam_1 + x_2\alpha_1\bet_2\gam_2+x_3\alp_1 b\del\gam_1\\
\nonumber&& +y_1\alp_1\bet_1\gam_1 + y_2 \alp_2\bet_2\gam_2+
y_3 \alp_2\bet_2\del\gam_1\\
\nonumber&& +T',
\end{eqnarray}
where $x_1,x_2,y_1,y_2,z_1,z_2$ are scalars and $T'\in\mathfrak{m}^5$.

A straightforward computation shows that the degree-2 component of 
$\widetilde{\mu}_4\widetilde{\mu}_3(T)$ is
\[
x_1\alp_1[\bet_1\gam_1] + x_2\alp_1 [\bet_2\gam_2]+x_3\alp_1 [\bet_2[\del\gam_1]]
+y_1\alp_2[\bet_1\gam_1] + y_2 \alp_2[\bet_2\gam_2]+y_3 \alp_2[\bet_2[\del\gam_1]].
\]
Since $\mu_4\mu_3(Q,T)$ is (right equivalent to) the reduced part of
$\widetilde{\mu}_4\widetilde{\mu}_3(Q,T)=
(\widetilde{\mu}_4\widetilde{\mu}_3(Q),\widetilde{\mu}_4\widetilde{\mu}_3(T))$
and since $T$ is non-degenerate, this implies that the matrix
\[
X=\left[\begin{array}{ccc}
x_1 & x_2 & x_3\\
y_1 & y_2 & y_3\end{array}\right]
\]
has rank 2. If the first two columns of this matrix are linearly independent,
then the rule
$\psi:\alp_1\mapsto x_1\alp_1+y_1\alp_2,
      \alp_2\mapsto x_2\alp+1+y_2\alp_2$, gives rise to an automorphism
$\psi$ of $\CQ{Q}$. Its inverse $\psi^{-1}$ maps $T$ to
$\psi^{-1}(T)=
S+\psi^{-1}(x_3\alp_1\bet_2\del\gam_1+y_3\alp_2\bet_2\del\gam_1+T')$.
Since
$\short(\psi^{-1}(x_3\alp_1\bet_2\del\gam_1+y_3\alp_1\bet_2\del\gam_1+T'))>3$,
we can apply Lemma~\ref{lemma:S-and-W-on-torus} to conclude that
$\psi^{-1}(T)$, and hence $T$, is right equivalent to $S$.

If the first two columns of $X$ are linearly dependent, then
\[
\left[\begin{array}{cc}u_1 & u_3\\
u_2 & u_4
\end{array}\right]
\left[\begin{array}{ccc}
x_1 & x_2 & x_3\\
y_1 & y_2 & y_3\end{array}\right]=
\left[\begin{array}{ccc}1 & w_1 & w_2\\
0 & 0 & 1\end{array}\right]
\]
for some invertible matrix
\[
D=\left[\begin{array}{cc}u_1 & u_3\\
u_2 & u_4
\end{array}\right]
\]
The rule
$\psi:\alp_1\mapsto u_1\alp_1+ u_2\alp_2,\alp_2\mapsto u_3\alp_1+u_4\alp_2$
gives rise to an $R$-algebra automorphism of $\CQ{Q}$ and sends $T$ to
\begin{eqnarray}\nonumber
\psi(T) & = & \alp_1\bet_1\gam_1 + w_1\alp_1\bet_2\gam_2+w_2\alp_1 \bet_2\del\gam_2+\alp_2\bet_2\del\gam_1+\psi(T'),
\end{eqnarray}
Since the degree-2 component of $\widetilde{\mu}_4(\psi(T))$ is
$w_1\alp_1 [\bet_2\gam_2]$ and since $T$ is non-degenerate, we deduce that
$w_1\neq 0$. The rule
$\varphi:\gam_2\mapsto w_1^{-1}\gam_2-w_1^{-1}w_2\del\gam_1$ sends $\psi(T)$ to
\[
\varphi\psi(T)  =  W+\varphi\psi(T')
\]
Noticing that $\short(\varphi\psi(T'))\geq 5$, we can apply
Lemma~\ref{lemma:S-and-W-on-torus} to conclude that $\varphi\psi(T)$,
and hence $T$, is right equivalent to $W$.
\end{proof}

We already know that the potential $S$ is Jacobi-finite and rigid. Since $W$
is non-degenerate, it makes sense to ask whether $W$ is also Jacobi-finite
and/or rigid. Here is the answer:

\begin{prop}
The potential $W$ is Jacobi-finite and non-rigid.
\end{prop}

\begin{proof} The finite-dimensionality of the Jacobian algebra  $\cP(Q,W)$ is
already evident in the proof of Lemma~\ref{lemma:S-and-W-on-torus}.
To show it is not rigid, we first establish a result concerning
cyclic equivalence for general quivers.

\begin{lemma} Let $Q$ be any quiver and let $\mathcal{C}$ be a system of the
representatives for the equivalence relation of rotation on the set of all
cycles on $Q$ (thus, two cycles are equivalent if they can be obtained from
each other by rotation, and $\mathcal{C}$ contains exactly one element from
each equivalence class). If $\sum_{\xi\in\mathcal{C}}x_\xi\xi$ is a (possibly
infinite) linear combination of elements of $\mathcal{C}$ that is
cyclically equivalent to zero, then $x_\xi=0$ for all $\xi\in\mathcal{C}$.
\end{lemma}

\begin{proof} That $u=\sum_{\xi\in\mathcal{C}}x_\xi\xi$ is cyclically equivalent
to zero means that it is the limit of a sequence $(u_n)_{n>0}$ of elements of
$\CQ{Q}$ that can be written as finite $\ka$-linear combinations of elements
of the form $\alp_1\alp_2\ldots \alp_d-\alp_2\ldots \alp_d\alp_1$ with
$\alp_1\ldots \alp_d$ a cycle on $Q$. Let $d>0$, then there exists $N>0$ such
that $u_n^{(d)}=u_{n+1}^{(d)}=u^{(d)}$ for all $n\geq N$.

Given a cycle $\alp_1\alp_2\ldots \alp_d$ on $Q$, and scalars $x_1,\ldots, x_d$,
we have
\begin{multline}\label{eq:lin-comb-of-commutators}
\sum_{k=1}^dx_k(\alp_k \alp_{k+1}\ldots \alp_d\alp_1\ldots\alp_{k-1}-
\alp_{k+1}\ldots \alp_d\alp_1\ldots \alp_{k-1}\alp_k)=\\
\sum_{k=1}^d (x_k-x_{k-1})\alp_k \alp_{k+1}\ldots \alp_d\alp_1\ldots \alp_{k-1},
\end{multline}
(with the convention that $\alp_{d+1}=\alp_1$ and $x_0=x_d$)
which means that~\eqref{eq:lin-comb-of-commutators} is non-zero if and only if
$x_1,\ldots,x_d$ are not all the same scalar, in which case, in the right hand
side of~\eqref{eq:lin-comb-of-commutators} will have at least two terms
appearing with non-zero scalar. Since such two terms are rotationally 
equivalent, we see that we cannot obtain $u_n^{(d)}=u^{(d)}$ as a finite sum of 
elements of the form~\eqref{eq:lin-comb-of-commutators} unless $u^{(d)}=0$. 
Therefore $u=0$.
\end{proof}

To show that $W$ is not rigid, it is enough to exhibit a cycle which is not
cyclically equivalent to any element of the Jacobian ideal $J(W)$. We claim
that $\alp_2\bet_2\gam_2$ is such cycle. Indeed, suppose that
$\alp_2\bet_2\gam_2$ is cyclically equivalent to an element $u$ of $J(W)$.
Then we can take $u$ to have the form
\begin{equation}\label{eq:W-not-rigid}
u=\sum_{\eta\in Q_1}u_\eta\partial_\eta(W).
\end{equation}
A direct check shows that the only summand in this expression that contributes
with a rotation of $\alp_2\bet_1\gam_1$ or $\alp_2\bet_2\gam_2$ to the
expression of $u$ as a (possibly infinite) linear combination of cycles, is
the summand with $\eta=\alp_1$. Moreover, the rotations $\bet_1\gam_1\alp_2$,
$\gam_1\alp_2\bet_1$, $\bet_2\gam_2\alp_2$ and $\gam_2\alp_2\bet_2$, are not
contributed at all by any summand in \eqref{eq:W-not-rigid}. Even more, in the
expression of $u$ as a (possibly infinite) $\ka$-linear combination of cycles,
the cycles $\alp_2\bet_1\gam_1$ and $\alp_2\bet_2\gam_2$ must appear accompanied
with the same scalar $x$. Since $\alp_2\bet_2\gam_2-u$ is cyclically equivalent
to zero, we see that $x=1$ and, at the same time, $x=0$, contradiction that
proves that $\alp_2\bet_2\gam_2$ is not cyclically equivalent to any element
of the Jacobian ideal $J(W)$. We conclude that $W$ is not rigid.
\end{proof}

\subsubsection{A tame non-degenerate potential on $T_2$}
It is not hard to check that
\[
S := \alp_1\bet_1\gam_1 + \alp_2\bet_2\gam_2
\]
is a non-degenerate potential on $T_2$, and that
$\cP(T_2,S)$ is a gentle algebra and therefore tame.

\subsubsection{A wild non-degenerate potential on $T_2$}

For
\[
W := \alp_1\bet_1\gam_1+\alp_1\bet_2\gam_2+\alp_2\bet_2\del\gam_1
\]
the Jacobian algebra $\LL_1:=\cP(T_2,W)$ is finite-dimensional.
Note that $W\in\CQ{T_2}$ is homogeneous with respect to the
$\ZZ\times\ZZ$-grading on $\CQ{T_2}$ given by
$\abs{\alp_1}=(1,1), \abs{\alp_2}=(1,0), \abs{\del}=(0,1), \abs{\bet_i}=0,
\abs{\gam_i}=0$. This yields a $\ZZ\times\ZZ$-grading on $\LL_1$.

We highlight  in the corresponding Galois covering $\tilde{\LL}_2$ a convex
hypercritical subcategory.
It belongs to one of the frames which are concealed of type
$\doubletilde{$\mathsf{E}$}_7$, see~\cite[p.~150]{Ung90}.
Thus $\LL_1=\cP(T_2,W)$ is wild.
\[
\tilde{\LL}_1:=\qquad\vcenter{\xymatrix@-.4pc{
&               &\ar@<-2ex>@{.}[rrr]&&&\\
&               &                 4  \ar[d]^{\bet_2}&\ar[l]^{\gam_2}    2 \ar[d]^{\gam_1}&\ar[l]^{\alp_2}    1\ar[dd]^{\alp_1}&\ar[l]^{\bet_1}    3\ar[d]^{\del}  \\
&2\ar[d]^{\gam_1} &\ar[l]^{\alp_2}\rV{1}\ar[dd]^{\alp_1}&\rA{l}^{\bet_1}\rV{3}\rA{d}^{\del}  &                                &                 4 \ar[d]^{\bet_2}&\ar[l]^{\gam_2}2 \ar[d]^{\gam_1} \\
\ar@<2ex>@{.}[ddd]&3\ar[d]^{\del}   &                                &              \rV{4}\rA{d}^{\bet_2}&\rA{l}^{\gam_2}\rV{2}\rA{d}^{\gam_1}&\ar[l]^{\alp_2}    1\ar[dd]^{\alp_1}&\ar[l]^{\bet_1}3 \ar[d]^{\del}&\ar@<-2ex>@{.}[ddd] \\
&4\ar[d]        &\ar[l]^{\gam_2}    2  \ar[d]^{\gam_1}&\ar[l]^{\alp_2}\rV{1}\ar[dd]^{\alp_1}&\rA{l}^{\bet_1}\rV{3}\rA{d}^{\del} &                                &             4 \ar[d]^{\bet_2} &\\
&1\ar[dd]^{\alp_1}&\ar[l]^{\bet_1}    3 \ar[d]^{\del}  &                                 &             \rV{4}\ar[d]^{\bet_2}&\rA{l}^{\gam_2}\rV{2}\rA{d}^{\gam_1}&\ar[l]^{\alp_2}1\ar[dd]^{\alp_1}&\\
&                &                4 \ar[d]^{\bet_2}&\ar[l]^{\gam_2}     2\ar[d]^{\gam_1} &\ar[l]^{\alp_2}    1\ar[dd]^{\alp_1}&\ar[l]^{\bet_1}\rV{3} \ar[d]^{\del} &              &\\
&2               &\ar[l]^{\alp_2}1                 &\ar[l]^{\bet_1}3 \ar[d]^{\del}       &                               &                  4 \ar[d]^{\bet_2}&\ar[l]^{\gam_2}2 \ar[d]^{\gam_1} \\
&                &                               &              4                  &\ar[l]^{\gam_2}    2             &\ar[l]^{\alp_2}     1               &\ar[l]^{\bet_1}3               \\
&               &\ar@<2ex>@{.}[rrr]&&&
}}\]
Note that the Galois covering  $\tilde{\LL}_1$ is (strongly) simply
connected. It is given by a $\ZZ\times\ZZ$-periodic quiver together with
all possible commutativities and the zero relations
$\gam_1\alp_1$, $\alp_1\bet_2$, $\bet_2\del\gam_1$ and $\gam_1\alp_2\bet_2$
wherever applicable.

\subsection{Surfaces $(\Sigma,\marked)$ with empty boundary
and $|\marked| = 1$}

\begin{prop}
Let $\tau$ be a triangulation of a marked surface $(\Sigma,\marked)$ with 
empty boundary and $|\marked|=1$.
Then there exist at least two non-degenerate potentials
on $Q(\tau)$ that are not weakly right equivalent.
\end{prop}

\begin{proof}

Consider the definition of $S(\tau,\bx)$ for $\bx=x$ a non-zero scalar 
(attached to the unique puncture of $\surf$). Since for any triangulation of 
$\surf$ the unique puncture always has valency greater than two, one can 
extend the definition of $S(\tau,\bx)$ to include the situation where 
$\bx=x=0$. One obtains $S(\tau,0)=S(\tau,\bx)_{\operatorname{min}}$ for any scalar 
$\bx=x$.

In \cite[Theorem 30]{Labardini1} it is proved that, for $\bx=x\neq 0$, the QPs 
$\mu_i(Q(\sigma_1),S(\sigma_1,\bx))$ and $(Q(\sigma_2),S(\sigma_2,\bx))$ are 
right equivalent for any two triangulations $\sigma_1$ and $\sigma_2$ related 
by the flip of an arc $i\in\sigma_1$. As pointed out by Ladkani 
in~\cite[Section 4.3]{L0}, the same proof from~\cite[Theorem 30]{Labardini1} 
is still valid when one allows $\bx=x=0$. This readily implies the 
non-degeneracy of $S(\tau,0)$. On the other hand, it is easily seen that the 
Jacobian algebra $\mathcal{P}(Q(\tau),S(\tau,0))$ is infinite-dimensional, 
while $\mathcal{P}(Q(\tau),S(\tau,\bx))$ is finite-dimensional for 
$\bx=x\neq 0$ as shown by Ladkani in \cite{L2}. The proposition follows.

\end{proof}

\begin{remark} It has been proved by Geuenich \cite{Geu}, that in the case of 
the once-punctured torus, the quiver $Q(\tau)$ has infinitely many 
non-degenerate potentials which are pairwise not weakly right equivalent.
\end{remark}

\subsection{The sphere with four punctures}
Let $(Q^{(1)}, W^{(1)}_t)$ be the QP defined below.
We have $Q^{(1)} = Q(\tau)$, where $\tau$ is the
usual tetrahedron triangulation of a  sphere $\surf$ with $|\marked| = 4$.
Let $\LL_t := \cP(Q^{(1)}, W^{(1)}_t)$ be the corresponding Jacobian algebra.

It is easy to see that each quiver which is mutation equivalent to
$Q^{(1)}$ is isomorphic to one of the quivers $Q^{(i)}$ for $i=1,2,3,4$ below.
The aim of this section is to prove the following result:

\begin{prop} \label{prp:Sp4}
For the quiver $Q := Q^{(1)}$ the following hold:
\begin{itemize}

\item[(i)]
Each non-degenerate potential on $Q^{(1)}$ is right equivalent to
one of the potentials $W^{(1)}_t$ with $t\in\ka\setminus\{0,1\}$;

\item[(ii)]
The Jacobian algebras
$\LL_t$ and $\LL_{t'}$ are isomorphic if and only if $t'\in\{t, t^{-1}\}$.
\end{itemize}
\end{prop}

\begin{coro}
Let $Q = Q(\tau)$ for a triangulation $\tau$ of a sphere $\surf$ with
$|\marked| = 4$.
Then there are infinitely many non-degenerate potentials on $Q$ up
to weak right equivalence.
\end{coro}

The Jacobian algebras $\LL_t$ for $t \not= 0,1$ are the trivial extension algebras of the tubular algebras of type $(2,2,2,2)$, and hence are tame finite-dimensional symmetric algebras, see
\cite{BSk}.
The algebra $\LL_1$ is an infinite dimensional clannish algebra, and the
corresponding potential $W^{(1)}_1$ is degenerate.

We have to consider the following four quivers with potential where we
assume that the coefficients $t, t_{ij}, t_{ijh}$ belong to $\ka^*$.
For typographical reasons we identify each of the quivers along
the dotted line. See also~\cite[p.~117]{GKO} for essentially the same list.
\[
Q^{(1)}:
\vcenter{\xymatrix@+0.3pc@L-.1pc{
1 \ar@{.}[d]
& \ar@/_/[l]_{\gam_{11}}\ar@/_/[ld]^(.3){\gam_{21}} 5
& \ar@/_/[l]_{\bet_{11}}\ar@/_/[ld]^(.3){\bet_{21}} 3
& \ar@/_/[l]_{\alp_{11}}\ar@/_/[ld]^(.3){\alp_{21}} 1\ar@{.}[d]\\
2
& \ar@/^/[l]^{\gam_{22}}\ar@/^/[lu]_(.3){\gam_{12}} 6
& \ar@/^/[l]^{\bet_{22}}\ar@/^/[lu]_(.3){\bet_{12}} 4
& \ar@/^/[l]^{\alp_{22}}\ar@/^/[lu]_(.3){\alp_{12}} 2
}}
\quad\begin{aligned}
W_t^{(1)} &:= (t-1) \gam_{21}\bet_{11}\alp_{12} +
\sum_{i,j,k=1}^2 \gam_{ik}\bet_{kj}\alp_{ji}\\
W^{(1)}_{\text{gen}}&:= \sum_{i,j,k=1}^2 t_{kji}\gam_{ik}\bet_{kj}\alp_{ji},
\end{aligned}
\]
\[
Q^{(2)}:
\vcenter{\xymatrix@-.5pc@L-.1pc{
\ar@{.}[d] & \ar[ld]_{\alp_1} 3 &        & \ar[ld]_{\gam_1} 5 & \ar@{.}[d]\\
1&  & \ar[lu]_{\bet_1}\ar[ld]^{\bet_2} 2 & &\ar[lu]_{\del_1}\ar[ld]^{\del_2} 1\\
\ar@{.}[u] & \ar[lu]^{\alp_2} 4 &        & \ar[lu]^{\gam_2}  6 & \ar@{.}[u]
}}\quad
\begin{aligned}
W^{(2)}_t &:= (t-1)\alp_1\bet_1\gam_1\del_1 +\sum_{i,j=1}^2 \alp_i\bet_i\gam_j\del_j \\
W^{(2)}_{\text{gen}} &:= \sum_{i,j=1}^2 t_{ij}\alp_i\bet_i\gam_j\del_j,
\end{aligned}
\]
\[
Q^{(3)}:
\vcenter{\xymatrix@-.5pc@L-.1pc{
\ar@{.}[d] & \ar[ld]_{\alp_1} 3 &              & \ar[ld]_{\gam} 5& \ar@{.}[d]\\
1&\ar[l]_{\alp_2} 4 & \ar[lu]_{\bet_1}\ar[l]_{\bet_2}\ar[ld]^{\bet_3} 2 &   &\ar[lu]_{\del}\ar[ll]_{\eps} 1\\
\ar@{.}[u] & \ar[lu]^{\alp_3} 6 &              &   & \ar@{.}[u]\\
}}
\quad
\begin{aligned}
W^{(3,1)}_t&:=(\sum_{i=1}^3\alp_i\bet_i)\eps+ (t\alp_1\bet_1+\alp_2\bet_2)\gam\del\\
W^{(3,2)}_t&:=(\sum_{i=1}^3\alp_i\bet_i)\eps+ (t\alp_2\bet_2+\alp_3\bet_3)\gam\del\\
W^{(3,3)}_t&:=(\sum_{i=1}^3\alp_i\bet_i)\eps+ (t\alp_3\bet_3+\alp_1\bet_1)\gam\del,\\
\end{aligned}
\]
\[
Q^{(4)}:
\vcenter{\xymatrix@-.8pc@L-.1pc{
\ar@{.}[dd] & \ar[ldd]_{\alp_1} 3 &              & & \ar@{.}[dd]\\
              &\ar[ld]^{\alp_2} 4\\
1&&\ar[luu]_{\bet_1}\ar[lu]^{\bet_2}\ar[ld]_{\bet_3}\ar[ldd]^{\bet_4}2&&\ar@<-1ex>[ll]_{\eps_1} \ar@<1ex>[ll]^{\eps_2} 1\\
             & \ar[lu]_{\alp_3} 6 \\
\ar@{.}[uu]  & \ar[luu]^{\alp_4}  5 &              & & \ar@{.}[uu]
}}
\
\begin{aligned}
W^{(4,1)}_t&:=(\sum_{i=1}^3\alp_i\bet_i)\eps_1+
(t\alp_1\bet_1+\alp_2\bet_2+\alp_4\bet_4)\eps_2\\
W^{(4,2)}_t&:=(\sum_{i=1}^3\alp_i\bet_i)\eps_1+
(t\alp_2\bet_2+\alp_3\bet_3+\alp_4\bet_4)\eps_2\\
W^{(4,3)}_t&:=(\sum_{i=1}^3\alp_i\bet_i)\eps_1+
(t\alp_3\bet_3+\alp_1\bet_1+\alp_4\bet_4)\eps_2.\\
\end{aligned}
\]
We start by collecting some elementary facts about the above family of
quivers with potential.

\begin{lemma} \label{lem:Sp4-el}
\begin{itemize}

\item[(a)]
The potential $W^{(1)}_{\text{gen}}$ on the quiver $Q^{(1)}$ above is 
right equivalent to the potential $W^{(1)}_t$ if and only if
\[
t=\frac{t_{112}t_{121}t_{211}t_{222}}{t_{111}t_{122}t_{212}t_{221}}.
\]
Moreover, $(Q^{(1)},W^{(1)}_t)=\widetilde{\mu}_1(Q^{(2)},W^{(2)}_t)$.

\item[(b)]
The potential $W^{(2)}_{\text{gen}}$ on $Q^{(2)}$ is right equivalent to $W^{(2)}_t$
if and only if
\[
t=\frac{t_{11} t_{22}}{t_{12} t_{21}}.
\]

\item[(c)]
The potentials $W^{(3,2)}_{t/(t-1)}$ and  $W^{(3,3)}_{(1-t)^{-1}}$ on $Q^{(3)}$
are right equivalent to $W^{(3,1)}_t$ if $t\neq 1$.
Moreover, $(Q^{(3)},W^{(3,1)}_t)=\widetilde{\mu}_6(Q^{(2)},W^{(2)}_t)$.

\item[(d)]
The potentials $W^{(4,2)}_{t/(t-1)}$ and  $W^{(4,3)}_{(1-t)^{-1}}$ on $Q^{(4)}$
are right equivalent to $W^{(4,1)}_t$ if $t\neq 1$.
Moreover, $(Q^{(4)},W^{(4,1)}_1)=\widetilde{\mu}_1(Q^{(3)},W^{(3,1)}_t)$.

\item[(e)]
$\mu_1(Q^{(4)}, W^{(4,1)}_t)= (Q^{(4)},W^{(4,1)}_{t^{-1}})^{\mathrm{op}}$ and
\begin{multline*}
\mu_2(Q^{(3)},W^{(3,1)}_t)=\\
\left(\vcenter{\xymatrix@-.5pc@L-.1pc{
\ar@{.}[d] & \ar[ld]_{\bet^*_1} 3 &        & \ar[ld]_{\del} 1& \ar@{.}[d]\\
2&\ar[l]_{\bet^*_2} 4 & \ar[lu]_{[\bet_1\gam]}\ar[l]_{[\bet_2\gam]}\ar[ld]^{[\bet_3\gam]} 5 &   &\ar[lu]_{\eps^*}\ar[ll]_{\gam^*} 2\\
\ar@{.}[u] & \ar[lu]^{\bet^*_3} 6 &              &   & \ar@{.}[u]\\
}},\
(\sum_{i=1}^3 \bet_i^*[\bet_i\gam])\gam^* +(t^{-1}\bet_1^*[\bet_1\gam]
+\bet_2^*[\bet_2\gam])\del\eps^*\right).
\end{multline*}

\end{itemize}
\end{lemma}

We omit the lengthy but straightforward calculations.

In order to state our findings in a compact way we need yet another 
generalization of the concept of right equivalence:
We say that $\nu=(\nu_0,\nu_1)\colon Q\ra Q'$ is an isomorphism of
quivers if $\nu_0\colon Q_0\ra Q'_0$ is a bijection and
$\nu_1\colon Q_1\ra Q'_1$ is a bijection such that
$s(\nu_1(\alp))=\nu_0(s(\alp))$ and $t(\nu_1(\alp))=\nu_0(t(\alp))$
for all $\alp\in Q_1$. Obviously, $\nu$ induces in this case an isomorphism
$\CQ{Q}\ra\CQ{Q'}$ which we denote also by $\nu$.
We say that to QPs $(Q,W)$ and $(Q',W')$ are \emph{isomorphic} if
there exists an isomorphism $\nu\colon Q\ra Q'$ such that
$\nu(W)$ and $W'$ are cyclically equivalent.

Since mutation is involutive on the right equivalence classes of
reduced potentials, it follows for example from Lemma~\ref{lem:Sp4-el}(c)
that $\mu_3(Q^{(3)},W^{(3,1)}_t)$ is isomorphic to $(Q^{(2)}, W^{(2)}_{t/t-1})$.

\begin{lemma} \label{lem:Sp4-A2}
Let $\check{\ka}:=\ka\setminus\{0,1\}$ and consider the
involutions $f_1$ and $f_2$  of $\check{\ka}$ which act as $f_1(t) = t^{-1}$ and
$f_2(t) = 1-t$. Then the following hold:
\begin{itemize}

\item[(i)]
The set
$F := \{\Id_{\check{\ka}}, f_1,f_2, f_2\circ f_1, f_1\circ f_2, f_1\circ f_2\circ f_1\}$
is with the composition of functions a Coxeter group of type $\mathsf{A}_2$.
In particular, $f_1\circ f_2\circ f_1= f_2\circ f_1\circ f_2$.

\item[(ii)]
If $t\in\check{\ka}$, then the orbit $F(t) := \{ f(t) \mid f\in F \}$ is the set
$\{t, t^{-1}, 1-t, (t-1)/t, 1/(t-1), t/(1-t)\}$.

\end{itemize}
\end{lemma}

We leave the easy proof to the reader. Putting together these facts we get
the following:

\begin{lemma} \label{lem:Sp4-ndeg}
For $t\in\ka\setminus\{0,1\}$ let
\[
S_4(t):=\{ (Q^{(i)},W^{(i)}_s) \mid i\in\{1,2,3,4\}, s\in F(t)\}
\]
where we abbreviate $W^{(i)}_s:= W^{(i,1)}_s$ for $i=3,4$. Then, if
$j\in\{1,2,\ldots, 6\}$ and $(Q,W)\in S_4(t)$ we have $\mu_j(Q,W)$
isomorphic to an element of $S_4(t)$. In particular, the elements of
$S_4(t)$ are non-degenerate QPs.
\end{lemma}

A slightly less precise result can be obtained by using tubular cluster
categories of type $(2,2,2,2)$, see \cite[2.5.1]{GeGo}.

Note that up to cyclic equivalence we may assume that any cycle  in $Q^{(2)}$
starts (and ends) in the vertex $1$. The length of any cycle is a multiple of 4.

\begin{lemma} \label{lem:Sp4-Q2t}
We have for the above described quiver $Q^{(2)}$ the following:
\begin{itemize}

\item[(a)]
Consider for $t\in\ka^*$ the ideal
$I_t:= \ebrace{\partial_\xi W^{(2)}_t \mid \xi\in Q^{(1)}_1}$ in the path algebra
$\kQ{Q^{(2)}}$. If $t\neq 1$, then any path $p$ of length $4$ with 
$s(p)\neq t(p)$ belongs to $I_t$ and any path of length $5$ belongs to $I_t$.
In particular, $\kQ{Q^{(1)}}/ I_t$ is isomorphic to the Jacobian algebra
$\cP(Q^{(2)}, W^{(2)}_t)$.

\item[(b)]
Let $W= W^{(2)}_t+W'$ be a potential on $Q^{(2)}$, where $W'$ is a possibly
infinite linear combination of cycles of length at least $8$ and $t\neq 1$,
then $W$ is right equivalent to $W^{(2)}_t$.
\end{itemize}
\end{lemma}

\begin{proof}
(a) Because of Lemma~\ref{lem:Sp4-el}(b) it is sufficient to analyze paths
which start in the vertex $1$ or $5$.
Modulo $I_t$ we have for example
\[
\del_2\alp_1\bet_1\gam_1\stackrel{\partial_{\gam_2} W^{(2)}_t}{=}
\del_2\alp_2\bet_2\gam_1\stackrel{\partial_{\del_1} W^{(2)}_t}{=}
t\del_2\alp_1\bet_1\gam_1,
\]
which implies $\del_2\alp_1\bet_1\gam_1\in I_t$ if $t\neq 1$.
This proves (a).
Part (b)
is a direct consequence of Lemma~\ref{lemma:killing-S'},
since $W^{(2)}_t$ is homogeneous of degree $3$ with respect to the path length
grading.
\end{proof}

\begin{lemma} \label{lem:Sp4-Q21deg}
Let $W_t:=W^{(2)}_t+W'$ be a potential on $Q^{(2)}$ where $W'$ is a possibly
infinite linear combination of cycles of length greater than $4$.
Then $W_t$ is degenerate for $t=1$.
\end{lemma}

\begin{proof}
The QP 
$(\tilde{Q},\tilde{W}_t):=\tilde{\mu}_6\tilde{\mu}_5\mu_4\mu_3(Q^{(2)},W^{(2)}_t)$
looks as follows:
\[
\tilde{Q}: \vcenter{\xymatrix@-.5pc@L-.1pc{
\ar@{.}[d]                   &3\ar[rd]^{\bet_1^*} &&5\ar[rd]^{\del_1^*} &\ar@{.}[d]\\
1\ar[ru]^{\alp_1^*}\ar[rd]_{\alp_2^*}&& \ar@<-.5ex>[ll]_{[\alp_1\bet_1]}\ar@<.5ex>[ll]^{[\alp_2\bet_2]} 2\ar[ru]^{\gam_1^*}\ar[rd]_{\gam_2^*}&&\ar@<-.5ex>[ll]_{[\gam_1\del_1]}\ar@<.5ex>[ll]^{[\gam_2\del_2]} 1\\
\ar@{.}[u]            &4\ar[ru]_{\bet_2^*} &&6\ar[ru]_{\del_2^*} &\ar@{.}[u]
}}\
\begin{aligned}
\tilde{W}_t =&
\sum_{i=1}^2 (\bet_i^*\alp_i^*[\alp_i\bet_i]+\del_i^*\gam_i^*[\del_i\gam_i])\\
& +(t-1)[\alp_1\bet_1][\gam_2\del_2]+\sum_{i,j=1}^2 [\alp_i\bet_i][\gam_j\del_j]\\
& + \tilde{W}'
\end{aligned}
\]
Here $\tilde{W}'$ is a (possibly infinite) sum of cycles of length
strictly greater than 2 built from the ``composed'' arrows $[\alp_i\bet_i]$ and
$[\gam_i\del_i]$, since up to cyclic equivalence we may assume that all
cycles in $W'$ start in the vertex $1$.

Next, note that
$\mu_4\mu_3(Q^{(2)}, W^{(2)}_t)=\tilde{\mu}_4\tilde{\mu}_3(Q^{(2)},W^{(2)}_t)$ is
2-acyclic, and the vertices $5$ and $6$ are not joined by an arrow in
$\mu_4\mu_3(Q^{(2)})$. Thus,
\[
\tilde{\mu}_5(\tilde{\mu}_4(\mu_4\mu_3(Q^{(2)}, W^{(2)}_t))_{\textrm{red}})_{\textrm{red}}=\tilde{\mu}_5\tilde{\mu}_4\mu_4\mu_3(Q^{(2)}, W^{(2)}_t)_{\textrm{red}}.
\]
By~\cite[Proposition~4.15]{DWZ1}
the reduced part of $(\tilde{Q},\tilde{W}_1)$ is 2-acyclic if and only
if $\det(\begin{smallmatrix}1&1\\t&1\end{smallmatrix})\neq 0$.
This is the case
if and only $t\neq 1$. This necessary condition for $W_t$ being non-degenerate
does clearly not depend on the choice of $W'$.
\end{proof}

\begin{proof}[Proof of Proposition~\ref{prp:Sp4}]
Note, that the quiver $Q^{(1)}$ has 8 oriented cycles of length 3.
These cycles are up to cyclic equivalence of the  form
$C_{kji}:=\gam_{ik}\bet_{kj}\alp_{ji}$ for $i,j,k\in\{1,2\}$.
Each $C_{kji}$ has to appear in a non-degenerate potential with a
coefficient $t_{kji}\in\ka^*$.

By Lemma~\ref{lem:Sp4-el}(a) it is sufficient to prove the claim of
Proposition~\ref{prp:Sp4} by proving the analogous claim
for~$Q^{(2)}$ in place of $Q^{(1)}$. Let $W= W_4+ W_{\geq 8}$ a non-degenerate
potential on $Q^{(2)}$, where $W_4$ comprises all cycles of length $4$ which
appear in $W$ with a non-zero coefficient. Now, we have
\[
\tilde{\mu}_1(Q^{(2)}, W) = (Q^{(1)}, [W_4]+[W_{\geq 8}]+C_{111}+C_{121}+C_{211}+C_{221}).
\]
Note that $[W_4]=\sum_{i,j=1}^{2} t_{ij} C_{ij2}$ and $[W_{\geq 8}]$ consists of 
cycles of length greater or equal to 6. By the above remark we conclude that
$W_4$ is of the form $W_{\mathrm{gen}}^{(2)}$, and after rescaling we may assume 
that $W_4= W^{(2)}_t$ with $t\in\ka^*$.
Now part (i) of Proposition~\ref{prp:Sp4} follows directly from
Lemmas~\ref{lem:Sp4-Q2t}, \ref{lem:Sp4-Q21deg} and \ref{lem:Sp4-ndeg}.
Part (ii) can be checked easily by a direct calculation.
\end{proof}

\subsection{Reminder on Galois coverings}\label{galoisreminder}
Let $\LL$ be a finite dimensional basic $\ka$-algebra. We say that a
group grading $\LL=\oplus_{g\in G} \LL_g$ for a group $G$ is \emph{compatible}
if
\begin{itemize}
\item
$\LL_g\cdot\LL_h\subset\LL_{gh}$ for all $g,h\in G$,
\item
$\rad\LL=\oplus_{g\in G}(\LL_g\cap\rad\LL)$, where $\rad\LL$ is the
Jacobson radical of $\LL$,
\item
$G$ is generated by $\{g\in G\mid \LL_g\neq 0\}$.
\end{itemize}
If $1_\LL=e_1+\cdots+e_n$ is a decomposition of the unit into primitive
orthogonal idempotents, we can construct in this situation a \emph{Galois
covering} $\tilde{\LL}$ which is a locally bounded $\ka$-category with
objects $G\times\{1,\ldots,n\}$, morphism spaces
$\tilde{\LL}((g,i),(h,j))=(e_j\LL_{hg^{-1}}e_i,g)$ for all
$\{(g,i),(h,j)\}\in G\times\{1,\ldots,n\}$ and the obvious composition
coming from the multiplication in $\LL$. Obviously, $G$ acts freely
on $\tilde{\LL}$ and the orbit category can be identified with the
subcategory $\LL'$ of the $\LL$-right modules which has the following
objects: $\{e_1\LL,\ldots,e_n\LL\}$.
It is a standard task to realize that  for a locally
bounded category $\tilde{\LL}$ the category of finitely presented
$\ka$-linear functors $\tilde{\LL}\ra\lmd{\ka}$ behaves pretty much
like the module category of a finite dimensional algebra and that in this
context an appropriate version of the tame-wild theorem holds,
see~\cite{DowSk86} for more details. In particular, a locally bounded
category is wild if and only if it contains a \emph{finite} subcategory
which is wild.  An important idea behind this construction is that
the category of $\tilde{\LL}$-modules is equivalent to the category
of $G$-graded $\LL$-modules, see for example~\cite{Green81}.

If $\LL$ has a Galois covering which is wild, then $\LL$ is also wild.
This follows from \cite[Proposition~2]{DS}.
Thus, if we can identify in a Galois covering $\tilde{\LL}$ of $\LL$ a
convex subcategory which belongs to Unger's list of concealed
minimal wild algebras~\cite{Ung90}, we can conclude that $\LL$ is
wild.

\section*{Acknowledgements}
We thank Henning Krause, Sefi Ladkani, Claus Ringel and Dieter Vossieck for 
helpful discussions.
The authors thank the Sonderforschungsbereich/Transregio SFB 45 in Bonn 
and CONACYT Grant 91948 (Mexico) for financial support.

\end{document}